%% file: frame_flow_revised-arxiv.tex
\documentclass[reqno,12pt]{amsart}
\usepackage{amsmath,amssymb,amsthm,graphicx,mathrsfs,url}
\usepackage{wrapfig}
\usepackage{enumitem}
\usepackage{mathtools}
\usepackage[usenames,dvipsnames]{xcolor}
\usepackage[colorlinks=true,linkcolor=Red,citecolor=Green]{hyperref}
\usepackage[super]{nth}
\usepackage[open, openlevel=2, depth=3, atend]{bookmark}
\hypersetup{pdfstartview=XYZ}
\usepackage[font=footnotesize]{caption}
\usepackage{a4wide}

\usepackage{epstopdf}
 
\usepackage{hyperref}

\theoremstyle{plain}
\newtheorem{theo}{Theorem}
\newtheorem{lemma}{Lemma}[section]
\newtheorem{prop}[lemma]{Proposition}
\newtheorem{cor}[lemma]{Corollary}

\newtheorem{ass}{Assumption}

\theoremstyle{definition}
\newtheorem{definition}[lemma]{Definition}
\newtheorem{exam}[lemma]{Example}

\theoremstyle{remark}
\newtheorem{rem}{Remark}[section]
 
\numberwithin{equation}{section}

\newcommand{\C}{\mathbb{C}}
\newcommand{\R}{\mathbb{R}}

\newcommand{\Z}{\mathbb{Z}}
\newcommand{\D}{\mathcal{D}}
\newcommand{\N}{\mathbb{N}}

\newcommand{\eps}{\varepsilon}

\newcommand{\mc}{\mathcal}
\newcommand{\rr}{\mathbb{R}}
\newcommand{\nn}{\mathbb{N}}
\newcommand{\cc}{\mathbb{C}}
\newcommand{\hh}{\mathbb{H}}
\newcommand{\zz}{\mathbb{Z}}

\newcommand{\norm}[1]{\left\Vert#1\right\Vert}
\newcommand{\la}{\lambda}

\newcommand{\pl}{\partial}
\newcommand{\x}{\times}

\newcommand{\til}{\widetilde}
\newcommand{\bbar}{\overline}
\newcommand{\cjd}{\rangle}
\newcommand{\cjg}{\langle}

\newcommand{\demi}{\tfrac{1}{2}}

\DeclareMathOperator{\Res}{Res}

\let\Im=\Imag

\DeclareMathOperator{\Op}{Op}

\let\Re=\Real

\DeclareMathOperator{\supp}{supp}

\DeclareMathOperator{\WF}{WF}

\usepackage{tikz}
\usetikzlibrary{positioning}
  
  \def\Ddots{\mathinner{\mkern1mu\raise\p@
\vbox{\kern7\p@\hbox{.}}\mkern2mu
\raise4\p@\hbox{.}\mkern2mu\raise7\p@\hbox{.}\mkern1mu}}
\makeatother

\newcommand{\eklm}[1]{\left\langle #1 \right\rangle}

\renewcommand{\d}{\,d}

\newcommand{\A}{{\mathcal A}}

\newcommand{\E}{{\mathcal E}}

\renewcommand{\H}{{\mathcal H}}
\newcommand{\I}{{\mathcal I}}

\newcommand{\M}{{\mathcal M}}
\newcommand{\T}{{\rm T}}

\newcommand{\Xbf}{{\mathbf X}}

\renewcommand{\epsilon}{\vararepsilon}

\newcommand{\bdm}{\begin{displaymath}}
\newcommand{\edm}{\end{displaymath}}
\newcommand{\bq}{\begin{equation}}
\newcommand{\eq}{\end{equation}}
\newcommand{\bqn}{\begin{equation*}}
\newcommand{\eqn}{\end{equation*}}

\newcommand{\Cinft}{{C^{\infty}}}
\newcommand{\CT}{{C^{\infty}_c}}

\renewcommand{\L}{{\rm L}}

\renewcommand{\S}{{\mathcal S}}

\newcommand{\SO}{\mathrm{SO}}
\newcommand{\SU}{\mathrm{SU}}

\newcommand{\g}{{\bf \mathfrak g}}

\newcommand{\aL}{{\bf \mathfrak a}}
\newcommand{\nL}{{\bf \mathfrak n}}
\renewcommand{\k}{{\bf \mathfrak k}}
\newcommand{\m}{{\bf \mathfrak m}}

\newcommand{\p}{{\bf \mathfrak p}}
\newcommand{\U}{{\mathfrak U}}
\newcommand{\so}{{\bf \mathfrak so}}

\newcommand{\Ad}{\mathrm{Ad}}

\renewcommand{\Im}{\mathrm{Im}\,}
\renewcommand{\Re}{\mathrm{Re}\,}

\newcommand{\gam}{\Gamma\backslash}

\definecolor{darkblue}{rgb}{0,0,0.4}

\newcommand{\tu}[1]{\textup{#1}}

\colorlet{c1}{blue!90!black}
\definecolor{mint}{rgb}{0.24, 0.71, 0.54}
\colorlet{c2}{orange!70!black}
\colorlet{c3}{red!100!black}
\colorlet{c4}{yellow!40!orange}
\input{commands}

\author[C. Guillarmou]{Colin Guillarmou}
\email[Corresponding author]{colin.guillarmou@math.u-psud.fr}
\address{Universit\'e Paris-Saclay, CNRS,  Laboratoire de math\'ematiques d'Orsay, 91405, Orsay, France.}

\author[B. K\"uster]{Benjamin K\"uster}
\email{bkuester@math.upb.de}
\address{Institut f\"ur Mathematik, Universit\"at Paderborn, 33098 Paderborn, Germany.\bigskip}
\makeatletter
\let\@wraptoccontribs\wraptoccontribs
\makeatother

\contrib[with an appendix by]{Charles Hadfield}
\email{charles.hadfield@ibm.com}
\address{IBM T.J. Watson Research Center, Yorktown Heights, NY 10598, United States.}

\title{Spectral theory of the frame flow on hyperbolic 3-manifolds}

\begin{document}

\begin{abstract}
We study the spectral theory and the resolvent of the vector field generating the frame flow of closed hyperbolic $3$-dimensional manifolds on some family of anisotropic Sobolev spaces. We show the existence of a spectral gap and prove resolvent estimates using semiclassical methods. 
\end{abstract}

\maketitle

\section{Introduction}

In the last twenty years, there has been developed a new spectral approach to studying hyperbolic dynamics
via transfer operators acting on appropriate anisotropic Sobolev spaces on which the transfer operators (for diffeomorphisms) or their generators (for flows) have discrete spectrum, see \cite{GL06,BT07,BL07,FRS08,FS11,DZ16,DG16,BW17}. 
In particular, the use of microlocal and harmonic analysis in the spirit of quantum scattering theory proved to be efficient for describing long time dynamics. For example, exponential mixing for hyperbolic flows is equivalent to the existence of a spectral gap together with polynomial bounds on the resolvent of the generator. Such gaps have been obtained for contact Anosov flows \cite{Do98, Li04,Ts12,NZ15,FT17a}, Axiom A flows \cite{Na05,St11}, geometrically finite hyperbolic manifolds \cite{Edwards-Oh}, or Sinai Billiards \cite{BDL18}. In dimension $3$, exponential mixing is known for all topologically mixing Anosov flows by a recent result of Tsujii-Zhang \cite{tsujiizhang20} (see also Tsujii \cite{Tsujii18} for a generic result).
For partially hyperbolic flows, much less is known and a natural geometric example is given by the frame flow, defined as follows: let $({\bf M},g)$ be an $n$-dimensional oriented Riemannian manifold and let $F{\bf M}$ be the principal bundle over ${\bf M}$ formed by the oriented orthonormal frames $e=(e_1,\dots,e_n)$ of the tangent spaces. Then we define the frame flow to be 
\[ \til{\varphi}_t: F{\bf M}\to F{\bf M}, \quad \til{\varphi}_t(x,e)=(\pi(\varphi_t(x,e_1)),e(t)),\]
where $x\in {\bf M}$, $\varphi_t:S{\bf M}\to S{\bf M}$ is the geodesic flow on the unit tangent bundle, $\pi:S{\bf M}\to {\bf M}$ the projection onto the base, and $e(t)$ the frame obtained by parallel transport along the geodesic $\gamma(s):=\pi(\varphi_s(x,e_1))\in {\bf M}$ for $s \in [0,t]$. This is an extension of the geodesic flow $\varphi_t$ since, if $\til{\pi}:F{\bf M}\to S{\bf M}$ is the projection defined by $\til{\pi}(x,e):= (x,e_1)$, one has 
\[\til{\pi}(\til{\varphi}_t(x,e))=\varphi_t(x,e_1)=\varphi_t(\til{\pi}(x,e)).\] 
If $({\bf M},g)$ has negative curvature, then it is a classical result \cite{Brin-Pesin} that the flow $\til{\varphi}_t$ is partially hyperbolic: one has a flow-invariant decomposition 
\begin{equation}\label{stable/unstable/frame} 
T(F{\bf M})=\til{E}_0\oplus \til{E}_s\oplus \til{E}_u,
\end{equation}
where $d\til{\varphi}_t$ is contracting on $\til{E}_s$ (resp.\ $\til{E}_u$) in positive (resp.\ negative) time and $\til{E}_0=\R\til{X}\oplus V$ with $V=\ker d\til{\pi}$, where $\til{X}$ is the vector field  generating the flow $\til{\varphi}_t$.

If $\dim {\bf M}>2$, the dynamical behavior of the frame flow is qualitatively different from that of the geodesic flow because the latter is an Anosov flow but the former is not: besides the flow direction $\til{X}$, the frame flow possesses additional neutral directions described by the non-zero subbundle $V$.

In this paper, we focus on the case $\dim {\bf M}=3$. Then $\til{\pi}: F{\bf M}\to S{\bf M}$ is a principal $\SO(2)$-bundle and there is precisely one additional neutral direction besides  $\til{X}$.

For an Anosov flow $\varphi_t$ generated by a smooth vector field $X$ on a compact manifold, it is known \cite{BL07,FS11,DZ16} that for each $N\geq 0$ there are anisotropic Sobolev spaces $\mc{H}_N$ such that 
the linear operator $-X$ has discrete spectrum in the half-plane $\{\la\in \C\;|\; {\rm Re}(\la)>-N\}$. Moreover, for $N\not=N'$ the spectrum and the (generalized) eigenfunctions of $-X$ 
in $\mc{H}_N$ and in $\mc{H}_{N'}$ coincide in the region  ${\rm Re}(\la)>-\min(N,N')$. This intrinsic spectrum, whose elements are called Pollicott-Ruelle resonances, is exactly the set of the poles of the meromorphic continuation of the resolvent $R_{X}(\la):=(-X-\la)^{-1}$, originally defined in ${\rm Re}(\la)>0$ by the convergent expression
\[R_{X}(\la): C^\infty(S{\bf M})\to L^\infty(S{\bf M}), 
\quad R_X(\la)f=-\int_{0}^\infty e^{-\lambda t}\varphi^*_{-t}f \, dt,\] 
to the whole complex plane $\C$. Here the extended operator is viewed as a 
continuous map $R_X(\la):C^\infty(S{\bf M})\to \D'(S{\bf M})$, where $\mc{D}'$ denotes the space of distributions. The works \cite{Do98, Li04,Ts12,NZ15,FT17a} mentioned above show that if the flow is a \emph{contact} Anosov flow, then there is a half-plane $\{{\rm Re}(\la)>-\eps\}$ containing no elements in the spectrum except $\la=0$ and the flow exhibits exponential decay of correlations on smooth observables. For the geodesic flow on a compact hyperbolic manifold ${\bf M}$ (i.e., with constant curvature $-1$), one can fully describe the Pollicott-Ruelle resonance spectrum: it is given in terms of eigenvalues of Laplacians on symmetric tensors on ${\bf M}$ and there are only finitely many Pollicott-Ruelle resonances in ${\rm Re}(\la)>-n/2$, where $n=\dim {\bf M}-1$, see \cite{dfg}.\\

The goal of our work is to describe, in the same spirit, the spectral theory of the generator $\til{X}$ of\footnote{We use the notation $\til{X}$ for the frame flow generator exclusively in the introduction.}   the frame flow $\til{\varphi}_t:F{\bf M}\to F{\bf M}$ in the case where ${\bf M}=\Gamma\backslash \H^3$ is an oriented closed hyperbolic manifold (here $\Gamma\subset \textrm{PSO}(1,3)$ is a co-compact subgroup). In that case, the frame bundle 
$F{\bf M}$ can be written as $F{\bf M}=\Gamma\backslash G$ where $G:=\textrm{PSO}(1,3)$, it then inherits a natural measure $\mu_G$ induced by the Haar measure on the Lie group $G$, and $\til{X}$ preserves the measure in the sense that $\mc{L}_{\til{X}}\mu_G=0$.
For $\lambda\in \C$ with $\Re \lambda>0$, the operator $-\til{X}$ on $F{\bf M}$ has a well-defined resolvent $R_{\til{X}}(\lambda):=(-\til{X}-\lambda)^{-1}$ defined by
\[
R_{\til{X}}(\lambda)f:=-\int_{0}^\infty e^{-\lambda t}\til{\varphi}^*_{-t}f \, dt,\qquad f\in C^\infty(F{\bf M})
\]
and this operator extends to a bounded operator on $L^2(F{\bf M},\mu_G)$.

Our main result, stated in a slightly more detailed form as Theorem \ref{prop:gap} in Section \ref{sec:resolvestim}, is:
\begin{theo}\label{th1intro} Let ${\bf M}$ be a compact oriented hyperbolic $3$-manifold. Then, there are Hilbert spaces $\mathcal H^{1,1}$, $\mathcal H^{1,0}$ 
with continuous inclusions $\Cinft(F{\bf M})\subset\mathcal H^{1,1}\subset\mathcal H^{1,0}\subset \D'(F{\bf M})$  such that the frame flow resolvent $R_{\til{X}}(\lambda)$ extends to the region
$\{\Re\lambda >-1\}\subset \C$
as a meromorphic family of bounded  operators $R_{\til{X}}(\lambda): \mathcal H^{1,1}\to \mathcal H^{1,0}$, and the only poles of $R(\lambda)$ in that region are given by the real numbers $\lambda_j:=  \sqrt{1-\nu_j}-1$, $0\leq j\leq J$, where $\nu_0=0,\nu_1,\ldots,\nu_J$ are the eigenvalues of the Laplace-Beltrami operator $\Delta$ on ${\bf M}$  in the interval $[0,1)$. Moreover, for every $\delta,r>0$ there is a constant $C_{\delta,r}>0$ such that for  $1<|{\rm Im}(\lambda)|$ and $-1+\delta<\Re \lambda <r$, one has the following resolvent estimate:
\bqn
\norm{R_{\til{X}}(\lambda)}_{\mathcal H^{1,1}\to \mathcal H^{1,0}}\leq C_{\delta,r}\eklm{\lambda}^{3}.
\eqn
\end{theo}
Here we use the notation $\eklm{\lambda}:=\sqrt{1+|\lambda|^2}$. This result shows that the frame flow has a spectral gap and that this gap is of size $1$ away from the real axis, like for the geodesic flow. The spaces $\mc{H}^{1,1}$ and $\mc{H}^{1,0}$ are anisotropic spaces that are related by $\mc{H}^{1,0}=\mc{H}^{1,1}+R\mc{H}^{1,1}$ where $R$ is a non-vanishing vector field tangent to the vertical space $V$ (i.e., the $\SO(2)$ fibers) of the fibration $F{\bf M}\to S{\bf M}$; the norms on $\mc{H}^{1,1}$ and $\mc{H}^{1,0}$ are related by a Fourier decomposition in the $\SO(2)$ fibers, see Section \ref{sec:hilbertconstr}. 
They correspond to distributions in some negative Sobolev space $H^{-k}(F{\bf M})$ for some $k$  but with extra regularity in the unstable directions.
We show in Section \ref{sec:mix} that Theorem \ref{th1intro} implies that for 
$f\in \mc{H}^{1,1}$ such that $\til{X}^kf\in \mc{H}^{1,1}$ for all $k\leq 5$, 
and for $f'\in (\mc{H}^{1,0})'$ (where $(\mc{H}^{1,0})'$ is the dual to $\mc{H}^{1,0}$),   we get for all $\beta\in (0,1)$ and $t\geq 0$ 
\[ \cjg e^{-t\til{X}}f,f'\cjd
=\sum_{j=1}^J e^{t\la_j}\cjg\Pi_jf_0,f_0'\cjd+\mc{O}(e^{-t\beta})\|(-\til{X}+1)^5f\|_{\mc{H}^{1,1}}\|f'\|_{\mc{H}^{-1,0}},\]
where the constant in $\mc{O}$ depends only on $\beta$. Here the $\Pi_j$ are the spectral projectors of $-\til{X}$ at $\la_j$ on the space $\mc{H}^{1,1}$; as mentioned in Theorem \ref{th1intro}, the elements in ${\rm Ran}(\Pi_j)$, called Pollicott-Ruelle resonant states, are in turn related to the eigenfunctions of the Laplacian $\Delta$ at $\nu_j$ by pushforward along the projection onto the base $F{\bf M}\to {\bf M}$, see \cite{dfg}.

Notice that this implies an exponential decay of correlations for the geodesic flow by taking $f_0,f_0'\in \ker R$.\\

The mixing of the frame flow for compact hyperbolic manifolds follows from Howe-Moore \cite{Howe-Moore79}, and the exponential mixing is a consequence of the work of Moore \cite{Moore87} on the decay of matrix coefficients for rank-one symmetric spaces. Both results mainly use tools of representation theory. In contrast, we use here purely analytic and semiclassical methods, with hope that they could be extended to variable curvature settings. Using the Fourier decomposition in the $\SO(2)$ fibers of $F{\bf M}$, we approach the problem by introducing a semiclassical family of operators ${\bf X}_n$ on powers $\mc{L}^n$ of a complex line bundle $\mc{L}$ over $S{\bf M}$, in a way similar to  geometric quantization. This approach was suggested to us by F.\ Faure and has been successfully applied before in the works of Faure \cite{faure2011} and Faure-Tsujii \cite{Faure-TsujiiAST}. Similar techniques have been used by Arnoldi \cite{arnoldi12} for the non-abelian group $\SU(2)$ instead of $\SO(2)$ and by Naud \cite{Naud2019} for circle extensions of Anosov diffeomorphisms. 
We prove a polynomial upper bound on the norm of the resolvent $(-{\bf X}_n-\la)^{-1}$ on $n$-depending anisotropic Sobolev spaces for the family of operators ${\bf X}_n$ in the region ${\rm Re}(\la)>-1+\delta$, 
and these bounds are uniform with respect to $(n,|\la|)$.
The proof is done using semiclassical measures, inspired by the paper of Dyatlov \cite{Dyatlov16} in the setting of operators with normally hyperbolic trapping. The new difficulty, as compared to the case of a contact Anosov flow, is the uniform treatment with respect to the Fourier parameter $n$. This can be regarded as a kind of doubly semi-classical problem and
the key mechanism for the proof is the symplectic structure of the trapped set associated with $-{\bf X}_n-\la$ in $T^*(S{\bf M})$. In that respect our proof shares similarities with the works \cite{NZ15,Faure-TsujiiAST,FT17a}.

Considering simultaneous asymptotics of an energy parameter and a sequence of representations of a general compact Lie group $G$ has a long tradition in the context of trace formulas and quantum ergodicity for gauge fields; see, for example, the works of  Cahn-Taylor \cite{cahntaylor}, Hogreve-Potthoff-Schrader \cite{hogreve83}, Schrader-Taylor \cite{schradertaylor84,schradertaylor89}, Guillemin-Uribe \cite{guilleminuribe86,guilleminuribe89,guilleminuribe90}, and  Zelditch \cite{zelditch19921}. 
 
Although we do not fully prove it, our method 
should also give that, for each $N>0$, $R_{\til{X}}(\la)$ is analytic in the region 
\[ \{ \la\in \C \, |\, {\rm Im}(\la)\not=0, {\rm Re}(\la)\notin -\nn^*, {\rm Re}(\la)>-N \}\]
as a bounded operator from $\mc{H}^{1,1}_N$ to $\mc{H}^{1,0}_N$ for some  anisotropic Sobolev spaces 
$\mc{H}^{1,1}_N, \mc{H}^{1,0}_N$ similar to $\mathcal H^{1,1}$, $\mathcal H^{1,0}$ but with different scales of regularity depending only on $N$. However, we believe it is not meromorphic 
on the vertical lines ${\rm Re}(\la)\in -\nn^*$.

To study the spectrum of ${\bf X}_n$ on each of the individual line bundles $\mc{L}^n$, one can embed these line bundles into bundles of symmetric, trace-free tensors and apply the quantum-classical correspondence results of \cite{dfg}. This has been done by C.\ Hadfield whose  calculations are included in Appendix \ref{sec:appendix}. The computation of the spectra of ${\bf X}_n$ in Corollary \ref{cor:individualgap} strongly suggests that $R_{\til{X}}(\lambda)$ cannot be meromorphically extended to ${\rm Re}(\la)\in -\nn^*$.\\

We conclude this introduction by a discussion of the known properties of the frame flow in variable curvature. The ergodicity of the frame flow is known for a set of metrics with negative curvature that is open and dense in the $C^3$-topology by Brin \cite{Brin}, in all odd dimensions except $n=7$ by Brin-Gromov \cite{Brin-Gromov80}, in all even dimensions except $n=8$ if the curvatures are pinched enough by Brin-Karcher \cite{Brin-Karcher84}, and finally in all dimensions if the curvatures are pinched enough (very close to $1$) by Burns-Pollicott \cite{Burns-Pollicott03}. For geometrically finite hyperbolic manifolds, the mixing of the frame flow is proved by Flaminio-Spatzier \cite{Flaminio-Spatzier90} (see also Winter \cite{Winter}) and the exponential mixing is proved in the works of Mohammadi-Oh \cite{MoOh2015} (in some cases) and Sarkar-Winter \cite{SarkarWinter} (for convex co-compact cases).

\subsection*{Acknowledgements}
This project has received funding from the European Research Council (ERC) under the European Union’s Horizon 2020 research and innovation programme (grant agreement No. 725967). BK has received further funding from the Deutsche Forschungsgemeinschaft (German Research Foundation, DFG) through the Priority Programme (SPP) 2026 ``Geometry at Infinity''. Some part of this work is based on some discussion and some suggestion given to us by Fr\'ed\'eric Faure, whom we warmly thank for this, in particular for what concerns the decomposition into irreducibles and the symplectic structure of the trapped set. We thank Laurent Charles for sharing with us some results of his thesis used here and Semyon Dyatlov for pointing out to us that the method of \cite{Dyatlov16} could be used for proving spectral gaps of Anosov flows.  Last but not least, we warmly thank the two anonymous referees for their extremely careful reading of the manuscript and numerous suggestions of improvements and corrections.

\section{Setup and notation}
Throughout the paper we use the notation $\N^\ast=\{1,2,3,\ldots,\}$, $\N_0=\{0,1,2,\ldots\}$. 
\subsection{Algebraic description of the geometry of the hyperbolic space $\H^3$}\label{eq:algdescr} 
We start with the algebraic description of the hyperbolic space $\H^3$, its unit tangent bundle and its frame bundle. We will largely avoid abstract Lie theoretic terms.

Let $G:=\textrm{PSO}(1,3)$ with Lie algebra $\g=\so(1,3)$, considered here as a matrix algebra. With respect to the standard basis for $\R^{1,3}$ we obtain, as in\footnote{Our element $R$ is called $R_{2,3}$ in \cite{dfg}.} \cite{dfg}, a basis of $\g$ consisting of the elements
\begin{align*}
X&=\begin{pmatrix}
0 & 1 & 0 & 0\\
1 & 0 & 0 & 0\\
0 & 0 & 0 & 0\\
0 & 0 & 0 & 0\\
\end{pmatrix},& R&=\begin{pmatrix}
0 & 0 & 0 & 0\\
0 & 0 & 0 & 0\\
0 & 0 & 0 & 1\\
0 & 0 & -1 & 0\\
\end{pmatrix},& U^+_1&=\begin{pmatrix}
0 & 0 & -1 & 0\\
0 & 0 & -1 & 0\\
-1 & 1 & 0 & 0\\
0 & 0 & 0 & 0\\
\end{pmatrix},\\
U^+_2&=\begin{pmatrix}
0 & 0 & 0 & -1\\
0 & 0 & 0 & -1\\
0 & 0 & 0 & 0\\
-1 & 1 & 0 & 0\\
\end{pmatrix},& U^-_1&=\begin{pmatrix}
0 & 0 & -1 & 0\\
0 & 0 & 1 & 0\\
-1 & -1 & 0 & 0\\
0 & 0 & 0 & 0\\
\end{pmatrix},& U^-_2&=\begin{pmatrix}
0 & 0 & 0 & -1\\
0 & 0 & 0 & 1\\
0 & 0 & 0 & 0\\
-1 & -1 & 0 & 0\\
\end{pmatrix}.
\end{align*} 
 The  commutation relations between these elements are
\bq
\begin{alignedat}{6}\label{commutation_relations_XUR}
[X,U_j^\pm]&=\pm U_j^\pm, &\qquad [U_j^+,U_j^-]&=2X, &\qquad [U_1^\pm,U_2^\mp]&=2R, \\
[U_1^\pm,U_2^\pm]&=[X,R]=0, &[R,U_1^\pm]&=-U_2^\pm, & [R,U_2^\pm]&=U_1^\pm.
\end{alignedat}
\eq
\begin{rem}\label{rem:pmnotation}Note that the signs appearing in the commutation relations between $R$ and the elements $U_j^\pm$ are related to the lower index $j$ as opposed to the upper index $\pm$ which is the sign appearing in the commutation relations $[X,U_j^\pm]$. Unfortunately, the lack of a second ``$\pm$''-symbol makes it less elegant to deal with the commutation relations $[R,U_j^\pm]$. 
\end{rem}
 The Lie algebra splits as a direct sum $\g=\k\oplus \p$, where the subspaces $\k$, $\p$ are given by
 \[
 \k=\mathrm{span}_\R(R,K_1,K_2),\qquad \p=\mathrm{span}_\R(X,P_1,P_2)
 \]
with
 \[
K_{j}:=\frac{1}{2}(U_j^+- U^-_j),\qquad P_{j}:=\frac{1}{2}(U_j^++ U^-_j),\qquad j=1,2.
\] 
We introduce on $\g$ an inner product $\eklm{\cdot,\cdot}$ by declaring that $\{R,K_1,K_2,X,P_1,P_2\}$ form an orthonormal basis of $\g$ with respect to $\eklm{\cdot,\cdot}$.
We define the subgroup $K\subset G$ by 
$$K:=\exp(\k)\cong \SO(3),$$ where $\exp$ is the matrix exponential. We can identify the $3$-dimensional hyperbolic space $\H^3$, as a Riemannian manifold, with 
\[
G/K=\textrm{PSO}(1,3)/\SO(3)\simeq \H^3.
\]
Indeed, the tangent bundle of $\H^3$ is an associated\footnote{For a principal $G$-bundle $\pi:P\to {\bf M}$ and a representation $\varrho:G\to\End(V)$, the associated vector bundle $P\times_\varrho V$ is defined as
$
P\times_\varrho V:=(P\times V) /\sim$, where $(p,v)\sim (p\cdot g, \varrho(g^{-1}) v)$. Writing $[p,v]$ for an element in $P\times_\varrho V$, the vector bundle projection $P\times_\varrho V\to {\bf M}$ is given by $[p,v]\mapsto \pi(p)$.} vector bundle
\bq
T\H^3=T(G/K)\cong G\times_{\Ad(K)}\p,\label{eq:associatedp}
\eq
and the chosen $\Ad(K)$-invariant inner product defines a Riemannian metric on $T\H^3$ which is precisely the metric with constant sectional curvatures $-1$.
The manifold $G$ can thus be naturally identified with the frame bundle $F\hh^3$. The space $\k$ is the vertical space of the fibration $G\to G/K$ and $\p$ represents a horizontal space.
Any Lie algebra element $Y\in\g_\C$ acts on $\Cinft(G)$ by the left invariant vector field associated to $Y$, and as such, $X$ generates a flow on $G\simeq F\H^3$ which is the frame flow on hyperbolic space (see \cite[Section 1]{Pollicott92}). As we shall explain in Section \ref{sec:covariant}, viewing the flow of $X$ as a lift of the frame flow of $F{\bf M}$
(recall ${\bf M}=\Gamma\backslash\H^3$ being a hyperbolic co-compact quotient) one can use \eqref{commutation_relations_XUR} to obtain that  ${\rm span}(U_1^-, U_2^-)$ corresponds to $\tilde{E}_s$, ${\rm span}(U_1^+, U_2^+)$ to $\tilde{E}_u$ and ${\rm span}(R,X)$ to $\tilde{E}_0$ in the decomposition \eqref{stable/unstable/frame}.

 The commutation relations \eqref{commutation_relations_XUR} imply
\begin{align}\label{commutation_relations_XKP}
[X,K_i]=P_i,\qquad[X,P_i]=K_i,\qquad [K_i,P_j]=\delta_{ij}X.
\end{align}
  
 The Laplacian on $\g$
 \begin{equation}\label{laplacien}
 \Delta=-X^2-R^2-\frac{1}{2}\big((U_1^-)^2+(U_2^-)^2+(U_1^+)^2+(U_2^+)^2\big)
 \end{equation}  
 associated to $\cjg \cdot,\cdot\cjd$ satisfies $[R,\Delta]=0$.
 The inner product is positive definite and has the convenient property that it is invariant under the adjoint action $\Ad(K)$ of $K$ on $ \g$.  However, it is not $\Ad(G)$-invariant, as opposed to the Killing form on $\g$ (which is not positive definite).  Writing
 \[
\aL:=\mathrm{span}_\R(X),\qquad \m:=\mathrm{span}_\R(R),\qquad  \nL^\pm:=\mathrm{span}_\R(U_1^\pm,U_2^\pm),
 \]
 we have an orthogonal decomposition
 \bq
\g= \aL\oplus \m\oplus \nL^+\oplus \nL^-.\label{eq:Bruhat}
 \eq
An important role will also be played by the group
\[
M:=\exp(\m)\cong \SO(2),
\]
which is a subgroup of $K$ (as $\m$ is a subalgebra of $\k$). In the following, we will identify $\SO(3)=K$ and $\SO(2)=M$. 
Furthermore, in the complexified Lie algebra $\g_\C:=\g\otimes_\R \C$ the elements 
\bq\label{eq:etamu}
\eta_\pm:=\frac{1}{2}(U^-_1\pm iU^-_2)\in \nL^-_\C,\qquad \mu_\pm:=\frac{1}{2}(U^+_1\pm iU^+_2)\in \nL^+_\C,\qquad Q_\pm:=-(X\pm iR)
\eq
will play an important role due to the commutation relations
\begin{align}
\nonumber [X,\eta_\pm]&=-\eta_\pm,& [R,\eta_\pm]&=\pm i\eta_\pm,& [Q_\pm,\mu_\mp]&=-2\mu_\mp, \\
[X,\mu_\pm]&=\mu_\pm, & [R,\mu_\pm]&=\pm i\mu_\pm, &[\eta_\pm,\mu_\mp]&=Q_\pm,\label{eq:etamucomm} \\
\nonumber &&&& [\eta_\pm,\mu_\pm]&=[Q_\pm,\mu_\pm]=0.
\end{align}
From the first two columns in \eqref{eq:etamucomm} we read off that $\eta_\pm$ and $\mu_\pm$ act as raising and lowering operators on the joint spectrum of $X$ and $R$: if $u$ is a joint eigenfunction with eigenvalues $\lambda_X, \lambda_R \in \C$ with respect to $X$ and $R$, respectively, then $\eta_\pm u$ is a joint eigenfunction with eigenvalues $\lambda_X-1, \lambda_R\pm i$ and $\mu_\pm u$ is a joint eigenfunction with eigenvalues $\lambda_X+1, \lambda_R\pm i$.
\begin{rem}[Continuation of Remark \ref{rem:pmnotation}]\label{eq:signs}The notational ``sign management'' is now  different from the situation in Remark \ref{rem:pmnotation}: the index $\pm$ in $\eta_\pm$ and $\mu_\pm$ appears in the commutation relations with $R$, while the signs of the commutators of those elements with $X$ are no longer expressed in an index but in the two different symbols $\eta$ and $\mu$ themselves.
\end{rem} 

Using \eqref{eq:associatedp}, the unit tangent bundle $S\H^3\subset T\H^3$ can be identified with a quotient space:
\bq
G/M=\textrm{PSO}(1,3)/\SO(2)=S\H^3.\label{eq:spherebundlequot}
\eq
The identification is made using the diffeomorphism $G/M\owns gM\mapsto [g,X]\in S_{gK}(G/K)\subset G\times_{\Ad(K)}\p$. 
As mentioned above, the operator $X:C^\infty(G)\to C^\infty(G)$ is the generator of the frame flow on $S\H^3$, and as it commutes with the vector field $R$ tangent to the fibers of $G\to G/M$, it also induces a vector field  $X:C^\infty(G/M)\to C^\infty(G/M)$ that we can identify with the generator of the geodesic flow on the sphere bundle $S\H^3=G/M$ (see \cite[Section 3.3]{dfg} and \cite[Section 1]{Pollicott92}). To summarize the discussion, we have the correspondence of fibrations
\[G \to G/M\to G/K \,\textrm{ with }\, F\H^3\to S\H^3\to \H^3.\]
The Lie algebra element $Y\in\g_\C$ acts on $\Cinft(G)$ by the left invariant vector field associated to $Y$ and can thus be seen as a differential operator of order $1$, which we shall also denote by $Y$. The vector field $Y$ then also acts by duality on the space $\D'(G)$ of distributions on $G$.

\subsection{Representations, associated bundles, and their sections}\label{sec:bundles}

Consider the unitary representations $\varrho_n:\SOdeux\to\End(\C)$, $n\in\Z$, and $\tau:\SOdeux\to\End(\C^2)$, where $\varrho_n({\exp(\theta R)})=e^{-in\theta}$ and $\tau$ is the complexification of the standard representation of $\SOdeux$ on $\R^2$. Here $\C$ and $\C^2$ are equipped with the standard inner products, respectively. Note that under the identification $\C^2=(\C^2)^\ast$ one has $\tau^*=\tau$, where $\tau^*$ is the dual representation. We view $G\to G/M$ as a principal bundle with fiber $M$ and build associated vector bundles $\E,\L^n$on $S\hh^3=G/M$ by defining
\bq
\E:=G\times_\tau\C^2\cong \E^*, \qquad \L^n:=G\times_{\varrho_n}\C.\label{eq:defbundles}
\eq
The notation is chosen such that $\E$ corresponds to the complexification of the (rank 2 real) vector bundle $\E$ introduced in \cite{dfg}.
Note that the representation $\tau$ splits into irreducibles according to $\tau=\varrho_1\oplus \varrho_{-1}$ and that we have $\varrho_n=\varrho_{\pm 1}^{\otimes |n|}$ for $\pm n\geq 0$, which implies  that the line bundles $\L^n$ are tensor powers:
\[
\L^n=(\L^{\pm 1})^{\otimes |n|},\qquad \pm n\geq 0.
\]
Let us also mention that the bundles $\E$ and $\L^n$ carry natural Hermitian metrics corresponding fiber-wise to the standard inner products in $\C$ and $\C^2$, respectively. 

The Hilbert space $L^2(G)$, defined with respect to the Haar measure on $G$ corresponding to our choice of inner product on $\g$, decomposes by Fourier analysis (in other words, the Peter-Weyl theorem for $\SO(2)$ acting by the right regular representation on $L^2(G)$) into a Hilbert sum
\bq
L^2(G)=\bigoplus_{n\in \Z}L^2_n(G),\label{eq:Fourierdecomp}
\eq
where
\[
L^2_n(G):=\{f\in L^2(G) \,|\, f(g\cdot {\exp(\theta R)})=e^{in\theta}f(g),\, \forall\, \theta,\; \forall \,g\in G\}.
\]
For $n\in\Z$ define also the smooth functions equivariant by ${\rm SO}(2)$ of weight $n$:  
\[
C^\infty_n(G):= L^2_n(G)\cap C^\infty(G) = \{ f\in C^\infty(G)\,|\, Rf=inf\}.
\]
There is a natural identification between $C^\infty_n(G)$ and $C^\infty(S\mathbb H^3;\L^n)$. More generally, defining the distribution space
\[
\D'_n(G):=\{f\in \D'(G)\,|\, Rf=inf \},
\]
there is a natural identification between $\D'_n(G)$ and the space $\D'(S\hh^3;\L^{n})$ of distributional sections of the line bundle $\L^n$. 

In a similar spirit, we may identify the sections in $C^\infty(S\hh^3; \E)$ with equivariant functions on $G$ and the distributional sections in $\D'(S\hh^3;\E)$ with equivariant distributions on $G$.

\subsection{Canonical injections of $\L^n$ into $\otimes^m_S\E$}\label{sec:caninj}

Recall that the $\SO(2)$-representation $\tau$ splits into irreducibles according to $\tau=\varrho_1\oplus \varrho_{-1}$. To decompose for $m\in \N_0$ the symmetric tensor power $\otimes_S^m \tau$ into irreducibles, we introduce the surjective symmetrization map 
\bq
\mc{S}: \otimes^m \C^2\to \otimes_S^m \C^2\label{eq:sym1}
\eq
defined by linear extension of the weighted sum over all permutations of order $m$
\bq
v_{i_1}\otimes\cdots\otimes v_{i_m}\mapsto\frac{1}{m!} \sum_{\sigma\in \Pi_m} v_{\sigma(i_1)}\otimes \cdots\otimes v_{\sigma(i_m)}.\label{eq:sym2}
\eq
Choosing an orthonormal basis $\{v_+,v_{-}\}$ of $\C^2$ such that $\tau$ acts on $v_\pm$ by $\varrho_{\pm 1}$, each of the linearly independent elements $s_{l,m-l}:=\mc{S}(v_+^{\otimes l}\otimes v_-^{\otimes m-l})$, $0\leq l\leq m$, spans an irreducible subrepresentation of $\otimes_S^m \tau$ equivalent to $\varrho_{2l-m}$. This shows
\bq
\otimes_S^m \tau=\bigoplus_{l=0}^m\varrho_{2l-m},\qquad 
\otimes^m_S\E\cong\bigoplus_{l=0}^m\L^{2l-m},\label{eq:decomp1290890}
\eq
where $\L^{2l-m}$ injects into $\otimes^m_S\E$ by the map
\bq
\mathcal I^m_{2l-m}:\L^{2l-m} \ni [g, 1] \longmapsto [g,  s_{l,m-l} ]\in \otimes_S^{m} \E,\label{eq:injections}
\eq
extended by linearity. Clearly $\mathcal I^m_{2l-m}$ intertwines the left $G$-actions on $\L^{2l-m}$ and $\otimes_S^{m} \E$. 

Let us consider the action of the trace on the injections. Let $\T: \otimes^{m}_S (\C^2)^*\to \otimes^{m-2}_S (\C^2)^*$ be  the trace operator defined by $\T=0$ if $m\in\{0,1\}$ and for $m\geq 2$ by
\[
(\T\omega)(w_1,\ldots,w_{m-2}):=\omega(e_1,e_1,w_1,\ldots,w_{m-2})+\omega(e_2,e_2,w_1,\ldots,w_{m-2}),\qquad w_j\in \C^2,
\]
where $\{e_1,e_2\}$ is an arbitrary orthonormal basis of $\C^2$ consisting of real vectors $e_1,e_2$. 
We denote the induced bundle map $\T:\otimes^{m}_S \E\to \otimes^{m-2}_S \E$ by the same name,  using the identification $\mc{E}\simeq \mc{E}^*$. 
One computes that $\T$ acts as follows on the basis $\{s_{p,q}\}_{p+q=m}$ of $\otimes_S^{m} \C^2$:
\bq
\T(s_{p,q})=\begin{cases}\frac{pq}{(p+q)(p+q-1)}s_{p-1,q-1},\qquad & p,q\geq 1,\\
\smallskip0,&\text{else}.
\end{cases}\label{eq:trace}
\eq
From this and \eqref{eq:injections} we see that for  $n\in \N_0$ the  injections $\I^{n}_{\pm n}:\L^{\pm n}\hookrightarrow \otimes^{n}_S \E$ fulfill
\bq
\I^{n}_{n}(\L^{n})\oplus \I^{n}_{-n}(\L^{-n})=\{\omega\in \otimes^{n}_S \E\, |\, \T \omega=0\}=:\otimes^{n}_{S,0}\E, \label{eq:tracefree}
\eq
which means that the subbundle $\otimes^{n}_{S,0}\E=\ker \T\subset \otimes^{n}_{S}\E$ of trace-free symmetric tensors of order $n$ can be identified with $\L^{n}\oplus \L^{-n}$.

\subsection{Covariant derivatives, ladder and horocyclic operators, and Anosov decomposition}\label{sec:covariant}

Let us consider the differential operators 
\[\eta_\pm, \mu_\pm:C^\infty(G)\to C^\infty(G)\] 
defined by the Lie algebra elements $\eta_\pm,\mu_\pm$ from \eqref{eq:etamu}. It is a direct consequence of the commutation relations  \eqref{eq:etamucomm}  that $\eta_\pm,\mu_\pm$ restrict to operators 
$\eta_\pm, \mu_\pm: C^\infty_n(G)\to C^\infty_{n\pm 1}(G)$. 
Thus, they induce \emph{ladder operators}
\bq\label{eq:Xeta}
\eta_\pm,\mu_\pm : C^\infty(S\hh^3;\L^n)\to C^\infty(S\hh^3;\L^{n\pm 1})
\eq
which we denote by the same name for each $n$. 
Moreover, as already indicated at the end of Section \ref{eq:algdescr}, the commutation relation $[X,R]=0$ from \eqref{commutation_relations_XUR} implies that  the vector field $X$ on $G$ induces a vector field, also denoted by $X$, on $G/M=S\H^3$. On the other hand, considering $X\in \Cinft(G;TG)$ as a differential operator $C^\infty(G)\to C^\infty(G)$, it leaves $C^\infty_n(G)$ invariant for each $n\in \Z$ and therefore induces operators, denoted ${\bf X}$, 
\bq
{\bf X}:C^\infty(S\hh^3;\L^n)\to C^\infty(S\hh^3;\L^n),\qquad n\in \Z\label{eq:Xmathcal}.
\eq
 In fact, the operator ${\bf X}$ is the covariant derivative along the geodesic vector field $X$ on $S\H^3$ with respect to a linear connection on $\L^n$ for each $n\in \Z$. To explain this, note that there is a connection on the principal $\SO(2)$-bundle $G\to G/M= S\H^3$ defined by the one-form $\Theta\in \Omega^1(G,\m)$ that identifies the left invariant vertical vector field $R$ with the generator $R\in \m=\so(2)$. 
It induces linear connections, all denoted by $\nabla$, on $\L,\E$, and their tensor powers. The connections are compatible with the natural Hermitian metric on those bundles. If we regard smooth sections $f$ of one of those bundles as right-$\SO(2)$-equivariant smooth functions $\til{f}:G\to V$, where $V$ is either $\C$, $\C^2$ or a tensor power of the latter, and a vector field $Y\in C^\infty(S\hh^3; T(S\H^3))$ as a right-$\SO(2)$-equivariant smooth function $\til{Y}:G\to \aL\oplus  \nL^+\oplus \nL^-$, then $\nabla_Y$ is given by
\bq
\til{\nabla_Y (f)}(g)=\frac{d}{dt}\Big|_{t=0}\til{f}(g\exp(t\til{Y}(g))),\qquad g\in G.\label{eq:nabla}
\eq
We will prove later in Lemma \ref{lem:checkcondition} that the curvature $\Omega$ of this connection (and the induced one on the quotient $\Gamma\backslash G/M$) is invariant by the vector field $X$ in the sense that $\iota_X\Omega=0$.

For each $m\in\nn_0$, we define the operator ${\bf X}: C^\infty(S\hh^3; \otimes_S^m\mc{E}^*)\to C^\infty(S\hh^3; \otimes_S^m\mc{E}^*)$ by\footnote{Here we work with $\E^\ast$ rather than $\E$ to match the convention in the reference \cite{dfg}.} 
\bq\label{eq:defX}
{\bf X}f:=\nabla_X f .
\eq 
Note that on sections taking values in an embedded line bundle $\L^n$ inside $\otimes_S^{|n|}\mc{E}^*$ (by  \eqref{eq:injections} and identifying $(\L^n)^\ast=\L^{-n}$), the operator $\Xbf$ coincides with the one introduced in \eqref{eq:Xmathcal}.  
The tangent bundle $T(S\H^3)$ has an Anosov decomposition   into neutral, stable, and unstable subbundles:
\begin{align*}
T(S\H^3)=E_0\oplus E_s\oplus E_u.
\end{align*}
Here $E_0=\R X$ and the vector bundles $E_{s/u}$ are obtained as the associated bundles to the restricted $\Ad(\SO(2))$-representations on the  $\Ad(\SO(2))$-invariant subspaces
$
\nL^\pm\subset \g,
$ 
see the end of Section \ref{eq:algdescr}. 
Dually, we have a decomposition of the cotangent bundle
\begin{align*}
T^\ast(S\H^3)=E_0^*\oplus E_u^*\oplus E_s^*,
\end{align*}
where $E_{s/u}^*$ are defined by
\[
E_u^*(E_0\oplus E_u)=0, \, \quad E_s^*(E_0\oplus E_s)=0, \,\quad E_0^*(E_s\oplus E_u)=0.
\]
Similarly, the decomposition \eqref{eq:Bruhat} induces a decomposition 
\[
TG=\R R\oplus\R X\oplus \til{E}_u\oplus \til{E}_s,\quad T^*G=\R\Theta \oplus 
\til{E}_0^*\oplus \til{E}_s^*\oplus \til{E}_u^* \]
with $E_{s/u}^*(E_{s/u}\oplus \R R\oplus \R X)=0$ and $\til{E}_0^*(\R R\oplus \til{E}_s\oplus \til{E}_u)=0$, where $\Theta$ is the connection $1$-form introduced just below \eqref{eq:Xmathcal}. 
Note that $\til{E}_{s/u/0}$ project to $E_{s/u/0}$ by $d\tilde{\pi}$ if $\tilde{\pi}: G\to G/M$ is the projection.
By the commutation relations \eqref{commutation_relations_XUR}, we have on $G$
\[ d\varphi_t U_j^\pm=e^{\pm t}U_j^\pm,\quad d\varphi_t X=X, \quad d\varphi_t R= R\] 
which means that the differential of the flow $\varphi_t$ of $X$ on $G$ is exponentially contracting/expanding on $\widetilde E_s/\widetilde  E_u$ and it 
preserves the  norm on $\R R\oplus \R X$.  

There is a bundle isomorphism $\theta^-:\E\to \C E_u$  induced by the equivalence of representations $\tau\sim (\Ad(M)|_{\nL^-})_\C$ given by identifying the basis $\{v_+,v_-\}$ of $\C^2$ with the basis $\{\eta_+,\eta_-\}$ of $\nL^-_\C$ from \eqref{eq:etamu}. We use this bundle isomorphism to define the \emph{horocyclic operator} $\U_-:C^\infty(S\hh^3; \otimes_S^m\mc{E}^*)\to C^\infty(S\hh^3; \mc{E}^*\otimes \otimes_S^{m}\mc{E}^*)$ as in \cite{dfg} by 
\[ \forall f\in C^\infty(S\hh^3), \forall w_i\in \mc{E},\quad  
(\U_-f)(w_1,\dots,w_{m+1}) :=(\nabla_{\theta^-(w_1)}f)(w_2,\dots,w_{m+1}) .\]
It suffices for this article to consider these operators acting on (distributional) sections of the space $\otimes^m_S\E^*$. 
From \eqref{eq:nabla} one sees for a section $f\in C^\infty(S\hh^3;\otimes^m_S\E^*)$ that if
\[
f(gM) = [ g, \sum_{K\in\sA^m} \lambda_K\, v_{k_1}^\ast\otimes\cdots \otimes v_{k_m}^\ast ], \quad \lambda_K\in \Cinft(G),\quad g\in G,
\]
where $\sA^m=\{(k_1,\ldots,k_m)\,|\, k_j\in \{+,-\}\}$, the sum runs over all tuples $K$, and $\{v_+^\ast, v_-^\ast\}$ is the dual basis to $\{v_+,v_-\}$, then
\begin{align}\begin{split}
({\bf X} f)(g M) &= [g, \sum_{K\in\sA^m} (X \lambda_K) v_{k_1}^\ast\otimes\cdots \otimes v_{k_m}^\ast ], \\
(\mc{U}_-f)(g M) &= [g, \sum_{s\in \{+,-\}} \sum_{K\in\sA^m} (\eta_s \lambda_K) v_s^\ast\otimes v_{k_1}^\ast\otimes\cdots \otimes v_{k_m}^\ast ].\label{eq:XU_-}\end{split}
\end{align}
Note that in this notation the fact that we are dealing with symmetric tensors is implicitly reflected in relations between the coefficients $\lambda_K$, and similarly for $X \lambda_K$ and $\eta_s \lambda_K$. 
As remarked in \cite{dfg}, the operator $\U_-^mf$ is a symmetric tensor of degree $m$ if $f$ is a smooth function on $S\H^3$. In order to preserve this property when acting on symmetric tensors, we introduce the symmetrisation operator $\S$ and consider $\S\U_-$. In coordinates this amounts to stating
\bq
(\S\U_-f)(g M) = [g, \sum_{s\in \{+,-\}} \sum_{K\in\sA^m} (\eta_s \lambda_K)\, \mc{S}(v_s^\ast\otimes v_{k_1}^\ast\otimes\cdots \otimes v_{k_m}^\ast) ]\label{eq:Uminuscoord}
\eq 
The operator $\S\U_-$ yields for each $m\in \N_0$ a map
\bq
\S\U_-: C^\infty(S\hh^3; \otimes_S^m\E^*)\to C^\infty(S\hh^3; \otimes_S^{m+1}\E^*).\label{eq:SUminus}
\eq
This map extends by duality to an operator $\S\U_-$ on distributional sections.

\section{Resolvent of the frame flow}\label{sec:resolv}

In this article we are interested in resonances of the frame flow  associated to the geodesic flow on $\Gamma\backslash S\H^3$ when $\Gamma\subset G$ is a co-compact discrete subgroup with no torsion. Let us fix the notation  
\bq
 {\bf M}:=\Gamma\backslash\H^3=\Gamma\backslash G/K, \quad \M:= S{\bf M}=\Gamma\backslash S\hh^3=\Gamma\backslash G/M, \quad 
F{\bf M}:=\Gamma\backslash G.\label{eq:notation}
\eq
The frame flow is technically the flow on $\Gamma\backslash G$ generated by the vector field $X$. 
First, we notice that all the objects, operators, spaces, bundles introduced in the previous section on $G$ and $G/M$ descend to $\Gamma\backslash G=F{\bf M}$ and $\gam G/M=\M$, respectively. 
Due to the Fourier mode decomposition \eqref{eq:Fourierdecomp}, which descends to a Fourier mode decomposition of $L^2(\gam G)$, we can understand the resolvent of the generator $X$ of the frame flow on $F{\bf M}$ by analyzing ${\bf X}=\nabla_X$ acting on sections of the line bundles $\L^n$ 
over $\M$.

\subsection{The resolvent in ${\rm Re}(\la)>0$}\label{sec:ResRepos}

On $F{\bf M}=\Gamma\backslash G$, there is an invariant measure $\mu_G$ with respect to $X$, which implies that $iX$ is self-adjoint on $L^2(\Gamma\backslash G;\mu_G)$.
The operator $X: C^\infty(F{\bf M})\to C^\infty(F{\bf M})$ has a well-defined resolvent for ${\rm Re}(\la)>0$
\bq\label{resolventff}
R(\lambda):=(-X - \lambda)^{-1}: L^2({\bf M})\to L^2({\bf M}), \quad  
R(\lambda)f=-\int_{0}^\infty e^{-\lambda t}\varphi^*_{-t}f \, dt
\eq
where $\varphi_t=e^{tX}$ is the flow of $X$ at time $t$ on $F{\bf M}=\Gamma\backslash G$. Here the operator $R(\la):L^2(F{\bf M})\to L^2(F{\bf M})$ is bounded and fulfills
\bq
\norm{R(\lambda)}_{L^2(\Gamma\backslash G)\to L^2(\Gamma\backslash G)}\leq 1/{\rm Re}(\la),\label{eq:resnormRlambda}
\eq
as one checks by using that $\|\varphi_t^*f\|_{L^2}=\|f\|_{L^2}$ in \eqref{resolventff}.
In order to extend this operator family, we can use the fact that  $R(\lambda)$ preserves each space $L^2_n(\Gamma\backslash G)$ and thus
\begin{equation}\label{sumRn} 
R(\la)f= \sum_{n\in\zz} R_n(\la)f_n\quad \textrm{ with } R_n(\la)=(-X-\la)^{-1}: L^2_n(\Gamma\backslash G)\to L^2_n(\Gamma\backslash G)
\end{equation} 
where $f=\sum_{n\in\zz}f_n$ with $f_n\in L^2_n(\Gamma\backslash G)$.  In particular, the sum in \eqref{sumRn}  converges. We will then study each $R_n(\la)$, which we can also view as the resolvent of the operator ${\bf X}=\nabla_X$ 
on sections of the bundle $\mc{L}^n$ over $S{\bf M}$. From \eqref{eq:resnormRlambda} we get  
for ${\rm Re}(\la)>0$ the norm bound
\bq \|R_n(\la)\|_{L^2_n(\Gamma\backslash G)\to L^2_n(\Gamma\backslash G)}\leq 1/{\rm Re}(\la)\qquad \forall\;n\in \Z.\label{eq:resnormL2}
\eq
Since we shall analyze the family of operators $({\bf X}+\la)$ on sections of $\mc{L}^n$ over $\M=S{\bf M}$ using microlocal methods, it will be convenient to use an $n$-dependent quantization for families of bundles. This is the topic of the next section.
 
\subsection{Semiclassical analysis for line bundle tensor powers}\label{sec:calculus}
We introduce some semiclassical tools to analyze families of operators acting on sections of high-tensorial powers of a line bundle. We will use the formalism of Charles \cite{charles2000}.

Let $\M$ be a closed Riemannian manifold of dimension $d$ and consider a line bundle $\mc{L}$ over $\mc{M}$ equipped with a Hermitian product and a Hermitian connection $\nabla$.   
 We shall now be interested in the family $\mc{L}^n:=\L^{\otimes n}$ (for $n\in \N_0$) of tensor powers of $\L$ and consider the power $n$ as a semiclassical asymptotic parameter. More precisely, we consider a semiclassical parameter $h\in D\subset (0,2]$ for some set $D$ whose closure in $[0,2]$ contains $0$, together with a \emph{tensor power map}
 \[
D\owns h\mapsto n(h)\in \N_0
 \]
 of which we assume that it does not grow faster than the inverse of $h$:
 \bq
 \exists\;C_0>0: hn(h)\leq C_0\quad \forall\;h\in D.\label{eq:fundprop}
 \eq
A ``trivial'' example is $D:=(0,1]$, $n(h):=n_0$, $C_0:=n_0$ for some $n_0\in \N_0$, which leads to the usual semiclassical analysis on the fixed line bundle $\L^{n_0}$. The canonical ``non-trivial'' example consists in taking $D:=\{1/j \,|\, j\in \nn^*\}$ and $n(h):=1/h$, $C_0:=1$. Such a reciprocal relation between $h$ and $n(h)$ is used (in greater generality) in the gauge theory works \cite{hogreve83,schradertaylor84,schradertaylor89,guilleminuribe86,guilleminuribe89,
 guilleminuribe90,zelditch19921} mentioned in the introduction.  In some cases we will just take this choice but in certain cases we shall need  $n(h)$ so that $h n(h)\to 0$ as $h\to 0$.  
 The tensor power map $n(h)$ allows us to consider the family of line bundles $\mc{L}^{n(h)}$ for $h\in D$.
We will roughly follow the approach of \cite[Chapter 4]{charles2000} (see also \cite[Sections 12.11-12.13]{guilleminsternbergnotes}) with few modifications.
The connection $\nabla$ on $\L$ induces a connection $\nabla^{n(h)}$ on $\L^{n(h)}$ for all $h\in D$ such that for each local nowhere-vanishing section ${\bf s}\in \Cinft(W;\L)$ over some open set $W\subset \M$, one has
$$\nabla^{n(h)}(f\textbf{s}^{n(h)})= df \otimes {\bf s}^{n(h)}+f{n(h)}{\bf s}^{-1}(\nabla{\bf s})\otimes{\bf s}^{n(h)}\qquad \forall\; f\in \CT(W).$$ 
Here and in the following we define ${\bf s}^{-1}(\mathbf t):=\eta$ if $\mathbf t=\eta \otimes \mathbf s$ with $\eta\in \Cinft(W;\Lambda^1(T^\ast W)_\C)$. Observe that in such a local trivialisation, for each smooth vector field $X$, one has 
$${\bf s}^{-{n(h)}}h\nabla_X^{n(h)}(f{\bf s}^{n(h)})=P_{h}f,$$ where $P_{h}:=hX+hn(h) {\bf s}^{-1}\nabla_X{\bf s}$ is a first order semiclassical differential operator on $W$.

 In particular, one can define the semiclassical  Sobolev norms for each $N\in\R$
\bq
\forall u\in C^\infty(\mc{M};\mc{L}^{n(h)}), \quad \|u\|_{H^N_{h}(\L^{n(h)})}:= \| (1+h^2\Delta_{n(h)})^{N/2}u\|_{L^2(\L^{n(h)})},\label{eq:sobolevnorm1}
\eq  
where $\Delta_{n(h)}=(\nabla^{n(h)})^*\nabla^{n(h)}$, and we denote by $H^N_{h}(\mc{M};\mc{L}^{n(h)})$ the completion of  the space $\Cinft(\M;\L^{n(h)})$ using this norm. 

Let $\delta\in [0,1/2)$ be small.
We fix an order function $m\in S^0(\M)$ as in \cite{FS11} and we denote by 
$S_{h,\delta}^{m}(\mc{M})$ the set of symbols of order $m$, i.e., the smooth functions $a_{h}\in C^\infty(T^*\mc{M})$ satisfying for all $h\in D$, $\alpha,\beta$ multiindices, that  for any local chart $W\subset \mc{M}$ there is $C_{\alpha,\beta}>0$ independent of $h$ so that
\bq \forall (x,\xi)\in T^*W, \quad  |\pl_x^\alpha\pl_\xi^{\beta} a_{h}(x,\xi)| \leq C_{\alpha,\beta}\cjg \xi\cjd^{m(x,\xi)+\delta|\alpha|-(1-\delta)|\beta|}.\label{eq:symbcondition}
\eq
We define the set $\Psi^{m}_{h,\delta}(\M;\L^{n(h)})$ of semiclassical pseudodifferential operators of order  $m$ acting on smooth (and by duality on distributional) sections of $\L^{n(h)}$ as the set of all (families of) continuous linear maps $A_{h}:C^\infty(\mc{M};\mathcal L^{n(h)})\to C^\infty(\mc{M};\mathcal L^{n(h)})$ such that: 

\noindent 1) $\forall \chi,\chi'\in C^\infty(\mc{M})$ with $\supp(\chi)\cap \supp(\chi')=\emptyset$, $\chi A_{h} \chi': \mc{D}'(\mc{M};\mc{L}^{n(h)})\to C^\infty(\mc{M};\mc{L}^{n(h)})$ is continuous 
with operator norms  
\[ \forall N>0, \quad \|\chi A_{h} \chi'\|_{H_{h}^{-N}(\L^{n(h)})\to H_{h}^{N}(\L^{n(h)})}=\mc{O}(h^\infty),\]
2) in any local chart $W\subset \mc{M}$ with a local 
trivializing section ${\bf s}\in C^\infty(W;\L)$ fulfilling $\|{\bf s}\|=1$ fiber-wise, there exists $a_{h}\in S_{h,\delta}^{m}(W)$ 
such that for all $f\in \CT(W)$ and $x\in W$
\begin{equation}\label{eq:localhpdo}
{\bf s}^{-n(h)}A_{h}(f{\bf s}^{n(h)})(x)=\frac{1}{(2\pi h)^d}\int e^{\frac{i}{h}(x-x')\cdot\xi}a_{h}(x,\xi) f(x')dx'd\xi.
\end{equation}
Following \cite{charles2000}, one can define a notion of \emph{modified principal symbol} of $A_{h}$ as follows: if $\beta:=\sum_{j=1}^n\beta_j dx_j \in C^\infty(W;T^*\rr^d)$ is the local connection $1$-form  such that 
$\nabla {\bf s}=-i \beta\otimes {\bf s}$, then the modified principal symbol of an operator $A_h$ as in \eqref{eq:localhpdo} is defined by
\bq
\sigma_{h,n(h)}(A_{h})(x,\xi):=[a_h\big(x,\xi+ h n(h)\beta(x)\big)]\in S_{h}^{m}(\M)/h S_{h}^{m-1+2\delta}(\M),\label{eq:defprincsymb}
\eq
where $[a_{h}]$ means the class of $a_{h}$ in $S_{h}^{m}(\M)/h S_{h}^{m-1+2\delta}(\M)$. It is easy to check that $\sigma_{h,n(h)}(A_h)\in C^\infty(T^*\mc{M})$ is well-defined independently of the coordinate system and the trivialisation of $\mc{L}$: 
this follows from the fact that for a change ${\bf s}'=e^{-i\omega}{\bf s}$ of local trivialisation with $\omega\in C^\infty(W;\rr)$, the operator $A_{h}$ fulfills for $f\in \CT(W)$

\[\begin{split} 
{\bf s'}^{-n(h)}A_{h}(f{\bf s'}^{n(h)})(x)&=  e^{in(h)\omega(x)}\,{\bf s}^{-n(h)}A_{h}(fe^{-in(h)\omega}{\bf s}^{n(h)})(x)\\
&=  \frac{1}{(2\pi h)^d}\int e^{\frac{i}{h}[(\xi+hn(h)(d\omega_x+Q_{x,x'}(x-x')))\cdot(x-x'))]}a_{h}(x,\xi) f(x')dx'd\xi\\
&=  \frac{1}{(2\pi h)^d}\int e^{\frac{i}{h}\xi\cdot(x-x')}a_{h}(x,\xi-h n(h)(d\omega_x+Q_{x,x'})(x-x')) f(x')dx'd\xi,
\end{split}\]
where $Q_{x,x'}$ is a smooth symmetric matrix. We Taylor expand the symbol 
\[a_{h}(x,\xi-h n(h)d\omega_x)- hn(h)\int_0^1\pl_{\xi}a_{h}(x,\xi-h n(h)d\omega_x+tQ_{x,x'}(x-x'))\cdot Q_{x,x'}(x-x')dt\]
and use $h\pl_\xi(e^{\frac{i}{h}\xi\cdot(x-x')})=i(x-x')e^{\frac{i}{h}\xi\cdot(x-x')}$ with an integration by parts in $\xi$, showing that the remainder produces an extra $\mc{O}(h)$ in the symbol. In particular, since the connection form $\beta'$ in the trivialisation ${\bf s}'$ is related to $\beta$ by $\beta'=\beta+d\omega$, we see that $\sigma_{h,n(h)}(A_{h})$ is invariant under the change of trivialisation as an element in $S_{h,\delta}^{m}(\M)/h S_{h,\delta}^{m-1+2\delta}(\M)$.
 We notice that the condition $hn(h)\leq C_0$ is important in order that $\Psi^{m}_{h,\delta}(\M;\L^{n(h)})$ be well-defined (under change of coordinates and choice of local trivializing section ${\bf s}$). The fact that $\mc{M}$ is compact is also 
important to have the connection form uniformly bounded.
We can also define a quantization procedure 
\bq
{\rm Op}_{h,n(h)}: S_{h,\delta}^{m}(\mc{M})\to \Psi_{h,\delta}^{m}(\mc{M};\L^{n(h)})\label{eq:Oph}
\eq so that for $a_h\in S_{h,\delta}^{m}(\mc{M})$
\bq
\sigma_{h,n(h)} ({\rm Op}_{h,n(h)}(a_h))=[a_h].\label{eq:symbolrep}
\eq 
This is achieved, similarly as in the case of a trivial line bundle, by gluing together local quantization maps using a  partition of unity. If $\chi\in \CT(W)$ is a cutoff function on a chart $W$ with local connection $1$-form $\beta$ and trivializing section ${\bf s}$,  the local quantization of $a_h\chi$ is defined by the following formula, valid for any function  $f\in \CT(W)$:
\bq\label{expquantization}
{\bf s}^{-n(h)}{\rm Op}_{h,n(h)}(a_h\chi)(f{\bf s}^{n(h)})(x):=\frac{1}{(2\pi h)^d}\int e^{\frac{i}{h}(x-x')\cdot\xi}a_{h}(x,\xi- h n(h)\beta(x))\chi(x) f(x')dx'd\xi.
\eq
Thanks to the condition \eqref{eq:fundprop} and the compact support of $\chi$, the function $(x,\xi)\mapsto a_{h}(x,\xi- h n(h)\beta(x))\chi(x)$ lies in $S_{h,\delta}^{m}(W)$, so the above formula defines an operator in $\Psi_{h,\delta}^{m}(W;\L^{n(h)})$. By \eqref{eq:defprincsymb} the symbol of this local operator is represented by $a_h\chi$, as desired. 

We will also write $\Psi_h^{m}(\mc{M};\L^{n(h)}):=\cap_{\delta>0}\Psi^m_{h,\delta}(\mc{M};\L^{n(h)})$ and $S_{h}^{m}(\mc{M}):=\cap_{\delta>0} S_{h,\delta}^{m}(\mc{M})$. Furthermore, if it is clear from the context which tensor power map is used, we will use the simplified notation ${\rm Op}_{h}:={\rm Op}_{h,n(h)}$ and $\sigma_{h}:=\sigma_{h,n(h)}$.

Just like for the trivial line bundle case, we have all the same properties: composition, boundedness, elliptic estimates, G\aa rding inequalities (see \cite{zworski} and \cite[Appendix E]{dyatlovzworskibook} for details). 

For $A_{h}\in \Psi_h^{m}(\mc{M};\L^{n(h)})$ and $B_{h}\in \Psi_h^{m'}(\mc{M};\L^{n(h)})$, we have 
$A_{h}B_{h}\in \Psi_h^{m+m'}(\mc{M};\L^{n(h)})$ and 
\bq
 \sigma_{h,n(h)}(A_{h}B_{h})=\sigma_{h,n(h)}(A_{h})\sigma_{h,n(h)}(B_{h}).\label{eq:compositionformula}
 \eq
As in \cite[Theorem 13.13]{zworski}, operators  $A_{h}\in \Psi^{0}_h(\mc{M};\L^{n(h)})$ are bounded on $L^2(\M;\L^{n(h)})$ with norm 
\bq
 \|A_{h}\|_{L^2\to L^2}\leq \sup_{(x,\xi)\in T^*\mc{M}}\|\sigma_{h,n(h)}(A_{h})(x,\xi)\|+\mc{O}(h).\label{eq:normbound}
 \eq
We can then define the semiclassical wave-front sets of $h$-tempered sections $u\in\mc{D}'(\mc{M};\L^{n(h)})$ and of operators $A_{h}\in \Psi_h(\mc{M};\L^{n(h)})$ just as for trivial bundles (see for example \cite[Section E.2]{dyatlovzworskibook}); we will denote them by ${\rm WF}_{h}(u)\subset T^*\mc{M}$ and ${\rm WF}_{h}(A_{h})\subset T^*\mc{M}$.
\begin{exam}\label{ex:Ahcovderiv} Since it will be our main application, let us consider as an example the operator 
$-ih\nabla_Y^{n(h)}\in \Psi^1_{h}(\mc{M};\L^{n(h)})$ where $Y$ is a smooth vector field and $h\in D$. Since for each local section ${\bf s}$ and each $f\in C^\infty(\mc{M})$  one has 
\[ -ih\nabla_Y^{n(h)}(f{\bf s}^{n(h)})={\bf s}^{n(h)} \big(-ihYf-ih n(h)f {\bf s}^{-1}\nabla_Y{\bf s}\big)=
{\bf s}^{n(h)} \big(-ihY-h n(h)\beta(Y)\big)f,\]
we obtain
\begin{equation}\label{symbolPh}
\sigma_{h,n(h)}(-ih\nabla_Y^{n(h)})(x,\xi)=[\xi(Y(x))].
\end{equation} 
\end{exam}
The main novelty is the behavior of the principal symbol with respect to commutators (or, more precisely, the behavior of the subprincipal symbol which we do not define here). To describe this, let $\omega_0\in \Omega^2(T^\ast \M)$ be the canonical Liouville symplectic form on the total space of the cotangent bundle $T^\ast \M\stackrel{\pi}{\to}\M$ and $\Omega_{\nabla}\in \Omega^2(\M)$ the curvature form of the connection $\nabla$ on $\L$. Then, for each $\rho\in \R$,
\bq
\omega_\rho:=\omega_0 + i\rho\,\pi^\ast \Omega_{\nabla},\label{eq:omega_rho}
\eq
is a new symplectic form on $T^\ast \M$.  Indeed, $\Omega_{\nabla}$ is closed and $\omega_0$ remains symplectic upon addition of the pullback of any closed form along $\pi$. See \cite[Thm.\ 5.2]{marlesymp}, for example.  

The symplectic form $\omega_\rho$ defines a Poisson bracket $\{\cdot,\cdot\}_{\omega_\rho}$ on $C^\infty(T^\ast \M)$ and for each $f\in \Cinft(T^\ast \M)$ a Hamiltonian vector field $H_f^{\omega_\rho}$ characterized by\footnote{The opposite sign convention $\{f,g\}_{\omega_\rho}=-H_f^{\omega_\rho}g$ is also common in the literature. Our sign convention agrees with that of \cite[(A.2.1)]{dyatlovzworskibook}.}
\bq
\iota_{H^{\omega_\rho}_f}\omega_\rho=df,\qquad  \{f,g\}_{\omega_\rho}=H_f^{\omega_\rho}g\qquad \forall\; g\in C^\infty(T^\ast \M). \label{eq:defhamilt}
\eq
 We then  have the following result (compare \cite[(4.4)]{charles2000}): for $A_{h}\in \Psi^{m}_h(\M;\L^{n(h)})$ and $B_{h}\in \Psi^{m'}_h(\M;\L^{n(h)})$, we have $\frac{i}{h}[A_{h},B_{h}]\in \Psi_h^{m+m'-1}(\mc{M};\L^{n(h)})$  and if $a_h,b_h\in \Cinft(T^\ast \M)$ represent $\sigma_{h,n(h)} (A_{h})$, $\sigma_{h,n(h)} (B_{h})$, respectively, then
\bq\label{eq:commpoisson}
\sigma_{h,n(h)} \Big(\frac{i}{h}[A_{h},B_{h}]\Big)= [\{a_h,b_h\}_{\omega_{hn(h)}}].
\eq
This can be seen by considering $A_h={\rm Op}_{h,n(h)}(a_h)$, $B_h={\rm Op}_{h,n(h)}(b_h)$ using the expression \eqref{expquantization}, then the usual composition formula \cite[Theorem 4.14]{zworski} shows that the local symbol of $\frac{i}{h}[A_{h},B_{h}]$ in local coordinates and in the trivialization ${\bf s}$ is given by $\{\tilde{a}_h,\tilde{b}_h\}_{\omega_0}+\mc{O}(h)$ where 
$\tilde{a}_h=a_h \circ \psi$ and $\tilde{b}_h=b_h\circ \psi$, with $\psi(x,\xi):=(x,\xi-hn(h)\beta(x))$. In the trivialization $\omega_{hn(h)}=\omega_0+hn(h)\pi^*d\beta$,  
and $H_{a_h}^{\omega_{hn(h)}}=H_{a_h}^{\omega_0}-hn(h)\sum_{k,l}(\pl_{x_k}\beta_l-\pl_{x_l}\beta_k)\pl_{\xi_k}a_h\pl_{\xi_l}$ if $\beta=\sum_{k}\beta_k dx_k$. It is then an easy computation to check that
\[\{\tilde{a}_h,\tilde{b}_h\}_{\omega_0}=H_{\tilde{a_h}}^{\omega_0}(\tilde{b}_h)=H_{a_h}^{\omega_{hn(h)}}(b_h)\circ \psi=\{a_h,b_h\}_{\omega_{hn(h)}}\circ \psi,\]
giving \eqref{eq:commpoisson}.

If the tensor power map $n(h)$ is such that the limit $L:=\lim_{h\to 0}hn(h)\in [0,\infty)$ exists, then the symplectic form $\omega_{hn(h)}$ converges as $h\to 0$ to the $h$-independent symplectic form $\omega_{L}$, and in view of \eqref{eq:fundprop} one can always achieve this by making $D$ smaller (i.e., passing to a subsequence). One then also obtains 
\bq
\{f,g\}_{\omega_{hn(h)}}\stackrel{h\to 0}\longrightarrow \{f,g\}_{\omega_L}\quad \text{in }\Cinft(T^\ast \M)\qquad\forall\; f,g\in \Cinft(T^\ast \M). \label{eq:Poissonconvergence}
\eq

\begin{lemma}\label{lem:Hp}Let $Y$ be a smooth vector field on $\M$ satisfying $\iota_Y\Omega_\nabla=0$. Then the function $p_0:T^\ast \M\to \R$ defined by $p_0(x,\xi):=\xi(Y(x))$ satisfies
\bq
H_{p_0}^{\omega_\rho}=H_{p_0}^{\omega_0} \quad\forall\;\rho\in \R.\label{Hpomega=Hp2}
\eq
\end{lemma}
\begin{proof}
 Recall from \eqref{eq:omega_rho} that $\omega_\rho=\omega_0 + i\rho\,\pi^\ast \Omega_{\nabla}$ and that $H^{\omega_\rho}_p$ is defined by the identity $\iota_{H^{\omega_\rho}_{p_0}}\omega_\rho=dp_0$. Let us therefore check this identity for $H^{\omega_0}_{p_0}$:
\bq
\iota_{H^{\omega_0}_{p_0}}\omega_\rho=\iota_{H^{\omega_0}_{p_0}}\omega_0 + i\rho\,\iota_{H^{\omega_0}_{p_0}}\pi^\ast \Omega_{\nabla}=dp_0 + i\rho\,\iota_{H^{\omega_0}_{p_0}}\pi^\ast \Omega_{\nabla}.\label{eq:defeqcheck}
\eq
Since $p_0(x,\xi)=\xi(Y(x))$, one reads off that $d\pi H^{\omega_0}_{p_0}=Y$. Together with the assumption $\iota_Y\Omega_\nabla=0$, this implies that $\iota_{H^{\omega_0}_{p_0}}\pi^\ast \Omega_{\nabla}=0$ and we obtain \eqref{Hpomega=Hp2}.
\end{proof}

In this situation, the results in \cite[Section E.3]{dyatlovzworskibook} on semiclassical measures  apply just equally in our setting, with the Hamiltonian vector fields appearing in the results being defined with respect to $\omega_L$. Using Lemma \ref{lem:Hp} and Example \ref{ex:Ahcovderiv}, the results in \cite[Section E.4]{dyatlovzworskibook} on propagation estimates (real principal type, radial estimates for sink and source)  apply verbatim to any operator $-ihY+Q$ with $Q\in \Psi_{h}^{0}(\mc{M};\mc{L}^{n(h)})$ and $\iota_Y\Omega_\nabla=0$, i.e., with no changes to the proofs. Here all difficulties arising from the modified symplectic form do not arise due to the fact that the Hamilton field is the same as in \cite{dyatlovzworskibook}.

\subsection{Resonances and resonant states on the line bundles $\mc{L}^n$}\label{sec:sobG}

By the works \cite{FS11}, \cite{DZ16} (see for example \cite{DG16} for the case of general bundles), the operator ${\bf X}+\la: C^\infty(\mc{M};\mc{L}^n)\to C^\infty(\mc{M};\mc{L}^n)$ can be made Fredholm on some anisotropic Sobolev spaces, implying that the resolvent $R_n(\la)$ admits a meromorphic extension to $\cc$. Let us briefly recall these results, and in particular the definition of the anisotropic Sobolev spaces. 
By \cite[Lemma 1.2]{FS11}, there are functions $m,F\in C^\infty(T^*\mc{M})$ with $m$ (resp.\ $F$) homogeneous of degree $0$ (resp.\ $1$) for $|\xi|>r$ (for some large enough $r>0$), $F>0$ such that, if  $G:=m\log(F)$ and if 
$H^{\omega_0}_p$ is the Hamiltonian vector field of $p(x,\xi)=\xi(X)$ with respect to the standard symplectic form ${\omega_0}$ on $T^\ast \M$, we have for all $\xi$ with $|\xi|>r$
\[
H^{\omega_0}_pm(x,\xi) \leq 0, \qquad H^{\omega_0}_p G(x,\xi)\leq 0,\qquad 
m(x,\xi)=\begin{cases}
1 \quad & \textrm{ for }\xi \textrm{ near }E_s^*,\\
-1 & \textrm{ for }\xi \textrm{ near }E_u^*. \end{cases}
\]  
Now, we take $D:=\{1/n\, |\, n\in \nn^*\}\cup\{2\}$ and define $n(h):=1/h$ for $h\in D\cap (0,1]$ and $n(2):=0$. We can then use the microlocal quantization map $\Op_{h}=\Op_{h,n(h)}$ from \eqref{eq:Oph} to associate to every appropriate symbol function $a\in\Cinft(T^\ast \M)$ a pseudodifferential operator $\Op_{h}(a):\D'(\M;\mc{L}^{n(h)})\to \D'(\M;\mc{L}^{n(h)})$. Similarly, to deal with sections of $\mc{L}^{-n}$ for $n\in\nn^*$, we shall simply replace $\mc{L}$ by its inverse (dual) bundle $\mc{L}^{-1}$. 
Let us introduce for $N>0$ the operator $$A^N_{h,n(h)}:=\Op_{h}(e^{NG}):\D'(\M;\mc{L}^{n(h)})\to \D'(\M;\mc{L}^{n(h)}),$$ where $G$ as above can be chosen so that $A^N_{h,n(h)}:\Cinft(\M;\mc{L}^{n(h)})\to \Cinft(\M;\mc{L}^{n(h)})$ is invertible for all $h\in D$. Indeed, the composition formula \eqref{eq:compositionformula} and the norm bound \eqref{eq:normbound} together imply for all small enough $h\in D$ that $A^N_{h,n(h)}$ is invertible with 
\bq
 (A^N_{h,n(h)})^{-1}=\Op_{h,n(h)}(e^{-NG})(I+R(h))^{-1}\label{eq:Ahinverse}
\eq
for some operator $R(h)\in \Psi_h^{-1}(\mc{M};\L^{n(h)})\subset \Psi_h^{0}(\mc{M};\L^{n(h)})$ with 
\bq
\norm{R(h)}_{{L}^2(\M;\mc{L}^{n(h)})\to {L}^2(\M;\mc{L}^{n(h)})}\leq 1/2,\label{eq:Rhnorm}
\eq 
and using a rescaling $G(x,\eps \xi)$ for $\eps>0$ small enough we can arrange for the above to actually hold for all $h\in D$. 
For each $N>0$ and $h\in D$ we then define the anisotropic Sobolev space 
\bq
\mathcal H_{h}^{Nm}(\mc{M};\mc{L}^{n(h)}):=(A^{N}_{h,n(h)})^{-1}({L}^2(\M;\mc{L}^{n(h)})),\label{eq:defsemiclsobolev}
\eq
also denoted $\mathcal H_h^{Nm}$ for notational simplicity, which is equipped with the norm
\[
\norm{f}_{\mathcal H_h^{Nm}}:=\norm{A^{N}_{h,n(h)}f}_{{L}^2(\M;\mc{L}^{n(h)})}.
\]

\begin{lemma}\label{FaureSjostrand}
For each $N>0$, the operator ${\bf X}=\nabla^{n(h)}_{X}:\D^{Nm}_{h}\to \mathcal H_{h}^{Nm}$  is a closed unbounded operator on the domain
 $\D^{Nm}_h:=\{f\in\mathcal H_{h}^{Nm}\, |\, {\bf X} f\in \mathcal H_{h}^{Nm}\}$ 
and ${\bf X}+\lambda: \D^{Nm}_{h}\to \mathcal H_{h}^{Nm}$ is Fredholm for all $\lambda\in \C$ with $\Re \lambda >-N$ and all $h\in D$. 
\end{lemma}
\begin{proof} The proof is contained in Faure-Sj\"ostrand \cite{FS11}, the only non transparent fact is that the property holds in $\Re \lambda >-N$ (\cite{FS11} claim that there exists $c>0$ so that  the result holds for ${\rm Re}(\la)>-cN$). However since the curvature is $-1$, the contraction/dilation exponents of the geodesic flow are equal to $1$, so it follows readily from the proof of Faure-Sj\"ostrand \cite{FS11} that we can choose $c=1$. Indeed, take the proof of  \cite[Lemma 1.2]{FS11} and its notation, except for the Hamilton field which is denoted ${\bf X}$ in \cite{FS11} and that we denote by $H^{\omega_0}_{p}$: choose $u=-1,s=1$, note that $\tilde{m}=1-\mc{O}(\eps)$ in $\tilde{N}_s$ and 
 $\tilde{m}=-1+\mc{O}(\eps)$ in $\tilde{N}_u$ in \cite[(2.5) and (2.6)]{FS11}, the constant $C$ equals $1$ in \cite[(2.5) and (2.6)]{FS11}, implying that $H^{\omega_0}_{p}G_m\leq -1+\mc{O}(\eps)$ outside $\tilde{N}_0$. Now in \cite[Section 3.2]{FS11}, the operator $\hat{P}$ of \cite{FS11} is replaced by 
$\hat{P}:=(A^N_{h,n(h)})(-ih{\bf X})(A^N_{h,n(h)})^{-1}$ in our case.
Since our operator $\nabla_X$ is skew-adjoint on $L^2$, the principal symbol 
$P_2$ of $\frac{i}{2h}(\hat{P}^*-\hat{P})$ is given (see \cite[Lemma 5.3]{FS11}) by 
\[P_2=H_{p_0}^{\omega_0}(G_m).\]
Applying the arguments of Section 3.3. of \cite{FS11} (with the additional input of the semiclassical parameter) we obtain \cite[Lemma 3.4 and Lemma 3.5]{FS11} with the constants $C=0$ and $C_m=N(1-\eps)$ and $z=h\lambda$. This means that $(\hat{P}-ih\lambda)=-ih({\bf X}+\la)$ is Fredholm in ${\rm Re}(\la)>-N(1-\eps)$ and its resolvent is meromorphic there. Since $\eps>0$ can be made arbitrarily small, this yields the result.
\end{proof}
Since $A^N_{h,n(h)}$  is invertible on $\Cinft(\M;\mc{L}^{n(h)})$, each $\mathcal H_h^{Nm}$ contains $\Cinft(\M;\mc{L}^{n(h)})$, which leads to the following meromorphic extension result: 
\begin{prop}\label{prop:merom_resolvent}
The resolvent $R_{n}(\lambda)$, defined for each $n\in \Z$ and $\Re\lambda>0$ by \eqref{resolventff} and \eqref{sumRn},  has a meromorphic continuation to $\C$ as 
a family of continuous operators
\[ 
 R_{n}(\lambda): \Cinft(\M;\mc{L}^n) \to \mathcal D'(\M; \mc{L}^n).
\]
Given a pole $\lambda_0$ of order $J$, the resolvent takes the form
\bq
 R_{n}(\lambda) = R^H_{n}(\lambda) -\sum_{j=1}^J\frac{(-{\bf X}-\lambda_0)^{j-1}\Pi_{n}^{\lambda_0}}{(\lambda-\lambda_0)^j},\label{eq:presentationRn}
\eq
where $R^H_{n}(\lambda):\Cinft(\M;\mc{L}^n) \to \mathcal D'(\M; \mc{L}^n)$ is 
a holomorphic family of continuous operators and 
$\Pi_{n}^{\lambda_0}:\Cinft(\M;\mc{L}^n) \to \mathcal D'(\M; \mc{L}^n)$ is a finite 
rank operator. Furthermore, the image of the residue operator is given by\footnote{Here $\textup{WF}(s)$ is the wave front set of the distributional section $s$, which microlocally describes the directions in which $s$ is singular. See \cite[Appendix C]{kuester-weich18} for details about wave front sets of distributional sections of vector bundles.} 
\bq\label{eq:range_pi0}
 {\rm Ran}(\Pi_{n}^{\lambda_0}) = \{u\in\mathcal D'(\M; \mc{L}^n)\,|\, ({\bf X}+ \lambda_0)^Ju=0, 
 \textup{WF}(u)\subset E_u^*\}.
\eq
Conversely, if for some $\lambda_0\in\C$ there is $u\in\mathcal D'(\M; \L^n)\setminus\{0\}$
such that $\tu{WF(u)}\subset E_u^*$ and
$(\Xbf+\lambda_0)^ku=0$ for some $k\in \nn^*$, then $\lambda_0$ is a pole of $R_{n}(\lambda)$ 
and $u\in {\rm Ran}(\Pi^{\lambda_0}_{n})$.
\end{prop}
\begin{proof}
As explained in Lemma \ref{FaureSjostrand}, the proof is contained in \cite{FS11}. We also refer to 
\cite[eq (0.12)]{DG16} for the characterization of $\textup{Im}(\Pi_{\lambda_0})$ in \eqref{eq:range_pi0} and \cite[eq (3.44), (3.55)]{DG16} for the structure of the resolvent in a neighborhood of a pole.
\end{proof}

\begin{definition}\label{def:resonances}
 We call a pole of $R_{n}(\la)$ a \emph{(Pollicott-Ruelle) resonance on $\L^n$}. We write $\sigma^{\rm PR}_n$ for the set of all Pollicott-Ruelle resonances on $\L^n$ and we call $\sigma^{\rm PR}_n$ the \emph{resonance spectrum on $\L^n$}.  For  
 $\lambda\in \C$ we call 
 \[
  \Res_{n}(\lambda) :=\{u\in\mathcal D'(\M; \mc{L}^n)\,|\,
  ({\bf X}+\lambda)u=0, \;\tu{WF}(u)\subset E_u^*\}
 \]
 the space of \emph{Pollicott-Ruelle resonant states on $\L^n$} and for $k\in\nn^*$
 \[
  {\Res_{n}}(\lambda)^k:=\{u\in \mathcal D'(\M;\L^n)\,|\, 
  ({\bf X}+\lambda)^k u=0, \;\tu{WF}(u)\subset E_u^*\}
 \]
the space of \emph{generalized Pollicott-Ruelle resonant states on $\L^n$} of rank $k$. 
\end{definition}
\begin{rem}\label{rem:resprop}
\begin{enumerate}[leftmargin=*]
 \item  By Proposition~\ref{prop:merom_resolvent}, $\lambda\in \C$ is a Pollicott-Ruelle resonance on $\L^n$  
 iff $\Res_{n}(\lambda)\neq 0$.
 \item If $J$ is such that 
 ${\Res_{n}}(\lambda)^{J-1} \subsetneq {\Res_{n}}(\lambda)^J
 = {\Res_{n}}(\lambda)^{J+1}$, then the resolvent has a pole of order $J$.
In this case, there are distributional sections $u_1,\ldots, u_J$, $u_k\in {\Res_{n}}(\lambda)^{k}\setminus\{0\}$, 
such that $u_k=({\bf X}+\lambda)u_{k+1}$. We then say that $\lambda$ lies in a 
\emph{Jordan block of size $J$}. 
\item For any $\lambda\in\C$ with $\Re(\lambda)>0$ we know that $R_{n}(\cdot)$ is
holomorphic in a neighborhood of $\lambda$ and conclude $\Res_{n}(\lambda)=\{0\}$.
\end{enumerate}
\end{rem}

\section{Semiclassical resolvent estimates}\label{sec:resolvestim}

As shown in Appendix \ref{sec:appendix},  the union $\bigcup_{n\in\Z}\sigma^{\rm PR}_n$ of all individual line bundle resonance spectra intersects the region $\{\Re \lambda>-1\}$ only in finitely many resonances. In   view of this  and \eqref{sumRn}, it is natural to ask whether the  resolvent $R(\lambda)=(X+\lambda)^{-1}$ of the frame flow, defined by \eqref{resolventff} as an operator $L^2(F{\bf M})\to L^2(F{\bf M})$ in the region $\{\Re \lambda>0\}$, continues meromorphically to $\{\Re \lambda>-1\}$ as an operator between appropriate Hilbert spaces. This is what we shall prove in this section. In the process, we will obtain resolvent estimates for the geodesic flow acting on the line bundles $\L^{n}$, $n\in \Z$, over $\mc{M}=S{\bf M}$, where ${\bf M}=\gam \H^3$.  That ${\bf M}$ has this particular form will be used only in Section \ref{sec:gap}, whereas Sections \ref{sec:semiclformulation} and  \ref{sec:semiclmeasures} are formulated in a more general context.  
 
\subsection{Definition and properties of the required Hilbert  spaces}\label{sec:hilbertconstr}
In the following we use the families $\{\mathcal H^{Nm}_{1/n}\}_{n\in \Z\setminus\{0\}}$, $N>0$, of anisotropic Sobolev spaces from Section \ref{sec:sobG} to define useful Hilbert spaces on $F{\bf M}$. First, we define for $\varphi\in\Cinft(F{\bf M})$ and $n\in \Z$ the $n$-th Fourier mode $\varphi_n\in C^\infty_n(F{\bf M}):= \Cinft(F{\bf M})\cap  \ker(R-in)$ by
\bq
\varphi_n(\Gamma g):=\int_{ {\rm SO}(2)}\varrho_n(k^{-1}_0)\varphi(\Gamma g k_0 )\d M(k_0),\quad\Gamma g\in \gam G=F{\bf M}.\label{eq:Fouriermode}
\eq
Here $dM$ is the Haar measure on $M={\rm SO}(2)$ fixed by our chosen inner product on $\g\supset\m=T_eM$.  
The map \eqref{eq:Fouriermode} induces the orthogonal projection $L^2(F{\bf M})\to  L^2(F{\bf M})\cap\ker(R-in)$ by continuous extension; in particular, one has $\varphi=\sum_{n\in \Z}\varphi_n$ in $L^2(F{\bf M})$. Moreover, due to the fact that $M$ is Abelian, one has $(R^k\varphi)_n=R^k\varphi_n=(in)^k \varphi_n$ for all $n\in \Z$, $k\in \N_0$, and using this it is not difficult to see that $\varphi=\sum_{n\in \Z}\varphi_n$ in $\Cinft(F{\bf M})$. 
Dually, we define for  $f\in \D'(F{\bf M})$ the $n$-th Fourier mode $f_n\in \D'(F{\bf M})\cap \ker(R-in)$ by
\bqn
f_n(\varphi):=f(\varphi_{-n}),\qquad \forall\; \varphi\in \Cinft(F{\bf M}),
\eqn
so that the convergence $\varphi=\sum_{n\in \Z}\varphi_n$ in $\D(F{\bf M})$ for each $\varphi \in\Cinft(F{\bf M})$ implies 
\bq
f=\sum_{n\in \Z}f_n \qquad \text{in }\D'(F{\bf M}).\label{eq:f_nconv}
\eq
We can consider $f_n$ naturally as a distributional section of the line bundle $\L^{n}$ over $\M=S{\bf M}$. Let $h(n)=1/|n|$ if $n\not=0$ and $h(0)=2$, and we define for each $N>0$ and $k\in \R$ the Hilbert space 
\[\begin{gathered}
\mathcal H^{Nm,k}(F{\bf M}):=\big\{f\in \D'(F{\bf M})\,|\,f_n\in \mathcal H^{Nm}_{h(n)}(\mc{M};\mc{L}^n)\;\forall\;n\in \Z,\; \sum_{n\in \Z}\cjg n\cjd^{2k}\norm{f_n}_{\mathcal H^{Nm}_{h(n)}}^2<\infty\big\},\\ 
\norm{f}_{\mathcal H^{Nm,k}(F{\bf M})}^2:=\sum_{n\in \Z}\cjg n\cjd^{2k}\norm{f_n}_{\mathcal H^{Nm}_{h(n)}}^2,
\end{gathered}\]
with $\mathcal H^{Nm}_{h(n)}(\mc{M};\mc{L}^n)$ as in Section \ref{sec:sobG}. Note that $\mathcal H^{Nm,k}(F{\bf M})$ is a Hilbert space.  A basic observation is the following:
\begin{lemma}\label{lem:Cinftinclusion}
For all $N>0$ and $k\in\Z$, the following inclusion holds
$$
\Cinft(F{\bf M})\subset \mathcal H^{Nm,k}(F{\bf M}).
$$
Moreover, the inclusion map is continuous.
\end{lemma}
\begin{proof}Let $N>0$. 
Identifying $\Cinft(\M;\L^{n})=C^\infty_n(F{\bf M})$ for $n\in \Z$ (and similarly for $L^2$ and distributional sections), recall from Section \ref{sec:sobG} the definition of the anisotropic Sobolev spaces $\mathcal H^{Nm}_{h(n)}$ using the pseudodifferential operators $A^{N}_{h(n),n}$. For the case $n\not=0,$ we use the semiclassical calculus from Section \ref{sec:calculus} with $D:=\{1/n\,|\,n\in \nn^*\}\cup \{2\}$ and the two tensor power maps $n(h):=\pm 1/h$, $h\in D\setminus \{2\}$, while $n(2):=0$. Then $A^N_{h,n(h)}\in \Psi_{h}^{Nm}(\M;\L^{n(h)})\subset\Psi_{h}^{N+\delta}(\M;\L^{n(h)})$, where $\delta >0$ is arbitrarily small. The Laplacian $\Delta:\Cinft(F{\bf M})\to  \Cinft(F{\bf M})$ of the Riemannian metric on $F{\bf M}=\Gamma\backslash G$ induced by \eqref{laplacien} commutes with $R$ and thus  induces for each $n\in \Z$ an operator $\Delta_n:C^\infty_n(F{\bf M})\to  C^\infty_n(F{\bf M})$. Recalling the identification $\Cinft(\M;\L^{n})=C^\infty_n(F{\bf M})$, we note that the connection $\nabla^n$ on $\L^n$ induced by the connection $\nabla$ on $\L$ defined in Section \ref{sec:covariant} has the property that $\Delta_{n}=(\nabla^{n})^*\nabla^{n}$. Putting $n=n(h)$, the operator $\Delta_{n(h)}$ is in $\Psi_{h}^2(\M;\mc{L}^{n(h)})$. Given $f\in \Cinft(F{\bf M})$, $n\in \Z$, and $k\in \N_0$, one has
\begin{align}\begin{split}
\norm{R^k(I+\Delta)^{(N+1)/2}f}^2_{L^2(F{\bf M})}
&=\sum_{n\in \Z}n^{2k}\norm{(I+\Delta_n)^{(N+1)/2}f_n}^2_{L^2(\M;\L^{n})},\label{eq:sumf_n}\end{split}
\end{align}
as one easily checks using the facts that $M$ is abelian and that $\Delta$ commutes with the $M$-action on $F{\bf M}$. 
 Each $f_n$ is in $\mathcal H^{Nm}_{h(n)}$ because $\Cinft(\M;\L^{n})\subset\mathcal H^{Nm}_{h(n)}$ and we have for $h\in D$ with $f_h:=f_{n(h)}$ the estimate
 \begin{align*}
 & \norm{f_h}_{\mathcal H^{Nm}_{h}} \equiv \norm{A_{h,n(h)}^N f_h}_{L^2(\M;\L^{n(h)})}\\
  &=\norm{A_{h,n(h)}^N(1+h^2\Delta_{n(h)})^{-(N+1)/2}(1+h^2\Delta_{n(h)})^{(N+1)/2} f_h}_{L^2(\M;\L^{n(h)})}\\
  &\leq \norm{A_{h,n(h)}^N(1+h^2\Delta_{n(h)})^{-(N+1)/2}}_{L^2(\M;\L^{n(h)})\to L^2(\M;\L^{n(h)})}\norm{(1+h^2\Delta_{n(h)})^{(N+1)/2}f_h}_{L^2(\M;\L^{n(h)})} \\
 &\leq C_N \norm{(1+h^2\Delta_{n(h)})^{(N+1)/2}f_h}_{L^2(\M;\L^{n(h)})}\leq C_Nh^{N+1}\norm{(-R^2+\Delta_{n(h)})^{(N+1)/2}f_h}_{L^2(\M;\L^{n(h)})}
\end{align*}
for some $C_N>0$, since $A_{h,n(h)}^N(1+h^2\Delta_{n(h)})^{-(N+1)/2}\in \Psi_h^0(\M;\L^{n(h)})$ is a bounded operator on $L^2(\M;\L^{n(h)})$ with operator norm uniformly bounded  in $h$, see \eqref{eq:normbound}. Combining this with \eqref{eq:sumf_n}, we get  for $k\in \Z$
\[
\sum_{n\in \Z}\cjg n\cjd^{2k}\norm{f_n}_{\mathcal H^{Nm}_{h(n)}}^2
\leq C^2_N \norm{(1-R^2)^k (-R^2+\Delta)^{(N+1)/2}f}^2_{L^2(F{\bf M})}<\infty,
\]
concluding the proof of the inclusion. Continuity follows by estimating the above $L^2$-norm by the supremum norm, using that $F{\bf M}$ is compact.
\end{proof}

The main result of this section is:
\begin{theo}\label{prop:gap} 
Let ${\bf M}=\gam \H^3$ be a compact hyperbolic manifold and $F{\bf M}=\gam G$ its frame bundle, where $G=\mathrm{PSO}(1,3)$.  Then the frame flow resolvent $(X + \lambda)^{-1}$, which for $\Re \lambda>0$ is a holomorphic family of bounded operators $L^2(F{\bf M})\to L^2(F{\bf M})$ defined by  \eqref{resolventff}, extends  for each $N\geq 1$ to the region
$$
\{\Re\lambda >-1\}\subset \C
$$
as a meromorphic family of bounded  operators 
\[R(\lambda):=(-X - \lambda)^{-1}:
\mathcal H^{Nm,1}(F{\bf M})\to \mathcal H^{Nm,0}(F{\bf M}),\]
and the only poles of $R(\lambda)$ in that region are given by the real numbers $\lambda_j:=  \sqrt{1-\nu_j}-1$, $0\leq j\leq J$, where $\nu_0=0,\nu_1,\ldots,\nu_J$ are the eigenvalues of the Laplace-Beltrami operator $\Delta$ on ${\bf M}$  in the interval $[0,1)$. Moreover, for every $\delta,r>0$ there is a constant $C_{\delta,r}>0$ such that for  
$1<|{\rm Im}(\lambda)|$ and $-1+\delta<\Re \lambda <r$, one has the following estimate:
\bq
\norm{R(\lambda)}_{{\mc H}^{Nm,1}\to \mc{H}^{Nm,0}}\leq C_{\delta,r}\eklm{\lambda}^{2N+1}\label{eq:resolvestim}.
\eq 
\end{theo} 
Here $\eklm{\lambda}:=\sqrt{1+|\lambda|^2}$. Theorem \ref{th1intro} is obtained from Theorem \ref{prop:gap}  by taking $N=1$ and using the simplified notation $\mathcal H^{1,1}:=\mathcal H^{m,1}(F{\bf M})$, $\mathcal H^{1,0}:=\mathcal H^{m,0}(F{\bf M})$. 

In view of Proposition \ref{prop:merom_resolvent}, the proof of Theorem \ref{prop:gap} reduces to combining Corollary \ref{cor:individualgap} (for the location of the poles) with the following Proposition (for the resolvent bounds):
\begin{prop}\label{lem:key}
Let $c_0\in (0,1)$ and $c_1>0$. There is a $C_{c_0,c_1}>0$ such that if $N>c_0$
\bq
\norm{R_{0}(\lambda)}_{\mathcal H^{Nm}_{h(0)}\to \mathcal H^{Nm}_{h(0)}}\leq C_{c_0,c_1}\eklm{\lambda}^{2N+1}, \quad \forall\; \lambda \in \C;\; \Re \lambda\in [-c_0,c_1],\; |\Im \lambda|\geq 1 \label{eq:keybound0}
\eq
and for all $n\in \nn^*$ 
\bq
\norm{R_{\pm n}(\lambda)}_{\mathcal H^{Nm}_{h(n)}\to \mathcal H^{Nm}_{h(n)}}\leq C_{c_0,c_1}\cjg n\cjd\eklm{\lambda}^{2N+1}, \quad \forall\; \lambda \in \C;\; \Re \lambda\in [-c_0,c_1].\label{eq:keybound}
\eq
\end{prop} We note that here $N$ can be chosen to be equal to $1$, see Lemma \ref{FaureSjostrand}. 

The proof of Proposition \ref{lem:key} will be given on page  \pageref{proof:keylem}. The following sections are devoted to its  preparation.

\subsection{Semiclassical formulation of the problem}\label{sec:semiclformulation}

In this section, let ${\bf M}$ be an arbitrary compact negatively curved Riemannian manifold without boundary. 
We let $X$ be the generating vector field of the geodesic flow $\varphi_t$ on $\mc{M}=S{\bf M}$, and $\mc{L}\to \mc{M}$ a complex line bundle equipped with a metric $|\cdot|$ and a Hermitian connection $\nabla$ (i.e., preserving the metric on $\mc{L}$). We consider the flow acting on sections of powers $\mc{L}^n$ of the line bundle (here $n\in \nn$) by considering the operator 
\[ \Xbf u := \nabla^{n}_{X}u , \quad u \in C^\infty(\mc{M}; \mc{L}^n),\]
where we denote the induced connection on $\mc{L}^n$ by $\nabla^{n}$. 
We will assume that the curvature $\Omega$ of $(\mc{L},\nabla)$ is preserved by the flow of $X$, that is 
\begin{equation}\label{eq:OmegaX}
\iota_X\Omega =0.
\end{equation}
Moreover, we consider the family $\{\omega_\rho\}_{\rho\in \R}$ of symplectic forms on the cotangent bundle $T^\ast \M$ defined by \eqref{eq:omega_rho}. 

\subsubsection{Reduction to a semiclassical problem}\label{sec:semicl1}
To analyze the operator $\Xbf$ acting on $C^\infty(\mc{M}; \mc{L}^n)$ with $n\geq 0$, it is first convenient to view it as a semiclassical family.  
To apply the calculus introduced in Section \ref{sec:calculus} and the anisotropic Sobolev spaces from Section \ref{sec:sobG}, we define the set $D:=\{1/n:n\in \nn^*\}\cup\{2\}$ and the tensor power map $n:D\to \N_0$ given by $n(h):=1/h$ for $h\in D\cap (0,1]$ and $n(2):=0$.  As demonstrated in Example \ref{ex:Ahcovderiv}, the operator 
\[ P_{h}: C^\infty(\mc{M}; \mc{L}^{n(h)})\to  C^\infty(\mc{M}; \mc{L}^{n(h)}), \quad P_{h}u:=h\nabla^{n(h)}_{X} u,\]
is then in $\Psi_h^{1}(\mc{M};\L^{n(h)})$ and its modified principal symbol is represented by the function
\bq 
(x,\xi)\mapsto i\xi(X(x))=:i p_0(x,\xi),\qquad (x,\xi)\in T^\ast \M.\label{eq:princsymb}
\eq
We will consider the operator $$P_{h}(\la):=P_{h}+\la,$$ where $\Re\la\in [-c_0h, c_1 h]$ for some $c_0\in (0,1)$ and $c_1>0$ describing a spectral strip. Fix $N>\frac{c_0}{\gamma_\mathrm{min}}$, where $\gamma_\mathrm{min}>0$ is the minimal Anosov expansion rate of $X$ on $\M$ defined by
\bq
\gamma_\mathrm{min}:=\min(\gamma_+,\gamma_-),\qquad \gamma_\pm:= \inf_{v\in E_\pm,\norm{v}=1}\liminf_{t\to +\infty}\frac{1}{t}\log \norm{d\varphi_{\mp t} v},\label{eq:minexpansionrate}
\eq
where $E_+,E_-\subset T\M$ are the stable $(+)$ and unstable $(-)$ subbundles in the Anosov decomposition of $T\M$. For example, if $\M$ has constant curvature $-1$, then $\gamma_\mathrm{min}=1$. 

In order to prove Proposition \ref{lem:key} by contradiction, we show in the following that its negation implies certain convergence statements which are special cases of a somewhat more abstract semiclassical statement that we will formulate in Assumption \ref{ass:contradition}. 

If \eqref{eq:keybound} does not hold, then $P_h(\la)=h{\bf X}+\la$ does not satisfy the statement 
\begin{multline} \exists\,C>0: \forall \,h\in D, \;\forall \,\la \; {\rm s.t.} \;\Re \la\in [-c_0h,c_1h], \\ \|P_h(\la)^{-1}\|_{\mc{H}_h^{Nm}(\M;\L^{n(h)})\to \mc{H}_h^{Nm}(\M;\L^{n(h)})}\leq Ch^{-3-2N}(h+|\la|)^{2N+1},\label{eq:contradiction}
\end{multline}
where $\mc{H}_h^{Nm}(\M;\L^{n(h)}):=\Op_{h,n(h)}(e^{NG})^{-1}(L^2(\M;\L^{n(h)}))=\mc{H}_{h}^{Nm}$ as in Section \ref{sec:sobG}. 
This means that there are sequences $(h_j)_j\subset D$, $(u_j)_j\subset \mc{H}_{h_j}^{Nm}(\M;\L_{h_j})$, $(\la_j)_{j}\subset \C$, satisfying $\|u_j\|_{\mc{H}_{h_j}^{Nm}(\M;\L^{n(h_j)})}=1$, $\Re \lambda_j\in [-c_0h_j,c_1h_j]$, 
 such that 
\begin{equation}\label{Phu_j0}
 \|P_{h_j}(\la_j)u_j\|_{\mc{H}_{h_j}^{Nm}(\M;\L^{n(h_j)})}=o(h_j^{2N+3}(h_j+|\la_j|)^{-2N-1}) \qquad \text{as }j\to +\infty.
\end{equation}
We now distinguish two cases: either one has $h_j\to 0$, which is the ``large $n$'' case, or one has $h_j\not\to 0$, which we call the ``bounded $n$'' case.

\textbf{The case of large $n$.} To simplify the notation, we will just write $h$ for $h_j$ in what follows (i.e., we replace $D$ by $\{h_j\}\subset D$), keeping in mind that $h$ is a sequence going to $0$, and we shall write $\la=\la(h)$ instead of $\la_j$, considering it as a function depending on $h$. We will then write $u_h$ instead of $u_j$, so that \eqref{Phu_j0} reads
\begin{equation}\label{Phu_j}
 \|P_{h}(\la)u_h\|_{\mc{H}_{h}^{Nm}(\M;\L^{n(h)})}=o(h^{2N+3}(h+|\la|)^{-2N-1}) \qquad \text{as }h\to 0.
\end{equation}
Up to extracting a subsequence, we can assume that the convergence $h\to 0$ is strictly monotonously decreasing, that  
\begin{equation}\label{limitlaj}
\Im\la(h) \to \Upsilon,\quad \quad \frac{\Re\la(h)}{h}\to \nu\qquad\qquad \text{as }h\to 0
\end{equation}  
for some $\Upsilon \in \R\cup \{-\infty,\infty\}$, $\nu\in [-c_0,c_1]$. If $\Upsilon\in \{-\infty,\infty\}$, then we can assume that the convergence $|\Im \la(h)|\to+\infty$ is strictly monotone. 

Making all these assumptions, the fact that the limit $\Upsilon$ can be infinite is technically inconvenient.  Conceptually, it corresponds to an escape of energy at infinity for the Hamiltonian $p_0$. To prevent this, we will perform in the following a rescaling of our semiclassical parameter involving $\Im \lambda(h)$. As we shall see, this will allow us to obtain a formally completely analogous situation as in \eqref{Phu_j} and \eqref{limitlaj} but with a new limit of the imaginary part of the considered spectral parameter that is always finite, so that no energy escapes at infinity anymore.

Let us define the new semiclassical parameter
\bq
h'=h'(h):=\begin{dcases}\frac{h}{1+|\Upsilon|},\qquad &\text{if } \Upsilon\in \R,\\
\frac{h}{1+|\Im\la(h)|}, &\text{if } \Upsilon\in \{-\infty,\infty\}.\end{dcases}
\label{eq:hprime}\eq
The strict monotone convergence of $h$ to $0$ and the strict monotone growth of $|\Im \la(h)|$ in the unbounded case imply that the map $h\mapsto h'=h'(h)$  defined by \eqref{eq:hprime} is injective. This allows us to associate conversely to each $h'$ obtained in \eqref{eq:hprime} a unique corresponding $h\in D$, denoted $h(h')$, which has the property that $h(h')\to 0$ as $h'\to 0$. We can then switch to a semiclassical calculus with the asymptotic  parameter $h'$  in the sense of Section \ref{sec:calculus} by defining the new domain $D':=\{h'(h)\, |\, h\in D\}$ and the new tensor power map
\bq
n'(h'):=n(h(h'))=\frac{1}{h(h')},\qquad h'\in D'.\label{eq:newtensorpower}
\eq
In what follows, it is important to distinguish between the notions of \emph{semiclassical parameter} and  \emph{semiclassical calculus} and to keep in mind that the notation $h,h'$ actually has two formally different meanings. Let us explain this: \begin{enumerate}
\item $h$ and $h'$ can be regarded as semiclassical parameters in the sense of Section \ref{sec:calculus}. In this point of view,  $h$ and $h'$ denote the elements in the domains $D$ and $D'$ of the two tensor power maps $n:D\owns h\mapsto n(h)\in \N_0$ and $n':D'\owns h'\mapsto n'(h')\in \N_0$. Those maps are formally the relevant objects, each defining an associated calculus as in Section \ref{sec:calculus}. Thus, one should think of an $n$-calculus and an $n'$-calculus rather than an $h$- and an $h'$-calculus. 
\item $h$ and $h'$ can be regarded as functions of one another, i.e., $h$ denotes the function $D'\to D$ given by $h'\mapsto h(h')$ and $h'$ denotes the function $D\to D'$ given by $h\mapsto h'(h)$. These functions are bijective and each others' inverses. 
\end{enumerate}
The two meanings are closely related by the equation $n=n'\circ h'$ of maps $D\to \N_0$ and the equation $n'=n\circ h$ of maps $D'\to \N_0$. It is conceptually important in the following to distinguish between objects of the $n$-calculus that become parametrized by $h'$ by putting $h=h(h')$ and objects of the $n'$-calculus which are intrinsically parametrized by $h'$. Formally, we achieve this by introducing the following notational generalization:
\begin{rem}[Generalized quantization maps]\label{rem:generalizedop} 
The quantization map ${\rm Op}_{h,n(h)}$ introduced in \eqref{expquantization} depends on $h$ in two ways: via $n(h)$ and via occurrences of $h$ which are not in the argument of the map $n$. We now define the generalized quantization map 
 ${\rm Op}_{h(h'),n'(h')}$ as a map from $S_{h',\delta}^{m}(\mc{M})$ into the continuous linear maps $C^\infty(\mc{M};\mathcal L^{n'(h')})\to C^\infty(\mc{M};\mathcal L^{n'(h')})$ by replacing  in  \eqref{expquantization} every occurrence of $h$ by $h(h')$ and using that $n(h(h'))=n'(h')$. Similarly, we define ${\rm Op}_{h'(h),n(h)}$ as a map from $S_{h,\delta}^{m}(\mc{M})$ into the continuous linear maps $C^\infty(\mc{M};\mathcal L^{n(h)})\to C^\infty(\mc{M};\mathcal L^{n(h)})$ by replacing in the $n'$-calculus version of \eqref{expquantization} every occurrence of $h'$ by $h'(h)$ and using that $n'(h'(h))=n(h)$. We emphasize that we will not need to generalize the calculus using these quantizations, i.e., we will not need to generalize $\Psi^{m}_{h,\delta}(\M,\L^{n(h)})$, $\Psi^{m}_{h',\delta}(\M,\L^{n'(h')})$, $\sigma_{h,n(h)}$, and $\sigma_{h',n'(h')}$ in a similar way. The properties of the maps ${\rm Op}_{h(h'),n'(h')}$ and ${\rm Op}_{h'(h),n(h)}$ that we will use follow directly from statements which are formulated either entirely within the $n$- or entirely within the $n'$-calculus. 
 
 For example, the Hilbert space 
 \bq
 \mc{H}_{h(h')}^{Nm}(\M;\L^{n'(h')}):=\Op_{h(h'),n'(h')}(e^{NG})^{-1}(L^2(\M;\L^{n'(h')}))\label{eq:newhilbert}
 \eq
is well-defined, i.e., $\Op_{h(h'),n'(h')}(e^{NG})$ is invertible for all $h'\in D'$, as follows immediately from putting $h=h(h')$ in \eqref{eq:Ahinverse} and \eqref{eq:Rhnorm}.
\end{rem} 

Note that the terms $h'n(h')$ and $\frac{h'(h)}{h}$ are convergent as $h'\to 0$ and $h\to 0$, respectively:
\bq
\lim_{h'\to 0}h'n(h')=\lim_{h\to 0}\frac{h'(h)}{h}=\begin{dcases}\frac{1}{1+|\Upsilon|},\qquad & \text{if }\Upsilon\in \R,\\
0, & \text{if }\Upsilon\in \{-\infty,\infty\}.\end{dcases}\label{eq:convergencehn}
\eq
 We define for $h'\in D'$ the parameter $\lambda'=\lambda'(h')$ and $P'_{h'}(\lambda')\in \Psi_{h'}^{1}(\mc{M};\mc{L}^{n'(h')})$ by
\[
P'_{h'}(\lambda'):=h'X_{n'(h')}+\lambda',\qquad \lambda'=\lambda'(h'):= \frac{h'}{h(h')}\lambda(h(h')),
\]
which have the advantage that
\begin{equation*}
\Im\la'(h') \to \Lambda,\quad \quad \frac{\Re\la'(h')}{h'}\to \nu\qquad\qquad \text{as }h'\to 0,
\end{equation*}
where the new limit
\[
\Lambda:=\begin{dcases}\frac{\Upsilon}{1+|\Upsilon|},\qquad & \text{if }\Upsilon\in \R,\\
\pm 1, &\text{if }\Upsilon=\pm \infty,\end{dcases}
\]
is now always finite, and $\nu\in [-c_0,c_1]$ is as before. In particular $\la'$ is bounded. 

One has the equation
\bq
P'_{h'}(\lambda')=\frac{h'}{h(h')}P_{h(h')}(\la)\label{eq:exampleeq}
\eq
of operators in $\Psi_{h'}^{1}(\mc{M};\mc{L}^{n'(h')})$.  Using \eqref{eq:exampleeq}, and recalling again that $n(h(h'))=n'(h')$,  the relation \eqref{Phu_j} reads in terms of $h'$ as follows:
\begin{equation}
 \norm{P'_{h'}(\lambda')u_{h(h')}}_{\mc{H}_{h(h')}^{Nm}(\M;\L^{n'(h')})}=o\Big(h'h(h')(1+{h'}^{-1}|\la'|)^{-2N-1}\Big)\qquad \text{as }h'\to 0.\label{eq:temp1}
\end{equation} 
We must now take into account that the Hilbert space $\mc{H}_{h(h')}^{Nm}(\M;\L^{n'(h')})$ defined in  \eqref{eq:newhilbert}
 does not agree with the semiclassical Sobolev space of the  $n'$-calculus given by
$$\mc{H}_{h'}^{Nm}(\M;\L^{n'(h')})=\Op_{h',n'(h')}(e^{NG})^{-1}(L^2(\M;\L^{n'(h')})).$$
The two spaces agree as sets but their norms differ. More precisely, one has
\begin{align}\begin{split}
\norm{\cdot}_{\mc{H}_{h(h')}^{Nm}(\M;\L^{n'(h')})}&\leq C_L(h')  \norm{\cdot}_{\mc{H}_{h'}^{Nm}(\M;\L^{n'(h')})},\\ \norm{\cdot}_{\mc{H}_{h'}^{Nm}(\M;\L^{n'(h')})}&\leq C_R(h')  \norm{\cdot}_{\mc{H}_{h(h')}^{Nm}(\M;\L^{n'(h')})},\label{eq:normcompare}\end{split}
\end{align}
with $h'$-dependent bounds given by
\begin{align}\begin{split}
C_R(h')&:=\norm{\Op_{h',n'(h')}(e^{NG})\Op_{h(h'),n'(h')}(e^{NG})^{-1}}_{L^2(\M;\L^{n'(h')})\to L^2(\M;\L^{n'(h')})},\\
C_L(h')&:=\norm{\Op_{h(h'),n'(h')}(e^{NG})\Op_{h',n'(h')}(e^{NG})^{-1}}_{L^2(\M;\L^{n'(h')})\to L^2(\M;\L^{n'(h')})}.\label{eq:boundsc1c2}\end{split}
\end{align}
We then define for $h'\in D'$ the normalized distributions 
\bq
u'_{h'}:=\frac{u_{h(h')}}{\norm{u_{h(h')}}_{\mc{H}_{h'}^{Nm}(\M;\L^{n'(h')})}}\in \mc{H}_{h'}^{Nm}(\M;\L^{n'(h')}).\label{eq:newu_h}
\eq
Replacing $u_{h(h')}$ by $u'_{h'}$ and $\norm{\cdot}_{\mc{H}_{h(h')}^{Nm}(\M;\L^{n'(h')})}$ by $\norm{\cdot}_{\mc{H}_{h'}^{Nm}(\M;\L^{n'(h')})}$ in \eqref{eq:temp1},  we get the estimate
\begin{equation}
 \norm{P'_{h'}(\lambda')u_{h'}}_{\mc{H}_{h'}^{Nm}(\M;\L^{n'(h')})}=o\Big(h'h(h')(1+{h'}^{-1}|\la'|)^{-2N-1}C_L(h')C_R(h')\Big) \qquad \text{as }h'\to 0\label{eq:temp2}
\end{equation}
in which the left hand side is now expressed using the desired semiclassical Sobolev spaces. 
It remains to bound the constants $C_L(h')$, $C_R(h')$ in terms of $h'$. 
\begin{lemma} One has the following bounds in the limit $h'\to 0$:
\bq
 C_L(h')=\mc{O}\big(h(h')^{N}{h'}^{-N}\big),\qquad C_R(h')=\mc{O}\big({h(h')}^{N}{h'}^{-N}\big).\label{eq:boundsClCR}
 \eq
\end{lemma}
\begin{proof}
First, let us define the two bounded operators $L^2(\M;\L^{n'(h')})\to L^2(\M;\L^{n'(h')})$
 $$
 \tilde A_{h'}:=\Op_{h(h'),n'(h')}(e^{NG})\Op_{h',n'(h')}(e^{NG})^{-1},\quad \tilde B_{h'}:=\Op_{h',n'(h')}(e^{NG})\Op_{h(h'),n'(h')}(e^{NG})^{-1}.$$
As a first step, we apply \eqref{eq:Ahinverse} and \eqref{eq:Rhnorm} in the $n$-calculus and the $n'$-calculus to get
\begin{align*}
\Op_{h',n'(h')}(e^{NG})^{-1}&=\Op_{h',n'(h')}(e^{-NG})(I+R'(h'))^{-1},\\
 \Op_{h(h'),n'(h')}(e^{NG})^{-1}&=\Op_{h(h'),n'(h')}(e^{-NG})(I+R(h(h')))^{-1},
\end{align*}
where both operators $R'(h')$ and $R(h(h'))$ are bounded on $L^2(\M;\L^{n'(h')})$ with operator norms $\leq 1/2$ for all $h'\in D'$. It therefore suffices to bound the $L^2$-norms of the operators 
  $$
 A_{h'}:=\Op_{h(h'),n'(h')}(e^{NG})\Op_{h',n'(h')}(e^{-NG}),\quad B_{h'}:=\Op_{h',n'(h')}(e^{NG})\Op_{h(h'),n'(h')}(e^{-NG}).$$
A priori, it is unclear whether these operators belong to the $n'$-calculus, or whether the operators  $A_h:=A_{h'(h)}$ and  $B_h:=B_{h'(h)}$ given by
$$
 A_h=\Op_{h,n(h)}(e^{NG})\Op_{h'(h),n(h)}(e^{-NG}),\quad B_h=\Op_{h'(h),n(h)}(e^{NG})\Op_{h,n(h)}(e^{-NG})$$ belong to the $n$-calculus. However, thanks to the boundedness of $h'(h)/h$ as $h\to 0$ implied by \eqref{eq:convergencehn}, we can show that $A_h$ and $B_h$ are indeed semiclassical operators in $\Psi_{h}^{0}(\mc{M};\mc{L}^{n(h)})$  with modified principal symbols $\sigma_{h,n(h)}(A_h)$, $\sigma_{h,n(h)}(B_h)$ represented by the functions
\[a_{h}(x,\xi):=e^{N(G(x,\xi)-G(x,\frac{h'(h)}{h}\xi))},\qquad b_{h}(x,\xi):=e^{N(G(x,\frac{h'(h)}{h}\xi)-G(x,\xi))}.\]
Since the composition of two operators in the $n$-calculus belongs again to the $n$-calculus and in view of the composition formula \eqref{eq:compositionformula}, it suffices to show that $\Op_{h'(h),n(h)}(e^{\pm NG})\in \Psi^{\pm Nm}_{h}(\mc{M};\mc{L}^{n(h)})$ with modified principal symbol represented by $e^{\pm N\tilde G_h(x,\xi)}$, where we put $\tilde G_h(x,\xi):=G(x,\frac{h'(h)}{h}\xi)$. Here we use that $e^{\pm N\tilde G_h}\in S_{h,\delta}^{\pm Nm}(\mc{M})$ since $h'(h)/h$ is bounded as $h\to 0$.   Recalling the quantization expression \eqref{expquantization} and the definition of $\Op_{h'(h),n(h)}$ from Remark \ref{rem:generalizedop}, one obtains in a local trivialization with the change of variables $\xi\to \frac{h'(h)}{h}\xi$
\[\begin{split}
{\bf s}^{-n(h)}\Op_{h'(h),n(h)}(e^{\pm N G})(f{\bf s}^{n(h)})(x)&= \frac{1}{2\pi h'(h)^d}\int_{\rr^d}
e^{\frac{i(x-x')\cdot \xi}{h'(h)}}e^{N G(x,\xi-h'(h)n(h)\beta(x))}f(x')dx' d\xi\\
&= \frac{1}{2\pi h'(h)^d}\int_{\rr^d}
e^{\frac{i(x-x')\cdot \xi}{h'(h)}}e^{N \tilde G(x,\frac{h}{h'(h)}\xi- hn(h)\beta(x))}f(x')dx' d\xi\\
&=\frac{1}{2\pi h^d}\int_{\rr^d}
e^{\frac{i(x-x')\cdot \xi}{h}}e^{N\tilde G(x,\xi- hn(h)\beta(x))}f(x')dx' d\xi\\
&={\bf s}^{-n(h)}\Op_{h,n(h)}(e^{\pm N \tilde G})(f{\bf s}^{n(h)})(x).
\end{split}\]
Recall from Section \ref{sec:sobG} that $G=m\log(F)$ with $m,F\in \Cinft(T^\ast \M)$ having the property that there is a number $r>0$ such that $m(x,\xi)$ and $F(x,\xi)$ are  positively homogeneous of degrees $0$, $1$, respectively for $|\xi|\geq r$.  This gives us for $(x,\xi)\in T^\ast \M$ as $h\to 0$
\bqn
\sup_{(x,\xi)\in T^*\mc{M}}|a_{h}(x,\xi)|=\mc{O}({h}^{N}{h'(h)}^{-N}),\qquad \sup_{(x,\xi)\in T^*\mc{M}}|b_{h}(x,\xi)|=\mc{O}({h}^{N}{h'(h)}^{-N})
\eqn
and yields with \eqref{eq:normbound} that $A_h$ and $B_h$ have operator norms bounded by $\mc{O}({h}^{N}{h'(h)}^{-N})$  as operators $L^2(\M;\L^{n(h)})\to L^2(\M;\L^{n(h)})$ in the limit $h\to 0$. Plugging in $h=h(h')$, it follows that $A_{h'}$ and $B_{h'}$ have operator norms bounded by $\mc{O}(h(h')^{N}h'^{-N})$  as operators $L^2(\M;\L^{n'(h')})\to L^2(\M;\L^{n'(h')})$ in the limit $h'\to 0$, finishing the proof.
\end{proof}
We can now finish the transition from the $n$-calculus to the $n$-calculus: by the preceding lemma, \eqref{eq:temp2} implies
\begin{equation*}
 \norm{P'_{h'}(\lambda')u_{h'}}_{\mc{H}_{h'}^{Nm}(\M;\L_{h'})}=o({h'}^{1-2N}(h(h'))^{1+2N}(1+{h'}^{-1}|\la'|)^{-2N-1})=o({h'}^{2}) \qquad \text{as }h'\to 0,
\end{equation*}
where we used that $h(h')\sim h'$ if $|\Upsilon|\not=\infty$, and ${h'}^{-1}|\la'|> {h'}^{-1}/2$ for small $h'>0$ if $|\Upsilon|=\infty$.

\textbf{The case of bounded $n$.} 
 When $h_j\not\to 0$ in \eqref{Phu_j0}, we pass to a constant subsequence  with some value $1/n_0$, $n_0\in \nn^*$. We then arrive at the situation that there is a sequence of unit norm elements $u_j\in \mc{H}_{n_0}^{Nm}(\M;\L^{n_0})$, a sequence of complex numbers $\la_j$, and limits  $\nu\in [-c_0/n_0,c_1/n_0]$, $\Upsilon \in \R\cup \{-\infty,\infty\}$ such that $\Re \lambda_j \to \nu$, $\Im \lambda_j \to \Upsilon$,   and
\bq \label{eq:cond2}
 \|({\bf X}+\la_j)u_j\|_{\mc{H}_{h(n_0)}^{Nm}(\M;\L^{n_0})}=o(\cjg \la_j\cjd^{-2N-1}) \qquad \text{as }j\to +\infty.
\eq
Note that the negation of \eqref{eq:keybound0} implies the $n_0=0$ version of \eqref{eq:cond2} 
with the assumptions $\nu\in [-c_0,c_1]$ and $|\Upsilon|\geq 1$. We therefore generalize our situation to include also this case.
If $|\Upsilon|<\infty$, then \eqref{eq:cond2} is equivalent to 
\bqn
\|({\bf X}+\nu+i\Upsilon)u_j\|_{\mc{H}_{h(n_0)}^{Nm}(\M;\L^{n_0})}\longrightarrow 0 \qquad \text{as }j\to +\infty,
\eqn
which implies that $\nu+i\Upsilon\in \sigma^{\rm PR}_{n_0}$. Then \eqref{eq:resonancesunion} gives a contradiction since $c_0<1$. Thus  $\Upsilon\in \{-\infty,\infty\}$ and we can proceed similarly as in \eqref{eq:hprime} by  putting  
$h'_j:=1/(1+|\Im\la_j|)$, where  $h'_j\to 0$ as $j\to \infty$. 
For simplicity of notation we will remove the $j$ index and consider $h'\to 0$ to be a sequence of positive numbers and $\la'=\la'(h'):=h_j'\la_j$ as a family depending on $h'$ so that $\Im\la'(h')$ converges to a limit $\Lambda \in \{-1,1\}$, $(h')^{-1} \Re\la'(h')\to\nu$ as $h'\to 0$,  
and we have a family $u_{h'}\in \mc{H}_{h(n_0)}^{Nm}(\M;\L^{n_0})$ such that, if $P'_{h'}(\la'):=h'{\bf X}+\la'$, one has  
\begin{equation}\label{casenbounded}
\|P'_{h'}(\la')u_{h'}\|_{\mc{H}_{h(n_0)}^{Nm}(\M;\L^{n_0})}=o({h'}^{2N+2}) \qquad \text{as }h'\to 0.
\end{equation} 
As above, it is more convenient to work on $\mc{H}_{h'}^{Nm}(\mc{M};\mc{L}^{n_0})$, and we  
notice that $P'_{h'}(\la')\in \Psi_{h'}^1(\mc{M};\mc{L}^{n_0})$, which fits in the calculus of Section \ref{sec:calculus} by choosing the trivial constant tensor power function $h'\mapsto n(h')={n_0}$. The bound becomes, with the same argument as above, 
\[\|P'_{h'}(\la')u_{h'}\|_{\mc{H}_{h'}^{Nm}(\M;\L^{n_0})}=o({h'}^{2}).\]

\subsubsection{The fundamental assumption}
The upshot of Section \ref{sec:semicl1} is  that it suffices to consider the following situation:

\begin{ass}\label{ass:contradition} Let $\mc{L}$ be a Hermitian line bundle with connection $\nabla$ on $\mc{M}=S{\bf M}$. 
For some positive real numbers $c_0,c_1,N$ with\footnote{ Here $\gamma_\mathrm{min}>0$ has been defined in \eqref{eq:minexpansionrate}; one has $\gamma_\mathrm{min}=1$ if $\M=S{\bf M}$ with ${\bf M}=\gam \H^3$ hyperbolic.}  $N>\frac{c_0}{\gamma_\mathrm{min}}$, some domain $D\subset (0,2]$ with a tensor power map $D\owns h\mapsto n(h)$ as in Section \ref{sec:calculus} and some function $D\owns h\mapsto\lambda(h)\in \C$ such that for some $\Lambda \in [-1,1]$ and $\nu\in [-c_0,c_1]$ one has
\begin{equation}\label{eq:limitlajneq}
\Im\la(h) \to \Lambda, \qquad \frac{\Re\la(h)}{h}\to \nu,\qquad hn(h)\to 1-|\Lambda| \qquad \text{as }h\to 0,
\end{equation}
there is for each $h\in D$ a distributional section $u_h\in \mc{H}_{h}^{Nm}(\M;\L^{n(h)})$ of norm $1$ such that
\begin{equation}\label{Phu_jnewcond}
 \norm{P_{h}(\lambda)u_h}_{\mc{H}_{h}^{Nm}(\M;\L^{n(h)})}=o(h^{2}) \qquad \text{as }h\to 0,\qquad P_{h}(\lambda)=hX_{n(h)}+\lambda(h),
\end{equation}
where one has $\mc{H}_{h}^{Nm}(\M;\L^{n(h)})=\Op_{h}(e^{NG})^{-1}(L^2(\M;\L^{n(h)}))$ with $G$ as in Section \ref{sec:sobG} and $\Op_{h}$ as in \eqref{eq:Oph}. 
\end{ass}
\subsubsection{Principal symbol and Hamiltonian vector fields of $P_h(\lambda)$}
A function $p$ representing the modified principal symbol of  $-i P_h(\lambda)$ is 
\bq
p=p_0+{\rm Im}(\la),\label{eq:princsymplambd}
\eq
 where the $h$-independent function $p_0$ was introduced in  \eqref{eq:princsymb}. Since $dp=dp_0$, we see from \eqref{eq:defhamilt} that $p$ and $p_0$  have the same Hamiltonian vector fields, i.e., 
 $ 
 H^{\omega_\rho}_{p}=H^{\omega_\rho}_{p_0}$ for every $\rho\in \R$. By Lemma \ref{lem:Hp} and the assumption \eqref{eq:OmegaX}, we actually have 
 \begin{equation}\label{Hpomega=Hp} 
H_p^{\omega_\rho}=H_p^{\omega_0}=H_{p_0}^{\omega_0}, \quad\forall\;\rho\in \R.
\end{equation} 
In particular, the flow of $H^{\omega_\rho}_p$ agrees with the  flow $\Phi_t$ of $H^{\omega_0}_{p_0}$ for each $\rho\in \R$.
 
\subsection{Support and regularity of semiclassical measures}\label{sec:semiclmeasures}

In this section we consider the same setting as in Section \ref{sec:semiclformulation} and we use the notation from  Assumption \ref{ass:contradition}.   
By \cite[Theorem E.42]{dyatlovzworskibook}, up to replacing $D$ by a smaller domain (i.e., passing to a subsequence), there is a semiclassical measure $\mu\geq 0$ associated to $u_h$: for each $a\in C_c^\infty(T^\ast \M)$
\[ \cjg \Op_{h}(a)u_h, u_h\cjd_{\mc{H}_{h}^{Nm}} \to \int_{T^*\mc{M}}a\d\mu\quad \textrm{ as }h\to 0.\]
Here and in the following we write $\mc{H}_{h}^{Nm}:=\mc{H}_{h}^{Nm}(\mc{M};\mc{L}^{n(h)})$, $L^2:=L^2(\mc{M};\mc{L}^{n(h)})$. 

By \eqref{Phu_jnewcond} and the same argument as in \cite[Theorem E.43]{dyatlovzworskibook} we have 
\begin{lemma}\label{supportmu}
If Assumption \ref{ass:contradition} is fulfilled, the semiclassical measure $\mu$ satisfies $$\supp(\mu)\subset \{(x,\xi)\in T^*\mc{M}\,|\, 
\xi(X(x))=\Lambda\}.$$
\end{lemma} 
\begin{proof} For $A={\rm Op}_{h}(a) \in \Psi_h^{\rm comp}(\mc{M};\mc{L}^{n(h)})$ we have  
 as $h\to 0$
\[\begin{split} 
i\int_{T^*\mc{M}}(\xi(X)-\Lambda)a\, d\mu& = 
\lim_{h\to 0}\cjg AP_{h}(\la)u_h,u_h\cjd_{\mc{H}_{h}^{Nm}}
  =\lim_{h\to 0}\cjg P_{h}(\la)u_{h},{\rm Op}_{h}(a)^*u_h\cjd_{\mc{H}_{h}^{Nm}}\\
 |\cjg P_{h}(\la)u_{h},{\rm Op}_{h}(a)^*u_h\cjd_{\mc{H}_{h}^{Nm}}| & \leq \|P_{h}(\la)u_{h}\|_{\mc{H}_{h}^{Nm}}\|\Op_{h}(e^{NG}){\rm Op}_{h}(a)^*\Op_{h}(e^{NG})^{-1}\|_{L^2\to L^2}\|u_h\|_{\mc{H}_{h}^{Nm}}\\
\end{split} \]
that tends to $0$.
Here we used \eqref{Phu_j} and the fact that the middle term in the last inequality is bounded as $h\to 0$ by \eqref{eq:normbound} because $a$ is compactly supported. 
\end{proof}
We can also apply the argument of \cite[Theorem E.44]{dyatlovzworskibook}:
\begin{lemma}\label{flowinvmu} 
If Assumption \ref{ass:contradition} is fulfilled, one has 
\[ \forall\, a\in C_c^\infty(T^*\mc{M}),\quad \int_{T^*M}(H^{\omega_0}_{p_0}a-2\nu a)\, d\mu=0,\]
so that the pushforward of $\mu$ along the flow $\Phi_t$ of  $H^{\omega_0}_{p_0}$ fulfills
\bq
(\Phi_t)_\ast\mu=e^{2\nu t}\mu,\qquad t\in \R.\label{eq:mupshfwd}
\eq
\end{lemma} 
\begin{proof} Let $f_h:=P_h(\la)u_h\in \mc{H}_{h}^{Nm}(\mc{M};\mc{L}^{n(h)})$. Then for $A\in \Psi_{h}^{\rm comp}(\mc{M};\mc{L}^{n(h)})$ with $A^*=A={\rm Op}_{h}(a)$ for some $a\in C_c^\infty(T^*\mc{M})$, we obtain using the relation $\Xbf^\ast = -\Xbf$ in $L^2(\mc{M};\mc{L}^{n(h)})$
\[h^{-1}{\rm Re}(\cjg Af_h,u_h\cjd_{L^2})=(2h)^{-1}\cjg [A,P_{h}]u_h,u_h\cjd_{L^2}+{\rm Re}(h^{-1}\la)\cjg Au_h,u_h\cjd_{L^2}\]
where we notice that the $L^2$-pairing makes sense due to the fact that ${\rm Op}_h(a):\mc{D}'(\mc{M};\mc{L}^{n(h)})\to C^\infty(\mc{M};\mc{L}^{n(h)})$. Note that 
\[\cjg Af_h,u_h\cjd_{L^2}=\cjg  {\rm Op}_h(e^{NG})^{-1}A{\rm Op}_h(e^{NG})^{-1}{\rm Op}_h(e^{NG})f_h, {\rm Op}_h(e^{NG})^*u_h\cjd_{L^2}=o(h^{2})\]
by using $\eqref{Phu_jnewcond}$ and that ${\rm Op}_h(e^{NG})^{-1}A{\rm Op}_h(e^{NG})^{-1}\in \Psi_h^{\rm comp}(\mc{M};\mc{L}^{n(h)})$ is uniformly bounded on $L^2$. 
This gives by applying \eqref{eq:defhamilt}, \eqref{eq:commpoisson},  \eqref{eq:Poissonconvergence}, \eqref{eq:limitlajneq} and  passing to the limit $h\to 0$ 
\[ 0=\int_{T^*\mc{M}}(\demi H^{\omega_{1-|\Lambda|}}_p(a)-\nu a)d\mu.\]
In view of  \eqref{Hpomega=Hp}, the proof is finished.
\end{proof}
Let us define
\bq
\Gamma_+:=E_0^*\oplus E_u^*, \qquad \Gamma_-:=E_0^*\oplus E_s^*,\qquad 
K:=\Gamma_+\cap\Gamma_-,\label{eq:defKGamma}
\eq
and for $\rho\in \R$
\begin{align*}
\Gamma_+(\rho)&:=\{ (x,\xi)\in T^*\mc{M}\,|\, \xi\in \rho \alpha(x)+E_u^*(x)\}\subset \Gamma_+,\\
\Gamma_-(\rho)&:=\{ (x,\xi)\in T^*\mc{M}\,|\, \xi\in \rho \alpha(x)+E_s^*(x)\}\subset \Gamma_-,\\
K_\rho&:=\{(x,\xi)\in T^*\mc{M}\, |\, \xi=\rho \alpha(x)\}=\Gamma_+(\rho)\cap \Gamma_-(\rho)\subset K,
\end{align*}
where $\alpha$ is the contact form on $\M=S{\bf M}$ (identifying $S{\bf M}=S^\ast {\bf M}$ using the metric). 
Next we use radial point estimates to get 
\begin{lemma}\label{supportmu2}
If Assumption \ref{ass:contradition} is fulfilled, the support of $\mu$ is contained in $\Gamma_+(\Lambda)$. Moreover, each open set $V\subset  T^\ast \M$ containing $K_\Lambda$ satisfies $\mu(V\cap \Gamma_+(\Lambda))>0$.
\end{lemma}
\begin{proof} 
We apply  the high-regularity 
radial estimate \cite[Theorem E.52]{dyatlovzworskibook} to the  semiclassical operator  $-iP_{h}(\la)$. First, we observe that $\Gamma_-(\Lambda)$ 
is a radial source for $p_0+\Lambda$: indeed, the function $p$ from \eqref{eq:princsymplambd} is real and, recalling  \eqref{Hpomega=Hp}, the flow of $H^{\omega_{1-|\Lambda|}}_p=H^{\omega_0}_{p_0}$  is simply given by 
\[ \Phi_t: (x,\xi)\mapsto (\varphi_t(x), 
(d\varphi_t(x)^{-1})^T\xi).\]
Let $\bbar{T}^*\mc{M}$ be the radial compactification in the fibers of $T^*\mc{M}$.
The hyperbolicity of the vector field $X$ implies that $L:=E_s^*\cap \pl \bbar{T}^*\mc{M}$ is a hyperbolic repulsor for the flow $\Phi_t$ viewed on $\bbar{T}^*\mc{M}$. It is then a radial source in the sense of \cite[Definition E.50]{dyatlovzworskibook}.  The idea of this radial source propagation estimate is that the negativity of some subprincipal term in $P_h(\la)$ allows to control the microlocal mass of $u_h$ near $L$ by the microlocal mass of $P_h(\la)u_h$ in a slightly bigger neighborhood. The hyperbolicity of the flow is crucial there.
A function $\nu_h$ representing the symbol $\sigma_{h}(h^{-1}{\rm Im}(-iP_{h}(\la)))$ is given by $\nu_h:=-{\rm Re}(\la)/h \in [-c_1,c_0]$.  Let us now check that the eventual negativity of some subprincipal term of $P_h(\la)$ required in \cite[Theorem E.52]{dyatlovzworskibook} is satisfied.
Since  $N>\frac{c_0}{\gamma_\mathrm{min}}$, for all $0<\eps<N\gamma_{\min}-c_0$ there is $T>0$ large enough such that for all $\xi$ near $L$ in the radial compactification $\bbar{T}^*\mc{M}$ and all $h\in D$
\begin{equation}\label{eventual<0}
\begin{split}
\int_0^T \Big(\nu_h+N\frac{H_{p}^{\omega_{1-|\Lambda|}}\cjg \xi\cjd}{\cjg\xi\cjd}\Big)\circ \Phi_t \, dt&= T\nu_h+N\int_0^T \pl_t\log(\cjg \xi\circ \Phi_t\cjd) dt \\
 &=T\nu_h+N\log \cjg \xi \circ d\varphi_T^{-1})\\
 &\leq  T(\nu_h-N\gamma_{\min}+\eps)<0
\end{split}\end{equation}
(recall that $\gamma_{\min}$ is defined in \eqref{eq:minexpansionrate}).
Now, choose $B\in \Psi^0_{h}(\mc{M};\mc{L}^{n(h)})$ such that $\mathrm{ell}_h(B)\supset L$ and $BP_{h}(\la)u_h\in H_h^N(\M;\L^{n(h)})$.  This is possible by the assumption that $P_{h}(\la)u_h\in\mc{H}_h^{Nm}(\M;\L^{n(h)})$ and the construction of $\mc{H}_h^{Nm}(\M;\L^{n(h)})$, which is near $E_s^*$  microlocally equivalent to $H_h^N(\M;\L^{n(h)})$. We can then  conclude the following from \cite[Theorem E.52, Exercise E.35]{dyatlovzworskibook}: 
there is $A_L\in \Psi^0_{h}(\mc{M};\mc{L}^{n(h)})$ with ${\rm WF}_h(1-A_L)\cap L=\emptyset$ and some $C>0$ such that for all $h\in D$
\begin{equation}\label{estimeesource} 
\|A_Lu_{h}\|_{\mc{H}_h^{Nm}(\mc{M};\L^{n(h)})}\leq Ch^{-1}\|BP_{h}(\la)u_h\|_{\mc{H}_h^{Nm}(\M;\L^{n(h)})}
+Ch^{N}\|u_h\|_{H_h^{-N}(\M;\L^{n(h)})}=o(h).
\end{equation}
Here we used \eqref{Phu_jnewcond} and that $\mc{H}_h^{Nm}(\mc{M};\L^{n(h)})$ is microlocally equivalent to $H_h^{N}(\mc{M};\L^{n(h)})$ near $E_s^*$.
 This implies that  for each $A\in \Psi^{\rm comp}_h(\mc{M};\L^{n(h)})$ with ${\rm WF}_h(A)\subset \{(x,\xi)\in T^*\mc{M}\,|\, \sigma_h(A_L)=1\}$, we have as $h\to 0$ 
\[ \cjg Au_h,u_h\cjd=\cjg AA_Lu_h,A_Lu_h\cjd+\mc{O}(h)\leq C\|A_Lu_h\|^2_{H_h^s(\M;\L^{n(h)})}+\mc{O}(h)\to 0,\]
which shows that $\supp(\mu)\cap U=\emptyset$ for some  small neighborhood $U$ of $L$ in $\bbar{T}^*\mc{M}$. Using the invariance of the support of $\mu$ by the Hamilton flow $\Phi_t$ of $H_{p_0}^{\omega_0}$ implied by  \eqref{eq:mupshfwd}, we deduce that $\supp(\mu)\cap V=\emptyset$ for every open set $V$ for which there is $t\in \rr$ such that $\Phi_t(V)\subset U$.
Since $L$ is a hyperbolic repulsor (source) in $\bbar{T}^*\mc{M}$ for the flow $\Phi_t$, around each point 
$(x,\xi)\in \{ 
\xi(X(x))=\Lambda\}\setminus \Gamma_+(\Lambda)$ there is a small ball $B(x,\xi)$ and $T>0$ large
such that $\Phi_{-T}(B(x,\xi))\subset U.$ Combining this with Lemma \ref{supportmu} implies the claim about the support of $\mu$. 

Let us next show that $\mu$ can not vanish near $K_\Lambda$. First, we shall use microlocal ellipticity to show that there is no microlocal mass of $u_h$ outside the characteristic set $\{\xi(X)=\Lambda\}$.
By \cite[Theorem E.33]{dyatlovzworskibook}, let $A_0\in \Psi_h^0(\mc{M};\L^{n(h)})$
be such that $P_h(\la)$ is semiclassically elliptic on ${\rm WF}_h(A_0)$, i.e., if ${\rm WF}_h(A_0)\subset \bbar{T}^*\mc{M}\setminus \{\xi(X)=\Lambda\}$, and ${\rm WF}(1-A_0)$ supported close to $\{\xi(X)=\Lambda\}$ in $\bbar{T}^*\mc{M}$, then there is $C>0$ such that for all $h\in D$ small
\[\|A_0 u_h\|_{\mc{H}^{Nm}_h(\M;\L^{n(h)})}\leq C\|P_h(\la) u_h\|_{\mc{H}_h^{Nm}(\M;\L^{n(h)})}+Ch^N\|u_h\|_{\mc{H}_h^{Nm}(\M;\L^{n(h)})}=o(h^{2}).\]

To analyze the mass of $u_h$ near the unstable bundle at infinity $L':=E_u^*\cap \pl \bbar{T}^*\mc{M}$, we shall use the radial sink estimates \cite[Theorem E.54]{dyatlovzworskibook}. Under a negativity assumption on some subprincipal term of $P_h(\la)$ on $L'$, we can control the microlocal mass of $u_h$ near $L'$ by the microlocal mass of $P_h(\la)u_h$ in some closed
set not intersecting $L'$ in $\bbar{T}^*\mc{M}$. First, we need to check the following eventual negativity: there is $T>0$ large enough so that for all $(x,\xi)$ near $L'$
 \[\int_0^T \Big(\nu_h-N\frac{H_{p}^{\omega_{1-|\Lambda|}}\cjg \xi\cjd}{\cjg\xi\cjd}\Big)\circ \Phi_t \, dt<0\]
But this follows exactly as in \eqref{eventual<0} under the assumption that $N\nu_{\min}>c_0$.
The radial estimate for the sink $L':=E_u^*\cap \pl \bbar{T}^*\mc{M}$ \cite[Theorem E.54]{dyatlovzworskibook} says that for each  $B_1\in \Psi_h^{0}(\mc{M};\L^{n(h)})$  with $\mathrm{ell}_h(B_1)\supset L'$  there are $A_{L'},B_{L'}\in \Psi_h^{0}(\mc{M};\L^{n(h)})$ with  ${\rm WF}_h(1-A_{L'})$ not intersecting $L'$ and $\WF_h(B_{L'})\subset \mathrm{ell}_h(B_1)\setminus L'$ such that if $h$ is small enough one has
\begin{equation}\label{radialsink} 
\begin{split}
\|A_{L'}u_h\|_{\mc{H}_h^{Nm}}&\leq  Ch^{-1}\|B_1P_h(\la)u_h\|_{\mc{H}_h^{Nm}}+C\|B_{L'}u_h\|_{\mc{H}_h^{Nm}}+Ch^N\\
&\leq  C\|B_{L'}u_h\|_{\mc{H}_h^{Nm}}+o(h)
\end{split}
\end{equation}
for some $C>0$ independent of $h$, where we used that $\mc{H}_h^{Nm}(\mc{M};\L^{n(h)})$ is microlocally equivalent to $H_h^{-N}(\mc{M};\L^{n(h)})$ near $L'$ and we can assume ${\rm WF}_h(A_{L'})$ to be contained in a small neighborhood of $L'$.

Assume now that there is $A_K\in \Psi^{\rm comp}_h(\mc{M};\L^{n(h)})$ with ${\rm WF}_h(1-A_K)\cap K_\Lambda=\emptyset$ such that $\|A_Ku_h\|_{\mc{H}_h^{Nm}(\mc{M};\L^{n(h)})}=o(1)$ for some subsequence $h\in D$ going to $0$. We note that by choosing $T$ large enough,
\[ Z_T:=\bigcup_{t\in [-T,T]}\Phi_t({\rm ell}_h(A_0)\cup {\rm ell}_h(A_L)\cup {\rm ell}_h(A_{K}))\]
is such that $\bbar{T}^*\mc{M}\setminus Z_T$ is a small neighborhood of $L'$ 
not intersecting ${\rm WF}_h(B_{L'})$ and contained in the region where $A_{L'}=1$ microlocally. 
This means that there is $T>0$ large enough so that $\Phi^{-T}(\WF(B_{L'}))\subset Z_T$ and we can thus control, by propagation of singularities applied to $P_h(\lambda)u_h$, the  mass of $B_{L'}u_h$ by that of $P_h(\lambda)u_h$ and of $A_{\bullet}u_h$ for $\bullet\in \{K,L,0\}$.
Indeed, by
\cite[Thm.\ E.47]{dyatlovzworskibook}, there is $C>0$ such that 
\[ \|B_{L'}u_h\|_{\mc{H}_h^{Nm}}\leq C\|A_Ku_h\|_{\mc{H}_h^{Nm}}+
C\|A_Lu_h\|_{\mc{H}_h^{Nm}}+C\|A_0u_h\|_{\mc{H}_h^{Nm}}+C\|P_h(\la)u_h\|_{\mc{H}_h^{Nm}}=o(1).\]
 Here we argue as in the proof of \cite[Proposition 3.4]{DZ16} by first conjugating $P_h$ by ${\rm Op}_h(e^{NG})$ so that the propagation estimate is done on $L^2$.
Similarly for $A_R:=1-A_K-A_L-A_0-A_{L'}$
\[\|A_Ru_h\|_{\mc{H}_h^{Nm}}\leq C\|A_Ku_h\|_{\mc{H}_h^{Nm}}+
C\|A_Lu_h\|_{\mc{H}_h^{Nm}}+C\|A_0u_h\|_{\mc{H}_h^{Nm}}+C\|P_h(\la)u_h\|=o(1).\]
Thus, we obtain using \eqref{radialsink}
\[ \|A_{L'}u_h\|_{\mc{H}_{h}^{Nm}}=o(1),\]
which finally leads to $\|u_h\|_{\mc{H}_h^{Nm}(\mc{M};\L^{n(h)})}=o(1)$, leading to a contradiction.
This shows that $\mu(V)>0$ for some neighborhood $V$ of $K_\Lambda$.
 \end{proof}

\subsection{Spectral gap}\label{sec:gap} We assume now that ${\bf M}=\gam \H^3$ is a compact hyperbolic $3$-manifold and that the line bundle $\L$ over $\M=S{\bf M}=\gam G/M$ is of the form 
\bq
\L=\L^{n_0}=\gam G\times_{\varrho_{n_0}}\C,\qquad n_0\in\{-1,0,1\}
\label{eq:Lform}
\eq 
in the notation of Section \ref{sec:bundles};  when $n_0=0$, $\mc{L}^0=\mc{M}\times \cc$ is the trivial bundle since $\rho_0$ is the trivial representation. We equip $\L$ with the connection $\nabla$ from \eqref{eq:nabla}.
\begin{lemma}\label{lem:checkcondition}$(\L,\nabla)$ fulfills the condition \eqref{eq:OmegaX}.
\end{lemma}
\begin{proof}
Let $x\in \M$. First, we claim that there is a neighborhood $W\subset \M=S{\bf M}$ around $x$ on which $\L$ is trivialized by a section ${\bf s}\in \Cinft(W;\L)$ such that $
\nabla_X {\bf s}=0$: indeed, fixing a transverse slice $S$
 to $X$ containing $x$, one can take ${\bf s}_0$ to be a unit length section of $\mc{L}|_{S}$ and the equation $\nabla_X{\bf s}=0$ with boundary condition ${\bf s}|_S={\bf s}_0$ can be solved by the method of characteristics (this is parallel transport along flow lines of $X$); the solution ${\bf s}$ is of unit length since $\nabla$ preserves the Hermitian norm on $\mc{L}$. Now, let  $\tilde{\pi}:F{\bf M}\to S{\bf M}$ be the projection, and $\widetilde W:=\tilde{\pi}^{-1}(W)\subset F{\bf M}$. Then, if $\vartheta$ is the local connection $1$-form of $\nabla$ on $W$ defined by $\nabla {\bf s}=\vartheta\otimes {\bf s}$ and $\tilde \vartheta=\tilde{\pi}^\ast \vartheta$ is its lift to a $1$-form on  $\widetilde W$, we have $\tilde \vartheta(X)=0$ and the commutation relations \eqref{commutation_relations_XUR} imply $X(\tilde \vartheta(U_j^\pm))=\pm\tilde \vartheta(U_j^\pm)$. This gives us for $j=1,2$
\bqn
\Omega(X,d\tilde{\pi} U_j^\pm)=d\tilde \vartheta(X,U_j^\pm)
=X(\tilde \vartheta(U_j^\pm))-U_j^\pm(\tilde \vartheta(X))-\tilde \vartheta([X,U_j^\pm])=(X\mp 1)(\tilde \vartheta(U_j^\pm))=0.
\eqn
As $d\tilde{\pi}$ is surjective, $d\tilde{\pi} X=X$ and $d\tilde{\pi} R=0$, and the vector fields $X,R$ and $U_j^\pm$ span $T(F{\bf M})$, this proves \eqref{eq:OmegaX}.
\end{proof}
The first tool developed in this section is the following technical local result:
\begin{lemma}\label{lem:Ubf}
For each $x\in \mc{M}=S{\bf M}$, there is a neighborhood $W$ of $x$ in $\M$ on which there exist vector fields ${\bf U}^\pm_j$, $j=1,2$, such that ${\bf U}^+_1,{\bf U}^+_2$ span $E_s|_W$, ${\bf U}^-_1,{\bf U}^-_2$ span $E_u|_W$, and the formal adjoints of $\nabla_{{\bf U}_j^\pm}:\Cinft(W;\L^{n_0})\to \Cinft(W;\L^{n_0})$ are given by
\bq
\nabla_{{\bf U}_j^\pm}^\ast=-\nabla_{{\bf U}_j^\pm} - f_j^\pm,\qquad j=1,2, \label{eq:nablaUadj}
\eq
with some smooth functions $f_1^\pm,f_2^\pm:W\to \R$.  
Furthermore, the functions $\varphi^\pm_j:T^\ast W\to \R$ defined by 
$
\varphi^\pm_j(x,\xi):=\xi(\mathbf{U}_j^\pm(x))
$ 
fulfill for every $\rho \in \R$
\bq\begin{split}
H^{\omega_\rho}_{p_0}\varphi_j^\pm &= \pm \varphi_j^\pm,\\ 
\{\varphi^\pm_1,\varphi^\pm_2\}_{\omega_\rho}&=F_2^\pm\varphi^\pm_1-F_1^\pm\varphi^\pm_2,\\
\{\varphi_1^\pm,\varphi_2^\mp\}_{\omega_\rho} &= -\frac{n_0 \rho}{2} + F^\pm_2\varphi_1^\mp - F^\mp_1\varphi_2^\pm,\\
  \{\varphi_j^+,\varphi_j^-\}_{\omega_\rho} &=  \frac{p_0}{2} + F^-_{j'}\varphi_{j'}^+ - F^+_{j'}\varphi_{j'}^-,\qquad   1':=2,\quad 2':=1,\label{eq:commq_j}
\end{split}
\eq
where $\omega_\rho$ is the symplectic form defined in \eqref{eq:omega_rho}, $n_0\in \{-1,0,1\}$ the weight of the representation defining $\L$ in  \eqref{eq:Lform}, $p_0:T^\ast W\to \R$ the function defined in  \eqref{eq:princsymb} restricted to $T^\ast W$, and $F_j^\pm=f_j^\pm\circ \pi$, $\pi:T^\ast W\to W$ being the cotangent bundle projection.
\end{lemma} 
\begin{rem} The functions $\varphi_j^\pm$ satisfy 
\[\begin{split}
\Gamma_-|_W&=(E_s^*\oplus E_0^*)|_W=\{(x,\xi)\in T^*W\,|\,\varphi_1^+(x,\xi)=\varphi_2^+(x,\xi)=0\}, \\
\Gamma_+|_W&=(E_u^*\oplus E_0^*)|_W=\{(x,\xi)\in T^*W\,|\,\varphi_1^-(x,\xi)=\varphi_2^-(x,\xi)=0\} 
\end{split}\]
and the relations \eqref{eq:commq_j} reflect the fact that, for each $\rho$ so that $n_0\rho\not=0$, 
$K=E_0^*=\Gamma_+\cap \Gamma_-$ is  a symplectic submanifold of $T^\ast \M$ with respect to the symplectic form $\omega_\rho$. In the case $n_0\rho=0$ (which contains the case of the trivial bundle $n_0=0$), $K\cap \{p_0\not=0\}$ is still a symplectic submanifold. This fact is crucial for the existence of a uniform spectral gap in what follows, indeed this is the symplectic structure of the trapped set $K$ which provides the mechanism causing that gap in a way comparable to \cite{Li04,NZ15,Ts12}.
\end{rem}
\begin{proof}[Proof of Lemma \ref{lem:Ubf}]For notational simplicity,  we prove only the statements with the ``$-$'' sign. This is justified by Remark \ref{eq:signs}: the statements with the ``$+$'' sign will be obtained by replacing $\eta_\pm$ by $\mu_\pm$ in the following, taking into account that the only difference between $\eta_\pm$ and $\mu_\pm$ with respect to the commutation relations \eqref{eq:etamucomm} is that $X\eta_\pm =-\eta_\pm$ whereas $X\mu_\pm=\mu_\pm$.  A possibly confusing point here is that the lower index $\pm$ of those elements does not correspond to the symbol $\pm$ in \eqref{eq:commq_j} and \eqref{eq:nablaUadj}. Instead, we will see that the ``$\pm$'' in $\eta_\pm$ and $\mu_\pm$ corresponds to $j=1,2$ in \eqref{eq:commq_j} and \eqref{eq:nablaUadj}. 

Let $x\in \M$. As seen in the proof of Lemma \ref{lem:checkcondition} there is a neighborhood $W\subset \M$  around $x$ on which $\L^1$ is trivialized by a section ${\bf s}^+\in \Cinft(W;\L^{1})$ of norm $1$ such that $
\nabla_X {\bf s}^{+}=0$. 
 It corresponds to a non-vanishing complex-valued function $\tilde s^+$ on the subset $\widetilde W:=\tilde{\pi}^{-1}(W)\subset F{\bf M}$ of the frame bundle $F{\bf M}=\Gamma\backslash G$, where  $\tilde{\pi}: F{\bf M}\to S{\bf M}$ is the projection. Then the complex conjugated function $\tilde s^-:=\overline{\tilde s^+}$ fulfills $\tilde s^+(w)\tilde s^{-}(w)=1$ for all $w\in \widetilde W$ and induces a local section ${\bf s}^-\in \Cinft(W;\L^{-1})$ trivializing $\L^{-1}$,  
and the functions $\tilde s^\pm$ fulfill 
\bq
R \tilde s^{\pm}=\pm  i \,\tilde s^{\pm},\qquad X \tilde s^{\pm}=0. \label{eq:equivs0}
\eq
By \eqref{eq:etamucomm} and \eqref{eq:equivs0}, the vector fields ${\tilde s}^{\mp}\eta_{\pm}$ on $F{\bf M}$ commute with $R$, so they descend to vector fields
\bq
{\bf s}^{\mp}\otimes\eta_{\pm} := d\tilde{\pi}({\tilde s}^{\mp}\eta_{\pm })= \frac{1}{2}d\tilde{\pi}({\tilde s}^{\mp}(U_1^-\pm i\,U_2^-))\label{eq:smpetapm}
\eq
on $\M=S{\bf M}$.  
As the two vector fields ${\tilde s}^{\mp 1}(U_1^-\pm iU_2^-)$ span the complexified lifted unstable bundle $\cc \til{E}_u|_{\widetilde W}=(d\tilde{\pi})^{-1}(\cc E_u|_W)$ and $\til{E}_u\cap \R R=0$, the vector fields ${\bf s}^{-}\eta_{+},{\bf s}^+\eta_{-}$  span $\cc E_u|_W$, and recalling that $\overline{\tilde s^\pm}=\tilde s^{\mp}$, we see that the real vector fields 
\bq
{\bf U}^-_1:=\mathrm{Re}({\bf s}^{-}\eta_{+})= \mathrm{Re}({\bf s}^{+}\eta_{-}),\quad {\bf U}^-_2:=\mathrm{Im}({\bf s}^{-}\eta_{+})=-\mathrm{Im}({\bf s}^{+}\eta_{-}) \label{eq:Ubfdef}
\eq
span $E_u|_W$. As mentioned above, ${\bf U}^+_j$ are defined similarly using $\mu_\pm$ instead of $\eta_\pm$. 
%The identity \eqref{eq:nablaUadj} is then obvious using that $\nabla$ is Hermitian, for each $n_0\in \{-1,0,1\}$.

To prove \eqref{eq:commq_j}, we consider commutators of the operators $\nabla_{{\bf U}_j^-}$ with each other and with $\Xbf=\nabla_X$ when acting on sections of $\mc{L}^{n_0}$. Note that instead of $\nabla$ we could write more precisely $\nabla^{n_0}$ for the connection on $\L^{n_0}$, however we avoid this to keep the notation simple. We get using $X \tilde s^{\pm}=0$ 
\bqn
[X,{\tilde s}^{\mp}\eta_\pm]={\tilde s}^{\mp}[X,\eta_\pm]=-{\tilde s}^{\mp}\eta_\pm.
\eqn
Passing to real and imaginary parts (taking into account that $\Xbf$ is real) gives
\bq
[\Xbf,\nabla_{{\bf U}_j^-}]=- \nabla_{{\bf U}_j^-},\qquad j=1,2.\label{eq:commXnabla}
\eq
The case $n_0=0$ tells us that $[X,{\bf U}_j^-]=-{\bf U}_j^-$, which using \eqref{Hpomega=Hp}  
directly implies that $H^{\omega_\rho}_p \varphi_j^-=-\varphi_j^-$. 
The same argument produces $H^{\omega_\rho}_p \varphi_j^+=\varphi_j^+$.

Next we compute $[\nabla_{{\bf U}^-_1},\nabla_{{\bf U}^-_2}]$. For $f\in \Cinft(\widetilde W)$ one gets using $[\eta_+,\eta_-]=0$
\begin{align*}
{\tilde s}^{-}\eta_+ ({\tilde s}^{+}\eta_-f)-{\tilde s}^{+}\eta_-({\tilde s}^{-}\eta_+f) &=\eta_+\eta_-(f)-\eta_-\eta_+(f)+{\tilde s}^{-}\eta_+ ({\tilde s}^{+})\eta_-(f) -{\tilde s}^{+}\eta_-({\tilde s}^{-})\eta_+(f)\\
&=- \eta_+ ({\tilde s}^{-}){\tilde s}^{+}\eta_-(f) +\eta_-({\tilde s}^{+}){\tilde s}^{-}\eta_+(f),
\end{align*}
where we used that $0=\eta_\pm(1)=\eta_\pm({\tilde s}^{-}{\tilde s}^{+})={\tilde s}^{+}\eta_\pm({\tilde s}^{-}) +{\tilde s}^{-}\eta_\pm({\tilde s}^{+})$. This shows
\bq
[\nabla_{{\bf U}^-_1},\nabla_{{\bf U}^-_2}]=f_2^-\nabla_{{\bf U}^-_1}-f_1^-\nabla_{{\bf U}^-_2}\label{eq:Ucomm1},
\eq
\bq
f_1^-:=\Re(\eta_+ {\bf s}^{-})=\Re(\eta_-{\bf s}^{+}),\qquad f_2^-:=\Im(\eta_+ {\bf s}^{-})=-\Im(\eta_-{\bf s}^{+}). \label{eq:deffpmj}
\eq
We notice that the case $n_0=0$ is included in this discussion with $\nabla_{{\bf U}_j^-}$ being simply ${\bf U}_j^-$ and \eqref{eq:Ucomm1} meaning $[{\bf U}^-_1,{\bf U}^-_2]=f_2^-{\bf U}^-_1-f_1^-{\bf U}^-_2$. We thus obtain
\[ \Omega_{\nabla}({\bf U}^-_1,{\bf U}^-_2)=[\nabla_{{\bf U}^-_1},\nabla_{{\bf U}^-_2}]-\nabla_{[{\bf U}^-_1,{\bf U}^-_2]}=0.\]
Since $\omega_\rho=\omega_0+i\rho\pi^*\Omega_{\nabla}$, using $\varphi_j^-=\xi({\bf U}_j^-)$ and $d\pi H_{\varphi_j^-}^{\omega_\rho}=d\pi H_{\varphi_j^-}^{\omega_0}={\bf U}^-_j$ we obtain 
\[\{\varphi^-_1,\varphi^-_2\}_{\omega_\rho}=\{\varphi^-_1,\varphi^-_2\}_{\omega_0}=\pi^*(f_2^-)\varphi^-_1-\pi^*(f_1^-)\varphi^-_2,\] 
$\pi:T^\ast W\to W$ being the cotangent projection. 
Before we determine the remaining Poisson brackets, let us compute the formal adjoint of $\nabla_{{\bf U}_j^-}$ for $j=1,2$. 
To this end, we first note that thanks to the $G$-invariance of the measure $\mu_G$ on $F{\bf M}$ the formal adjoint of $\eta_\pm:\Cinft(F{\bf M})\to\Cinft(F{\bf M})$ is given by $-\eta_\mp$. We then find for $f,g\in \Cinft(F{\bf M})$ using $ \overline{\tilde s^\pm}=\tilde s^{\mp}$
\begin{align*}
\int_{F{\bf M}}({\tilde s}^{\mp}\eta_\pm f) \overline g \d\mu_G=\int_{F{\bf M}}(\eta_\pm f) \overline {{\tilde s}^{\pm}g}\d\mu_G&=\int_{F{\bf M}}f\cdot  (\overline{-\eta_\mp({\tilde s}^{\pm}g)})\d\mu_G\\
&=\int_{F{\bf M}}f \cdot\big(\overline {-\eta_\mp({\tilde s}^{\pm})g- {\tilde s}^{\pm}\eta_\mp(g)}\big)\d\mu_G,
\end{align*}
which proves $
({\bf s}^{\mp}\eta_\pm)^\ast=-{\bf s}^{\pm}\eta_\mp - \eta_\mp {\bf s}^{\pm}$.
 Passing to real and imaginary parts and taking into account that $(A+iB)^\ast = A^\ast -i B^\ast$ for real operators $A,B$, we get \eqref{eq:nablaUadj}.

Finally, let us compute the remaining Poisson brackets,  assuming that we have repeated all of the above steps with $\eta_\pm$ replaced by $\mu_\pm$ to treat the ``$+$'' case. Then \eqref{eq:etamucomm} gives us
\bqn
[{\tilde s}^{\mp}\eta_{\pm}, {\tilde s}^{\pm}\mu_{\mp}] =\mp iR -X +\mu_{\mp}({\tilde s}^{\pm}){\tilde s}^{\mp}\eta_{\pm}- \eta_{\pm} ({\tilde s}^{\mp}){\tilde s}^{\pm}\mu_{\mp},
\eqn
which is equivalent, for $n_0\in\{-1,0,1\}$, to the relation on $\mc{L}^{n_0}$
\begin{align*}
[\nabla_{{\bf U}_1^-},\nabla_{{\bf U}_1^+}]
+
[\nabla_{{\bf U}_2^-},\nabla_{{\bf U}_2^+}]
&=-\Xbf + 
f_1^+\nabla_{{\bf U}_1^-}+f_2^+ \nabla_{{\bf U}_2^-} -f_1^-\nabla_{{\bf U}_1^+} 
-f_2^-\nabla_{{\bf U}_2^+},\\
[\nabla_{{\bf U}_2^+},\nabla_{{\bf U}_1^-}]+
[\nabla_{{\bf U}_2^-},\nabla_{{\bf U}_1^+}]
&=-in_0+f_1^+\nabla_{{\bf U}_2^-}-f_2^+\nabla_{{\bf U}_1^-}+f_1^- \nabla_{{\bf U}_2^+}-f_2^-\nabla_{{\bf U}_1^+}\;
\end{align*}
where the operators act on $\Cinft(W;\L^{n_0})\cong \Cinft(\widetilde W)\cap \ker(R-in_0)$. 
Yet another analogous calculation using the commutation relations \eqref{eq:etamucomm} yields
\bqn
[{\tilde s}^{\mp}\eta_{\pm}, {\tilde s}^{\mp}\mu_{\pm}] = \eta_{\pm} ({\tilde s}^{\mp}){\tilde s}^{\mp}\mu_{\pm} -\mu_{\pm}({\tilde s}^{\mp}){\tilde s}^{\mp}\eta_{\pm},
\eqn
which is equivalent to
\begin{align*}
[\nabla_{{\bf U}_1^-},\nabla_{{\bf U}_1^+}]
-
[\nabla_{{\bf U}_2^-},\nabla_{{\bf U}_2^+}] 
&=
f_1^-\nabla_{{\bf U}_1^+} 
-f_2^-\nabla_{{\bf U}_2^+}-f_1^+\nabla_{{\bf U}_1^-}
 +f_2^+\nabla_{{\bf U}_2^-},\\
[\nabla_{{\bf U}_2^-},\nabla_{{\bf U}_1^+}] 
+
[\nabla_{{\bf U}_1^-},\nabla_{{\bf U}_2^+}]
&=f_2^-\nabla_{{\bf U}_1^+} +f_1^-\nabla_{{\bf U}_2^+} -f_2^+\nabla_{{\bf U}_1^-}
 -f_1^+\nabla_{{\bf U}_2^-}.
\end{align*}
Combining the equations gives us (with the short hand notation $1_\textrm{opp}:=2$, $2_\textrm{opp}:=1$)
\begin{align*}
[\nabla_{{\bf U}_j^+},\nabla_{{\bf U}_j^-}]
&=\frac{\Xbf}{2} 
+f_{j_\textrm{opp}}^-\nabla_{{\bf U}_{j_\textrm{opp}}^+}
-f_{j_\textrm{opp}}^+\nabla_{{\bf U}_{j_\textrm{opp}}^-},\\
[\nabla_{{\bf U}_1^\pm},\nabla_{{\bf U}_2^\mp}]
&= \frac{in_0}{2}
-f_1^\mp\nabla_{{\bf U}_2^\pm} 
+f_2^\pm\nabla_{{\bf U}_1^\mp}\qquad\qquad \text{on }\Cinft(W;\L^{n_0}).
\end{align*}
This gives 
\[ \Omega_{\nabla}({\bf U}_j^+,{\bf U}_j^-)=0, \quad 
\Omega_{\nabla}({\bf U}_1^\pm,{\bf U}_2^{\mp})=\frac{in_0}{2}\]
and we obtain  the remaining relations in \eqref{eq:commq_j}.
\end{proof}

\begin{lemma}\label{lem:nugt1}
Suppose that Assumption \ref{ass:contradition} is fulfilled with $\nu>-1$. Associate a distribution $\mu_{\omega_{1-|\Lambda|}}\in \D'(T^\ast \M)$ to the semiclassical measure $\mu$ by the expression
\[\cjg \mu_{\omega_{1-|\Lambda|}} , f\omega_{1-|\Lambda|}^{\wedge 5}\cjd:= \int_{T^\ast \M} f \d\mu\qquad  \forall f\in \CT(T^\ast \mc{M}),
\]
where the symplectic form $\omega_{1-|\Lambda|}$ is defined by  \eqref{eq:omega_rho}. Further, let $W\subset \M$, $F_j^-\in \Cinft(W)$, $\varphi_j^-\in \Cinft(T^\ast W)$, $j=1,2$, be as in Lemma \ref{lem:Ubf}. Then one has on $T^\ast W$
\[\forall j=1,2, \quad  H^{\omega_{1-|\Lambda|}}_{\varphi_j^-}\mu_{\omega_{1-|\Lambda|}}= F_j^-\mu_{\omega_{1-|\Lambda|}}.\]  
In particular, $\mu_{\omega_{1-|\Lambda|}}|_{T^\ast W}$ is smooth in the direction of ${\rm span}(H^{\omega_{1-|\Lambda|}}_{\varphi_1^-},H^{\omega_{1-|\Lambda|}}_{\varphi_2^-})$.
\end{lemma}
\begin{proof} 
By the commutation relation \eqref{eq:etamucomm}, we have 
\[h\eta_\pm f_h=h\eta_\pm(h{\bf X}+\la)u_h=(h{\bf X}+\la +h)h\eta_\pm(h)u_h=(P_h(\la)+h)h\eta_\pm u_h.\]
Let $\mc{H}_h^{Nm'}(\mc{M};\mc{L}^{n(h)\pm 1}):={\rm Op}_{h,n(h)\pm 1}(e^{NG'})^{-1}L^2(\mc{M};\mc{L}^{n(h)\pm 1})$
with 
\[G'(x,\xi):=(m(x,\xi)-\frac{1}{N})\log F(x,\xi), \qquad m'(x,\xi):=m(x,\xi)-\frac{1}{N}\]
and where ${\rm Op}_{h,n(h)\pm 1}$ are the quantization maps defined in \eqref{eq:Oph} for the semiclassical calculi associated with the two new tensor power maps $D\owns h\mapsto n(h)\pm 1$. The operators $\eta_\pm$ allow us to switch from the $n$-calculus to the $(n\pm 1)$-calculi, and tensoring with ${\bf s}^{\mp}$ switches back. 
First, we have that $h\eta_\pm f_h\in \mc{H}_h^{Nm'}(\mc{M};\mc{L}^{n(h)\pm 1})$ 
and, with $L^2_{n}:=L^2(\mc{M};\mc{L}^{n})$,  
\begin{equation}\label{littleoh}
 \begin{split}
&\|{\rm Op}_{h,n(h)\pm 1}(e^{NG'})h\eta_\pm f_h\|_{L^2_{n(h)\pm 1}} \\
& \qquad\qquad\leq \|{\rm Op}_{h,n(h)\pm 1}(e^{NG'})h\eta_\pm {\rm Op}_{h,n(h)}(e^{NG})^{-1}\|_{L^2_{n(h)}\to L^2_{n(h)\pm 1}} \|f_h\|_{\mc{H}^{Nm}_h}\\
 &\qquad\qquad =  o(h^{2}).\end{split}
 \end{equation}
Here we used that on each open set $W\subset \mc{M}$ and  local section ${\bf s}$ of $\mc{L}$ over $W$ and each $\chi\in C_c^\infty(W)$, the operator $\chi{\bf s}^{\mp}{\rm Op}_{h,n(h)\pm 1}(e^{NG'})h\eta_\pm {\rm Op}_{h,n(h)}(e^{NG})^{-1}\in \Psi^{0}_h(\mc{M};\mc{L}^{n(h)})$ belongs to the $n$-calculus and has a uniformly bounded principal symbol with respect to $h$.
 We also note that, from the construction of $m,F$ in \cite[Lemma 1.2]{FS11}, $G'$ is also an escape function satisfying $H^{\omega_0}_pG'(x,\xi)\leq 0$ and $H^{\omega_0}_p(m')(x,\xi)\leq 0$ provided $N>1$ and $|\xi|>r$ for some large enough $r$ independent of $N$. Indeed, our function $F$ corresponds to $\sqrt{1+f^2}$ in the notation of \cite[Lemma 1.2]{FS11} and the statements \cite[(2.4) and (2.7)]{FS11} remain true if Faure-Sj\"ostrand's function $\tilde m$ is replaced by $\tilde m':=\tilde m- 1/N$ since that function differs from  $\tilde m$ only by a constant. The arguments after \cite[(2.9)]{FS11} then still work for $m'$ because $m-1/N>0$ near $E_s$, $m-1/N$ is bounded uniformly in $N>1$, and ${\bf X}(m')={\bf X}(m)$ in the notation of Faure-Sj\"ostrand.
 
Assuming that $\nu>-1$, we claim that $P_h(\la)+h$ is invertible for small $h$ on the spaces $\mc{H}_h^{Nm'}(\mc{M};\mc{L}^{n(h)\pm 1})$ with the estimate 
\begin{equation}\label{normanisotrope}
\|(P_h(\la)+h)^{-1}\|_{\mc{H}_h^{Nm'}(\mc{M};\mc{L}^{n(h)\pm 1})\to \mc{H}_h^{Nm'}(\mc{M};\mc{L}^{n(h)\pm 1})} \leq Ch^{-1}.
\end{equation}
First, we note that $(P_h(\la)+h)^{-1}:L^2\to L^2$ is well defined and given by the convergent expression
\begin{equation}\label{resolvpropagator} 
(P_h(\la)+h)^{-1}=h^{-1}\int_{0}^{\infty}e^{-t({\bf X}+\la/h+1)}dt.
\end{equation} 
To prove \eqref{normanisotrope}, we can first write for $t\in [0,1]$
\begin{multline*}  \|{\rm Op}_{h,n(h)\pm 1}(e^{NG'})e^{-t{\bf X}}{\rm Op}_{h,n(h)\pm 1}(e^{NG'})^{-1}\|_{L^2\to L^2}
\\=\|e^{t{\bf X}}{\rm Op}_{h,n(h)\pm 1}(e^{NG'})e^{-t{\bf X}}{\rm Op}_{h,n(h)\pm 1}(e^{NG'})^{-1}\|_{L^2\to L^2}
\end{multline*} 
and by using Egorov's theorem we see that $Q(t):=e^{t{\bf X}}{\rm Op}_{h,n(h)\pm 1}(e^{NG'})e^{-t{\bf X}}{\rm Op}_{h,n(h)\pm 1}(e^{NG'})^{-1}$ is an operator in the class $Q(t)\in \Psi_h^{m'\circ \Phi_t-m'}(\mc{M};\L^{n(h)\pm 1})$ with principal symbol 
\[ \sigma_{h,n(h)\pm 1}(Q(t))(x,\xi)=\exp(N (G'(\Phi_t(x,\xi))-G'(x,\xi)))\leq 1\]
thus by \eqref{eq:normbound} and since $\|e^{t{\bf X}}\|_{L^2\to L^2}=1$, there is $C$ such that for all $t\in [0,1]$ and $h>0$ small
\[ \|{\rm Op}_{h,n(h)\pm 1}(e^{NG'})e^{-t{\bf X}}{\rm Op}_{h,n(h)\pm 1}(e^{NG'})^{-1}\|_{L^2\to L^2}\leq 1+Ch\]
which means that $\|e^{-t{\bf X}}\|_{\mc{H}_h^{Nm'}(\mc{M};\mc{L}^{n(h)\pm 1})\to \mc{H}_h^{Nm'}(\mc{M};\mc{L}^{n(h)\pm 1})}\leq 1+Ch$ for all $t\in [0,1]$. This directly implies that for all $t\geq 0$
\[ \|e^{-t{\bf X}}\|_{\mc{H}_h^{Nm'}(\mc{M};\mc{L}^{n(h)\pm 1})\to \mc{H}_h^{Nm'}(\mc{M};\mc{L}^{n(h)\pm 1})}\leq (1+Ch)^{t+1}\]
and thus the integral defining $(P_h(\la)+h)^{-1}$  is convergent with $\mc{H}_h^{Nm'}(\mc{M};\mc{L}^{n(h)\pm 1})$-norm 
$\mc{O}(h^{-1})$ if ${\rm Re}(\la)+h>\eps h$ for some $\eps>0$ and all $h$ small enough. 
We thus obtain using \eqref{littleoh} that 
\[ \|h\eta_\pm u_h\|_{\mc{H}_h^{Nm'}(\mc{M};\mc{L}^{n(h)\pm 1})}=o(h).\]
Switching back to the $n$-calculus using \eqref{eq:Ubfdef}, we deduce that in the open set $W$ we have for each $\chi\in C_c^\infty(W)$ 
\begin{equation}\label{heta+u}
\|\chi h\nabla_{{\bf U}_j^-} u_h\|_{\mc{H}_h^{Nm'}(\mc{M};\mc{L}^{n(h)})}=o(h),\quad j=1,2.
\end{equation}
From now on, we will remain in the $n$-calculus and return to the simplified notation ${\rm Op}_{h}={\rm Op}_{h,n(h)}$.  
Recalling the local formula for our quantization map \eqref{eq:Oph} and using the same trivializing section ${\bf s}\in \Cinft(W;\L)$ as in the proof of Lemma \ref{lem:Ubf},   for $a\in C_c^\infty(T^*\mc{M})$ supported inside $T^\ast W$ the operator
${\rm Op}_{h}(a): C^\infty(W;\L^{n(h)})\to  C^\infty(W;\L^{n(h)})$ is given by 
\[{\rm Op}_{h}(a)(f{\bf s}^{n(h)})(x)=\Big((2\pi h)^{-d}\int e^{i\frac{(x-x')\xi}{h}}a(x,\xi- h n(h)\beta(x))f(x')d\xi dx'\Big){\bf s}^{n(h)},\qquad h\in D,\]
where $-i\beta={\bf s}^{-1}\nabla{\bf s}$ is the connection $1$-form in the trivialisation given by ${\bf s}$; the principal symbol of ${\rm Op}_h(a)$ being represented by $a$ according to  \eqref{eq:defprincsymb}. 
Using \eqref{eq:nablaUadj} and applying \eqref{heta+u} with $\chi\equiv 1$ in a neighborhood of the projection of $\supp a$ to $W$, we can now write for $\til{\chi}\in C_c^\infty(W)$ such that 
$\til{\chi}\chi=\chi$
\bq \begin{split} 
o(1)&=\frac{1}{h}\big(\cjg {\rm Op}_{h}(a)\chi\,h \nabla_{{\bf U}_j^-}u_h,u_h\cjd
+\cjg{\rm Op}_{h}(a)\til{\chi}u_h,\chi h\nabla_{{\bf U}_j^-} u_h\cjd\big)\\
&= \Big< \frac{1}{h}\big[{\rm Op}_{h}(a),\chi h \nabla_{{\bf U}_j^-}\big]\til{\chi} u_h,u_h\Big>- \cjg{f_j^-\rm Op}_{h}(a)\til{\chi}u_h,u_h\cjd\label{eq:ohq-1}
\end{split}
\eq
with $f_j^-$ defined in \eqref{eq:deffpmj}.
As $\varphi_j^-$ represents  $\sigma_h(-i h \nabla_{{\bf U}_j^-})$, \eqref{eq:commpoisson} says that the principal symbol of the operator $\frac{1}{h}[{\rm Op}_{h}(a),\chi h \nabla_{{\bf U}_j^-}]\til{\chi}\in \Psi^{\rm comp}_h(W;\L^{n(h)})$ is represented by $\{ a,\varphi^-_j\}_{\omega_{hn(h)}}$. Further, the modified principal symbol of $f_j^-{\rm Op}_{h}(a)\til{\chi}\in \Psi^{\rm comp}_h(W;\L^{n(h)})$ is represented by $F_j^-a$. 
We thus deduce by letting $h\to 0$ in \eqref{eq:ohq-1}  and using \eqref{eq:limitlajneq}, \eqref{eq:defhamilt}, \eqref{eq:Poissonconvergence}:
\bq 0=\int_{T^*W}(H^{\omega_{1-|\Lambda|}}_{\varphi^-_j} + F_j^-)a\, d\mu.\label{eq:Hiq0}
\eq
We note that $H^{\omega_{1-|\Lambda|}}_{\varphi^-_j}$ preserves $\omega_{1-|\Lambda|}$, since it is a Hamiltonian vector field with respect to $\omega_{1-|\Lambda|}$,  thus it also preserves the associated symplectic measure $\omega_{1-|\Lambda|}^{\wedge 5}$. This implies that when we write  $\mu =\mu_{\omega_{1-|\Lambda|}}\omega_{1-|\Lambda|}^{\wedge 5}$, we get from \eqref{eq:Hiq0} the equality
\bqn
 H^{\omega_{1-|\Lambda|}}_{\varphi^-_j} \mu_{\omega_{1-|\Lambda|}}= F_j^-\mu_{\omega_{1-|\Lambda|}},\qquad j=1,2.
\qedhere
\eqn
\end{proof}

\begin{prop}\label{prop:gap2}
If $\L=\L^{1}$ or $\mc{L}=\mc{L}^{-1}$, Assumption \ref{ass:contradition} cannot be fulfilled with  $c_0<1$. 

If $\L=\L^{0}=\mc{M}\times \C$,  then Assumption \ref{ass:contradition} cannot be fulfilled with  $c_0<1$, $\Lambda \neq 0$.
\end{prop}
\begin{proof} 
In a local neighborhood $W$ as in Lemma \ref{lem:Ubf}, consider the functions  
$\varphi_j^\pm (x,\xi)=\xi({\bf U}^\pm_j(x))$, $j=1,2$. 
Because the ${\bf U}^+_j$ span $E_s|_W$ and the ${\bf U}^-_j$ span $E_u|_W$, the differentials $d\varphi_1^+,d\varphi_2^+,d\varphi_1^-,d\varphi_2^-$ are fiber-wise linearly independent. As the symplectic form $\omega_\rho$ is non-degenerate for each $\rho\in \R$, we conclude that the Hamiltonian vector fields $H^{\omega_\rho}_{\varphi_1^+},H^{\omega_\rho}_{\varphi_2^+}, H^{\omega_\rho}_{\varphi_1^-},H^{\omega_\rho}_{\varphi_2^-}$ are fiber-wise linearly independent for each $\rho\in \R$.  Next, we want to check on $K=\Gamma_+\cap \Gamma_-\subset T^\ast \M$,  introduced in \eqref{eq:defKGamma}, for which points $\kappa\in K\cap T^\ast W$ and which $\rho\in \R$ the vectors $H^{\omega_\rho}_{\varphi_j^\pm}(\kappa)\in T_\kappa (T^\ast W)$ are transverse to $T_\kappa K\subset T_\kappa(T^\ast W)$. To this end, we note that the functions $\varphi_j^\pm$ satisfy 
\bq 
\Gamma_\pm\cap T^\ast W=\{(x,\xi)\in T^\ast W \,|\, \varphi_1^\mp(x,\xi)=\varphi_2^\mp(x,\xi)=0\}.\label{eq:zeroset}
\eq
Also, for every $\rho\in \R$  we have
\begin{align*}
-d\varphi_1^\pm (H^{{\omega_\rho}}_{\varphi_2^\pm})&=d\varphi_2^\pm(H^{{\omega_\rho}}_{\varphi_1^\pm})=\{\varphi_1^\pm,\varphi_2^\pm\}_{{\omega_\rho}}=F_2^\pm\varphi_1^\pm-F_1^\pm\varphi_2^\pm,\\
-d\varphi_1^\pm(H^{\omega_\rho}_{\varphi_2^\mp})&=d\varphi_2^\mp(H^{\omega_\rho}_{\varphi_1^\pm})= \{\varphi_1^\pm,\varphi_2^\mp\}_{\omega_\rho} =  -\frac{n_0\rho}{2} + F^\pm_2\varphi_1^\mp - F^\mp_1\varphi_2^\pm,\\
-d\varphi_j^+(H^{\omega_\rho}_{\varphi_j^-})&=d\varphi_j^-(H^{\omega_\rho}_{\varphi_j^+})=\{\varphi_j^+,\varphi_j^-\}_{\omega_\rho} =  \frac{p_0}{2} + F^-_{j'}\varphi_{j'}^+ - F^+_{j'}\varphi_{j'}^-
\end{align*}
by \eqref{eq:defhamilt} and \eqref{eq:commq_j} with $1':=2$, $2':=1$, which implies 
\begin{align}\begin{split}
0&=d\varphi_1^\pm (H^{{\omega_\rho}}_{\varphi_2^\pm})|_{\Gamma_\mp\cap T^\ast W}=d\varphi_2^\pm(H^{{\omega_\rho}}_{\varphi_1^\pm})|_{\Gamma_\mp\cap T^\ast W},\\
-\frac{n_0\rho}{2}  &=-d\varphi_1^\pm(H^{\omega_\rho}_{\varphi_2^\mp})|_{K\cap T^\ast W}=d\varphi_2^\mp(H^{\omega_\rho}_{\varphi_1^\pm})|_{K\cap T^\ast W},\\
\frac{p_0}{2}|_{K\cap T^\ast W}&=-d\varphi_j^+(H^{\omega_\rho}_{\varphi_j^-})|_{K\cap T^\ast W}=d\varphi_j^-(H^{\omega_\rho}_{\varphi_j^+})|_{K\cap T^\ast W},\qquad j=1,2,
\label{eq:0Poisson}\end{split}
\end{align}
so that we get for $j=1,2$ and every $\rho\in \R$ 
\bq\begin{split}
H^{{\omega_\rho}}_{\varphi_j^\pm}(x)&\in T_x\Gamma_\mp \qquad\forall \;x\in \Gamma_\mp \cap T^\ast W,\\ \label{eq:dphivanish}
H^{{\omega_\rho}}_{\varphi_j^\pm}(\kappa)&\not\in T_\kappa \Gamma_\pm\qquad \begin{cases}\forall\; \kappa \in K,\quad &\text{if } n_0\rho\neq 0,\\
\forall\; \kappa \in K\setminus K_0, & \text{if } n_0\rho= 0.\end{cases}
\end{split}
\eq
Consequently, we have transversality of $H^{\omega_{1-|\Lambda|}}_{\varphi_1^-},H^{\omega_{1-|\Lambda|}}_{\varphi_2^-}$ to $TK$ on all of $K_\Lambda\cap T^\ast W$ provided that $n_0(1-|\Lambda|)\neq 0$ or $\Lambda \neq 0$. If $n_0\neq 0$, then one of these relations is always fulfilled, while for $n_0=0$ we get the condition $\Lambda\neq 0$. Assuming either $n_0\neq 0$ or $\Lambda \neq 0$,  we can apply the inverse function theorem to deduce that there is $s_0>0$ such that the map 
\[ \psi:  (-s_0,s_0)^2\x (K\cap T^\ast W)\owns  (s_1,s_2,\kappa)\mapsto e^{s_1H^{\omega_{1-|\Lambda|}}_{\varphi_1^-}+s_2H^{\omega_{1-|\Lambda|}}_{\varphi_2^-}}(\kappa)
\]
is a diffeomorphism from a neighborhood $(-\eps,\eps)^2\times \mc{O}_\Lambda$ of $\{0\}\times (K_\Lambda\cap T^\ast W)$ onto a neighborhood of $K_\Lambda\cap T^\ast W$ in $\Gamma_+\cap T^\ast W$, where $\mc{O}_\Lambda$ is a neighborhood of $K_\Lambda\cap T^\ast W$ in $K\cap T^\ast W$. Note that the image of $\psi$ is indeed contained in $\Gamma_+$ due to the first relation in \eqref{eq:dphivanish}. The relation \eqref{eq:0Poisson} also implies that $H^{\omega_{1-|\Lambda|}}_{\varphi_1^-}$ and $H^{\omega_{1-|\Lambda|}}_{\varphi_2^-}$ commute on $\Gamma_+\cap T^\ast W$.

Now, suppose that Assumption \ref{ass:contradition} is fulfilled.  Then, by Lemma \ref{supportmu2}, the support of $\mu$ is contained in $\Gamma_+(\Lambda)$, which means that the pullback measure $\psi^*\mu$ is well-defined. Let us check that with the assumption $\nu>-1$ the measure $\psi^*\mu$ satisfies 
\begin{equation}\label{mulipshitz}
\psi^*\mu(B_\delta)\leq C\delta^2
\end{equation}
if $B_{\delta}:=\{|s|\leq \delta\, |\, \kappa\in \mc{O}_\Lambda\}$ for small $\delta>0$. In the variables $(s_1,s_2,\kappa)$,
we have $\pl_{s_j}\psi^*\mu=F_j^-\psi^*\mu$ in the distributional sense by Lemma \ref{lem:nugt1} if $\nu>-1$. If $\chi_\delta(s_1,s_2,\kappa):=\chi(s_1/\delta,s_2/\delta,\kappa)$ is a function supported in $B_{2\delta}$ and equal to $1$ in $B_\delta$, we can use the Fourier transform in the $s$ variable to get  for every $l\in \N_0$
\[ \cjg \psi^*\mu,\chi_\delta\cjd = \delta^2\int \frac{\hat{\chi}(\delta\xi_1,\delta\xi_2,\kappa)}{(1+|\xi|^2)^l} (1+|\xi|^2)^l\widehat{\psi^*\mu}(\xi_1,\xi_2,\kappa) \d\kappa \d\xi_1\d\xi_2 \]
and $(1+|\xi|^2)^l\widehat{\psi^*\mu}=\widehat{f_l\psi^*\mu}$ with  the smooth function $f_l:=(1+(F_1^-)^2+(F_2^-)^2)^l$, so we deduce that 
$\cjg \psi^*\mu,\chi_\delta\cjd =\mc{O}(\delta^2)$.

Now, for $\delta>0$ small, we consider $U_\delta:=\{\zeta \in T^*\mc{M}\, |\, d_{T^*\mc{M}}(\zeta, K_\Lambda)\leq \delta\}$ where $d_{T^*\M}$ is the Sasaki Riemannian distance on $T^*\mc{M}$.
Using Lemma \ref{flowinvmu} we find that  
\begin{equation}\label{bound_propagation_1}
\mu(e^{-tH^{\omega_{1-|\Lambda|}}_p}(U_\delta))= e^{2t\nu}\mu(U_\delta)\qquad \forall\;t\geq0.
\end{equation}
Consider a covering $(W_\ell)_{\ell=0}^{\ell_0}$ of $U_\delta$ by finitely many charts, on which we obtain functions $\varphi_{\ell,j}^{\mp}$ for $j=1,2$ and diffeomorphisms $\psi_\ell$ as above.
Due to the relations $H^{\omega_{1-|\Lambda|}}_p\varphi_{\ell,j}^+ = \varphi_{\ell,j}^+$ that hold 
thanks to \eqref{eq:commq_j}, and in view of \eqref{eq:zeroset}, there is $C$ independent of $\delta>0$ such that for all $t\geq 0$:
\[ e^{-tH^{\omega_{1-|\Lambda|}}_p}(U_\delta\cap \Gamma_+)\subset U_\delta \cap \Gamma_+ \cap U_{C\delta e^{-t}},\]
thus, provided $t$ is large enough,  
\[ \mu(e^{-tH^{\omega_{1-|\Lambda|}}_p}(U_\delta))\leq \mu(U_{C\delta e^{-t}}\cap \Gamma_+).\]
Since one can find $C'>0$ such that 
\[U_{C\delta e^{-t}}\cap \Gamma_+\subset \bigcup_{\ell=0}^{\ell_0} \psi_\ell(B_{C'\delta e^{-t}}),\]
we can use \eqref{mulipshitz} to deduce that there is $C>0$, independent of $\delta$, 
such that for all $t\geq 0$ large enough
\[\mu(e^{-tH^{\omega_{1-|\Lambda|}}_p}(U_\delta))\leq  C\delta^2 e^{-2t}.\]
Combining with \eqref{bound_propagation_1} and recalling from Lemma \ref{supportmu2} that $\mu(U_\delta)>0$, we conclude that there is $C>0$ such that for all $t\geq 0$ large
\[   e^{2t\nu} \leq C e^{-2t}.\]
We conclude that if $n_0\neq 0$ or $\Lambda\neq 0$,  Assumption \ref{ass:contradition} can only be fulfilled if $\nu\leq-1$. Since $\nu\in [-c_0,c_1]$, this is possible only if $c_0\geq 1$.
\end{proof}
Finally, we can prove Proposition \ref{lem:key}:
\begin{proof}[Proof of Proposition \ref{lem:key}]\label{proof:keylem}
Let $c_0\in (0,1)$, $N> c_0$, $c_1>1$, and $n_0\in \{-1,0,1\}$. If $n_0\neq 0$, suppose that \eqref{eq:keybound} does not hold, so that \eqref{eq:contradiction} does not hold, and if $n_0=0$, suppose that \eqref{eq:keybound0} does not hold, so that one arrives at \eqref{eq:cond2} with $n_0=0$. Then Assumption \ref{ass:contradition} is fulfilled for the line bundle $\L=\L^{n_0}$ with the chosen $c_0,c_1,N$. Moreover, if $n_0=0$, the condition $|\Im \lambda|\geq 1$ implies that $\Lambda\neq 0$ in Assumption \ref{ass:contradition}. Applying Proposition \ref{prop:gap2}, we arrive at a contradiction.
\end{proof}

\subsection{Application to exponential mixing}\label{sec:mix}
We conclude by a discussion on the exponential mixing using the resolvent estimate. The argument is quite standard in scattering theory/resonance theory. 
First $e^{-tX}:L^2(F{\bf M})\to L^2(F{\bf M})$ is unitary and its generator $iX$ is self-adjoint on $L^2$. By Stone's formula \cite[Theorem VII.13]{Reed-Simon1}, we can write the spectral measure of $iX$ in terms of the resolvent: for $f\in L^2(F{\bf M})$ with $\cjg f,1\cjd=0$
\[\begin{split}
e^{-tX}f=e^{it(iX)}f=&\frac{1}{2\pi i}\lim_{\eps\to 0}
\int_{-\infty}^\infty e^{it\la}((iX-\la-i\eps)^{-1}-(iX-\la+i\eps)^{-1})fd\la\\
=&-\frac{1}{2\pi}
\int_{-\infty}^\infty e^{it\la} (R(i\la)- R^+(-i\la))f d\la
\end{split}\]
where $R^+(\la)$ is defined as in \eqref{resolventff} but with the flow in forward time: for ${\rm Re}(\la)>0$ 
\[ R^+(\la)f:=\int_0^\infty e^{-\la t}\varphi_t^*f dt.\]
The results of Theorem \ref{prop:gap} apply as well to the resolvent $R^+(\la)$ with the anisotropic spaces $\mc{H}^{Nm,k}$ replaced by $\mc{H}^{-Nm,k}$ for $k=0,1$, so $R^+(\la)$ extends analytically to $\{{\rm Re}(\la)>-1\}$ except at a finite number of poles. The operator $dE_X(\la)=(R(i\la)- R^+(-i\la))$ for $\la\in \R$ is the spectral measure of $X$ in the spectral theorem for $iX$; $dE_X(\la)$ 
is only well defined on $\mc{H}^{Nm,1}\cap\mc{H}^{-Nm,1}$ but its integral over each bounded interval $ [a,b]\owns \lambda$ produces a bounded operator on $L^2$ (equal to the spectral projector of $iX$ on $[a,b]$), and 
\[ -XdE_X(\la)=-dE_X(\la)X=i\la dE_X(\la).\]
We take $f\in \mc{H}^{Nm,1}\cap \mc{H}^{-Nm,1}$ such that $X^jf\in \mc{H}^{Nm,1}\cap \mc{H}^{-Nm,1}$ for $j\leq 5$. Then one can write 
\[ (i\la+1)^{-5}dE_X(\la)f= dE_X(\la)(-X+1)^5f\]
and therefore
\begin{equation}\label{e-tXint} 
e^{-tX}f= -\frac{1}{2\pi}\int_{-\infty}^\infty e^{it\la} (R(i\la)(-X+1)^5f- R^+(-i\la)(-X+1)^5f)\frac{d\la}{(i\la+1)^5}.
\end{equation}
By Theorem \ref{prop:gap} we have for ${\rm Re}(s)\in (-1+\delta,1)$ and $|{\rm Im}(s)|>1$
\begin{equation}\label{boundRsf} 
\|R(s)f\|_{\mc{H}^{Nm}}\leq C\cjg s\cjd^{3}\|f\|_{\mc{H}^{Nm,1}}, \quad 
\|R^+(s)f\|_{\mc{H}^{-Nm}}\leq C\cjg s\cjd^{3}\|f\|_{\mc{H}^{-Nm,1}}.
\end{equation} 
Note that we can pair $f\in \mc{H}^{Nm}$ with $f'\in \mc{H}^{-Nm}$:
\[ \cjg f,f'\cjd = \sum_{n\in \Z} \cjg A_{h(n)}^{N}f_n,A_{h(n)}^{-N}f_n'\cjd_{L^2(\mc{M};\mc{L}^n)}. \] 
Thus if $f'\in \mc{H}^{Nm}\cap \mc{H}^{-Nm}$ and $f\in \mc{H}^{Nm,1}$ 
we can compute $\cjg e^{-tX}f,f'\cjd$ by performing a contour deformation from $\la\in \R$ to $\la\in i\beta+\R$ for $\beta\in (0,1)$ in the integral \eqref{e-tXint}
\[ \begin{split}
\cjg e^{-tX}f,f'\cjd=& -\frac{e^{-t\beta}}{2\pi}\int_{-\infty}^\infty e^{it\la} \cjg R(-\beta+i\la)(-X+1)^5f,f'\cjd \frac{d\la}{(-\beta+i\la+1)^5}\\
& + \frac{e^{-t\beta}}{2\pi}\int_{-\infty}^\infty e^{it\la} \cjg R^+(\beta-i\la)(-X+1)^5f,f'\cjd \frac{d\la}{(-\beta+i\la+1)^5}\\
& + \sum_{j=1}^J e^{t\la_j}\cjg\Pi_jf_0,f_0'\cjd\\
& =\sum_{j=1}^J e^{t\la_j}\cjg\Pi_jf_0,f_0'\cjd+\mc{O}(e^{-t\beta})\|(-X+1)^5f\|_{\mc{H}^{Nm,1}}\|f'\|_{\mc{H}^{-Nm}}
\end{split}\]
where the $\la_j\in (-1,0)$ are the finitely many Pollicott-Ruelle resonances 
described in Theorem \ref{prop:gap} and the $\Pi_j$ are the corresponding spectral projectors (they come from $R_0(s)$ and are related to eigenvalues of the Laplacian on functions); here $f_0,f_0'$  are 
the averages of $f,f'$ in the $\SO(2)$ fibers of $F{\bf M}\to S{\bf M}$.
We also used \eqref{boundRsf} to prove convergence in the integral when we performed the contour deformation and the fact that $R(s)$ is meromorphic with finitely many poles at $s=\la_j$ by Theorem \ref{prop:gap}) (the terms $e^{t\la_j}\cjg\Pi_jf,f'\cjd$ appear as residues), and finally that 
$R^+(s)$ is analytic in ${\rm Re}(s)>0$. 

\appendix

\section{Band structure of line bundle resonances}\label{sec:appendix}
\smallskip
\begin{center}\textsc{Charles Hadfield}\end{center}
\medskip
Here we describe the resonances $\sigma^{\mathrm{PR}}_n$ of 
the operator ${\bf X}=\nabla^n_X$ on sections of the line bundles $\L^n$, using an approach similar to \cite{dfg}, based on horocyclic operators. As in Section \ref{sec:resolv} we use the notation \eqref{eq:notation}, in particular  $\M=S{\bf M}$. 

\subsection{Band structure and interaction between ladder and horocyclic operators}

The relation \eqref{eq:Uminuscoord} shows that (under the appropriate identifications provided by the injections \eqref{eq:injections} and taking into account the suppressed indices $m,n$) one has 
\bq
\S\U_-=\eta_++\eta_-.\label{eq:SUetaeta}
\eq
As a consequence of \eqref{eq:SUetaeta} and the commutation relation $[\eta_+,\eta_-]=0$, we then obtain the decomposition (when acting on $\L^n$ using the injection \eqref{eq:injections})
\begin{align}\label{decompositionSU}
(\S\U_-)^m = \sum_{a=0}^m \binom{m}{a} \eta_+^a \eta_-^{m-a}.
\end{align}
Using the horocyclic operators we define a notion of band structure for distributional sections of the line bundles $\L^n$:
\begin{definition}
For $\lambda \in \C$ and $n\in \Z$, say that $f \in \D'(\M;\L^n) \cap \ker({\bf X}+\lambda)$ is a Pollicott-Ruelle resonant state \emph{in the $m$-th band} if  $m\in\N_0$ is the smallest integer such that $(\S\Um)^mf\in\ker(\S\Um)$. 
\end{definition}  
We immediately get an equivalent characterisation due to \eqref{decompositionSU}, taking into account that $\eta_+$ and  $\eta_-$ map into different embedded line bundles, so their sum applied to a section vanishes iff the individual summands vanish:
\begin{lemma}\label{band_iff_etapm}
A Pollicott-Ruelle resonant state $f \in \D'(\M;\L^n) \cap \ker({\bf X}+\lambda)$ is in the $m$-th band iff $m$ is the smallest integer such that $\eta_+^a\eta_-^b f=0$ whenever $a+b=m+1$.
\end{lemma}
\begin{definition}We define the space
\[
\mathrm{Res}^m_{n}(\lambda):=\{u \in \mathrm{Res}_{n}(\lambda)\,|\, u\textup{ is in the }m\textup{-th band}\}
\]
of $m$-th band resonant states on the bundle $\L^n$, and we call $\lambda$ an \emph{$m$-th band resonance on $\L^n$} if $\mathrm{Res}^m_{n}(\lambda)\neq \{0\}$. We write
$\sigma^m_{n}$ for the set of $m$-th band resonances on $\L^n$.
\end{definition}

Next, we shall identify the $m$-th band of resonances on $\mc{L}^n$ with the $0$-th band of resonances on $\mc{L}^{n+a-b}$ where $a,b\in\N_0$ are such that $a+b=m$. 

Suppose $u \in \D'(\M;\L^n) \cap \ker({\bf X}+\lambda)$ is in the $m$-th band. Using the decomposition of $(\S\U_-)^m$ we obtain
\[
(\S\U_-)^m u = \sum_{a=0}^m \binom{m}{a} \eta_+^a \eta_-^{m-a} u,
\]
and we have for every combination $a+b=m$  using Lemma \ref{band_iff_etapm} and \eqref{eq:etamucomm}
\[
\eta_+^a\eta_-^b u \in \D'(\M;\L^{n+a-b})\cap\ker({\bf X}+\lambda+m) \cap \ker (\S\U_-).
\]
Taking into account that for $f \in \D'(\M;\L^n)  \cap \ker(\S\U_-)\cap \ker({\bf X}+z)$ for some $z\in \C$, the wavefront set condition $\mathrm{WF}(f)\subset E^\ast_u$ is automatically fulfilled by microlocal ellipticity (see \cite[Lemma 2.5]{kuester-weich18}), the operator $\eta_+^a\eta_-^b$ induces for each $\lambda \in \C$,  $n\in \Z$, $a,b\in \N_0$ a map
\bq
J_{n,\lambda}^{a,b}:\mathrm{Res}^{a+b}_{n}(\lambda)\to \mathrm{Res}^0_{n+a-b}(\lambda+a+b).\label{eq:Jnab}
\eq

We finish this section with some preparations that will later allow us to appeal to the Poisson transform bijectivity results in \cite{dfg}. Indeed, in order  to apply them, we would like for a distributional section $f$ in the $m$-th band to have that $m$ is minimal such that $\S\Um^m f\in\ker(\Um)$. This is true as a simple consequence of the following basic observation:

\begin{lemma}\label{xyab0}
Let $f\in \mc{D}'(G)$ and suppose, for all $a,b\in\N_0$ with $a+b=m+1$, that $\eta_+^a\eta_-^bf=0$. Then $(U_1^-)^a(U_2^-)^bf=0$ for all $a,b$ with $a+b=m+1$.
\end{lemma}
\begin{proof}
It suffices to write for $a+b=m+1$
\[ (U_1^-)^a(U_2^-)^bf= (-i)^b(\eta_++\eta_-)^a(\eta_+-\eta_-)^bf=\sum_{j+k=m+1}C_{jk}\eta_+^j\eta_-^kf=0\]
for some constants $C_{jk}\in \C$.
\end{proof}

\begin{lemma}\label{band_iff_SU}
An $f \in \D'(\M;\L^n) \cap \ker({\bf X}+\lambda)$ is in the $m$-th band iff $m$ is the smallest integer such that $(\S\Um)^m f\in\ker(\Um)$.
\end{lemma}
\begin{proof}
The only if direction is immediate. If $f$ is in the $m$-th band then by Lemma~\ref{band_iff_etapm} $\eta_+^a\eta_-^bf=0$ whenever $a+b=m+1$. We lift $f$ to the cover $G/M=S\H^3$ (and still denote it by $f$) and make our analysis on $G/M$. 
As a consequence of Lemma~\ref{xyab0} we conclude $(U_1^-)^a(U_2^-)^bf=0$ whenever $a+b=m+1$. Now, by \eqref{eq:Uminuscoord} one has 
\bqn
(\S\U_-)^mf(g M) = [g, \sum_{s_1,\ldots,s_m\in \{+,-\}} \sum_{K\in\sA^N} (\eta_{s_m}\cdots\eta_{s_1} \lambda_K)\, \mc{S}(v_{s_1}^\ast\otimes\cdots \otimes v_{s_m}^\ast\otimes v_{k_1}^\ast\otimes\cdots \otimes v_{k_N}^\ast) ]
\eqn
if \[
f(gM) = [ g, \sum_{K\in\sA^N} \lambda_K\, v_{k_1}^\ast\otimes\cdots \otimes v_{k_N}^\ast ], \quad \lambda_K\in \Cinft(G),\quad g\in G,
\]
where $\sA^N=\{(k_1,\ldots,k_N);\; k_j\in \{+,-\}\}$ and $\{v_+^\ast, v_-^\ast\}$ is the dual basis to the basis $\{v_+,v_-\}$ introduced in Section \ref{sec:caninj}. Inserting  the definition $\eta_{\pm}=\frac{1}{2}(U_1^-\pm iU_2^-)$, we obtain
\begin{multline*}
(\S\U_-)^mf(g M) = [g, \sum_{l=0}^m\sum_{s_1,\ldots,s_m\in \{+,-\}} \sum_{K\in\sA^N} c_{s_1,\ldots,s_m,K,l}((U^-_1)^l(U^-_2)^{m-l} \lambda_K)\\
\mc{S}(v_{s_1}^\ast\otimes\cdots \otimes v_{s_m}^\ast\otimes v_{k_1}^\ast\otimes\cdots \otimes v_{k_N}^\ast) ]
\end{multline*}
for some constants $c_{s_1,\ldots,s_m,K,l}\in \C$. By \eqref{eq:XU_-} we then get
\begin{multline*}
\U_-(\S\U_-)^mf(g M) = [g, \sum_{l=0}^m\sum_{s,s_1,\ldots,s_m,\in \{+,-\}} \sum_{K\in\sA^N} \frac{1}{2}c_{s_1,\ldots,s_m,K,l}((U_1^-s\, iU_2^-)(U^-_1)^l(U^-_2)^{m-l} \lambda_K)\\
v_s\otimes\mc{S}(v_{s_1}^\ast\otimes\cdots \otimes v_{s_m}^\ast\otimes v_{k_1}^\ast\otimes\cdots \otimes v_{k_N}^\ast) ].
\end{multline*}
On the other hand, we have for all $a,b$ such that $a+b=m+1$
\[
0=(U_1^-)^a(U_2^-)^bf(g M)=[ g, \sum_{K\in\sA^N} ((U_1^-)^a(U_2^-)^b\lambda_K)\, v_{k_1}^\ast\otimes\cdots \otimes v_{k_N}^\ast ],\qquad \forall\; g\in G,
\]
which implies $(U_1^-)^a(U_2^-)^b\lambda_K=0$. Since $U^-_1$ and $U^-_2$ commute, we can  conclude that $(U_1^-\pm iU_2^-)(U^-_1)^l(U^-_2)^{m-l} \lambda_K=0$ for all $l\in \{0,\ldots,m\}$ and hence $\U_-(\S\U_-)^mf=0$.
\end{proof}

\subsubsection{Horocyclic inversion in the ladder picture}

In what follows we aim to invert the horocyclic operators in the language of the ladder operators $\eta_\pm$ on the line bundles $\L^n$. More precisely,  in view of Lemma \ref{band_iff_etapm}, suppose $a+b=m$ and $f\in \D'(S\H^3;\L^{n+a-b})$ is a distributional section in the kernels of $\eta_+,\eta_-$, and ${\bf X}+\lambda+m$. The following calculations provide $f'\in\D'(S\H^3;\L^n)$ such that $\eta_+^a\eta_-^b f'=f$ and $({\bf X}+\lambda)f'=0$. Moreover, we will see that $\mathrm{WF}(f')\subset \mathrm{WF}(f)$. 
\begin{lemma}\label{lem:decomp1}
For each $m\in\nn^*$ the Lie algebra elements $\eta_\pm, \mu_\pm, Q_\pm$ from \eqref{eq:etamu} fulfill
\[
\eta_+^m\mu_-^m = \eta_+^{m-1}\mu_-^m\eta_+ + \eta_+^{m-1}\mu_-^{m-1}(mQ_+- m(m-1)),
\]
\[
\eta_-^m\mu_+^m = \eta_-^{m-1}\mu_+^m\eta_- + \eta_-^{m-1}\mu_+^{m-1}(mQ_-- m(m-1)).
\]
\end{lemma}
\begin{proof}
First we note that $Q_+\mu_-^k=\mu_-^k(Q_+-2k)$. Interchanging the innermost $\eta_+,\mu_-$ gives with the commutation relations \eqref{eq:etamucomm}
\begin{align*}
\eta_+^m\mu_-^m 	& = \eta_+^{m-1}(\mu_-\eta_++Q_+)\mu_-^{m-1} \\
				& = \eta_+^{m-1}\mu_-\eta_+\mu_-^{m-1} + \eta_+^{m-1}\mu_-^{m-1}(Q_+-2(m-1)).
\end{align*}
Continuing to shift the operator $\eta_+$ to the right ultimately gives
\[
\eta_+^m\mu_-^m	= 	\eta_+^{m-1}\mu_-^m\eta_+ + \sum_{k=1}^m \eta_+^{m-1}\mu_-^{m-1}(Q_+-2(m-k)),
\]
which provides the result as $\sum_{k=1}^m 2(m-k)=m(m-1)$. The second equation is handled similarly.
\end{proof}
Let us now regard $\eta_\pm, \mu_\pm$, and $Q_\pm$ as differential operators mapping (distributional) sections of $\L^n$ to (distributional) sections of $\L^{n\pm1}$ and $\L^{n}$, respectively. An immediate consequence of applying Lemma \ref{lem:decomp1} recursively is then
\begin{lemma}\label{lem:commutations_eta_mu_give_Q}
For each $m\in \N_0$ and $k\in \{0,1\ldots,m\}$ there are differential operators $\mathcal{B}^{m,k}_\pm$ and polynomials $P_{k,m}$ such that with $\mA_\pm:=\sum_{k=0}^m\mathcal{B}^{m,k}_\pm \eta_\pm P_{k,m}(Q_\pm)$ one has
\begin{align*}
\eta_+^m\mu_-^m &= \mA_+ + \prod_{k=1}^m (kQ_+- k(k-1)),
&
\eta_-^m\mu_+^m &= \mA_- + \prod_{k=1}^m (k Q_-- k(k-1)).
\end{align*}
\end{lemma}

\begin{lemma}\label{inversion_vu}
Let $v\in\D'(S\H^3;\L^{n+a-b})\cap\ker({\bf X}+\lambda+a+b)\cap\ker(\eta_++\eta_-)$ with $a,b\in \N_0$. Define two constants $q_\pm$ by $q_\pm = \lambda+a+b\pm(n+a-b)$ and a further two constants
\[
p_+= \prod_{k=1}^a (kq_+- k(k-1)), \qquad p_-= \prod_{k=1}^b (kq_-- k(k-1)).
\]
(If $a=0$ set $p_+=1$. If $b=0$ set $p_-=1$.) If $p_+p_-\neq0$ define $u\in\D'(S\H^3;\L^n)$ by $u := \frac{1}{p_+p_-}\mu_-^a\mu_+^b v.$ 
Then $\eta_+^a\eta_-^b u=v$ and $({\bf X} + \lambda) u =0$. Moreover, one has $\mathrm{WF}(u)\subset\mathrm{WF}(v)$.
\end{lemma}
\begin{proof}
Introduce the notation $\mathcal{Q}_\pm$ for the products found in Lemma~\ref{lem:commutations_eta_mu_give_Q} (with the products terminating at $a,b$ respectively). Then (since $[\eta_\pm,\mu_\pm]=0$)
\[
\eta_+^a\eta_-^b \mu_-^a\mu_+^b = (\mA_+ + \mathcal{Q}_+)(\mA_- + \mathcal{Q}_-).
\]
Note that $v\in\ker(Q_\pm-q_\pm)$ and also $v\in\ker(\mA_\pm)$. The first result follows:
\[
\eta_+^a\eta_-^b u = (p_+p_-)^{-1}(\mA_+ + \mathcal{Q}_+)(\mA_- + \mathcal{Q}_-)v = (p_+p_-)^{-1}\mathcal{Q}_+\mathcal{Q}_-v= v. 
\]
The relation $({\bf X} + \lambda) u =0$ is a consequence of the commutation relations $[X,\mu_\pm]=\mu_\pm$. Finally, the statement about the wavefront sets follows from the observation that $u$ is obtained from $v$ by applying a differential operator and differential operators do not enlarge wavefront sets.
\end{proof}

The inversion results above allow us to learn more about the band structure of resonances:
\begin{cor}\label{cor:isoJ2}Let $a,b\in \N_0$. The map 
\bqn
J_{n,\lambda}^{a,b}:\mathrm{Res}^{a+b}_{n}(\lambda)\to \mathrm{Res}^0_{n+a-b}(\lambda+a+b)
\eqn
defined in $\eqref{eq:Jnab}$ is surjective if $\lambda\not\in \A_{a,b}:=(-n+[-2a,-a-1]\cap \Z)\cup (n+[-2b,-b-1]\cap \Z)$. 
Moreover, if for some  $m\in \N_0$ one has  $\lambda\not\in\A_{m}:= \bigcup_{a+b=m}\A_{a,b}$, the map
\[
J^m_{n,\lambda}:=\sum_{a+b=m}J_{n,\lambda}^{a,b}:\mathrm{Res}^{m}_{n}(\lambda)\to \bigoplus_{a+b=m} \mathrm{Res}^0_{n+a-b}(\lambda+m)  
\]
is an isomorphism.
\end{cor}
\begin{proof}
To get surjectivity of $J_{n,\lambda}^{a,b}$, we lift the resonant state $v\in \mathrm{Res}^0_{n+a-b}(\lambda+a+b)$ to $G/M$ and apply Lemma  \ref{inversion_vu}: the distribution $u$ such that $\eta_+^a\eta_-^bu=v$ descends to $\M=\Gamma\backslash G/M$ and is in $\mathrm{Res}^{a+b}_{n}(\lambda)$ thanks to the wavefront condition. It then suffices to observe that the condition $p_+p_-\neq 0$ in Lemma \ref{inversion_vu} is equivalent to $\lambda\not\in (-n+[-2a,-a+1]\cap \Z)\cup (-n+[-2b,-b+1]\cap \Z)$. Knowing that each $J_{n,\lambda}^{a,b}$ is surjective, the map $J^m_{n,\lambda}$ is  surjective. However, it is also injective by definition of the notion of $m$-th band.
\end{proof}
This result shows that $\la\in \C\setminus -\nn^*$ is a resonance in the $m$-th band 
if and only if $\la+m$ is a resonance in the $0$-th band, and thus to study the full resonance set (except possibly at $-\nn^*$), it suffices to understand the $0$-th band of resonance for the action of ${\bf X}$ on each of the bundles $\mc{L}^{n}$. See Figure \ref{fig:lines} on page \pageref{fig:lines} for an illustration.

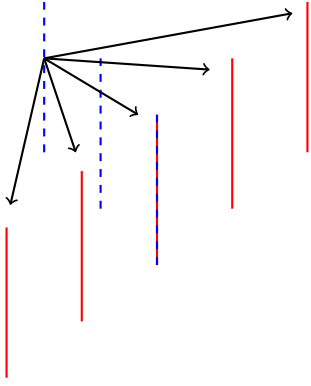
\begin{figure}[h!]
\begin{tikzpicture}
\draw[red, solid, thick] (-0.5,-3.5)--(-0.5,-1.5);
\draw[red, solid, thick] (0.5,-2.75)--(0.5,-0.75);
\draw[red, solid, thick] (1.5,-2)--(1.5,0);
\draw[red, solid, thick] (2.5,-1.25)--(2.5,0.75);
\draw[red, solid, thick] (3.5,-0.5)--(3.5,1.5);
\draw[blue, dashed, thick] (0,-.5)--(0,1.5);
\draw[blue, dashed, thick] (0.75,-1.25)--(0.75,0.75);
\draw[blue, dashed, thick] (1.5,-2)--(1.5,0);
\draw[->, thick](0,0.75)--(-0.45,-1.2);
\draw[->, thick](0,0.75)--(0.42,-0.5);
\draw[->, thick](0,0.75)--(1.25,0);
\draw[->, thick](0,0.75)--(2.2,0.6);
\draw[->, thick](0,0.75)--(3.3,1.35);
\end{tikzpicture}
\caption{\label{fig:lines}Schematic picture of the band structure. The dashed blue lines represent the bands $\sigma^0_{0}$, $\sigma^1_{0}$, $\sigma^2_{0}$. The red solid lines represent the bands $\sigma^0_{-2}$, $\sigma^0_{-1}$, $\sigma^0_{0}$, $\sigma^0_{1}$, $\sigma^0_{2}$. The arrows illustrate that the operators $J^3_{0,\lambda}$ from Corollary \ref{cor:isoJ2} relate $\sigma^2_{0}$ and $\bigcup_{a+b=2}\sigma^0_{a-b}$. }
\end{figure}

\subsection{Zeroth band resonant states and Laplacian eigensections}

Recall from \eqref{eq:injections} and \eqref{eq:tracefree} that for each $n\in \N_0$ the bundles $\L^{\pm n}$ inject $G$-equivariantly into the tensor bundle $\otimes_{S,0}^{n}\E$, so that for each $\lambda\in \C$ the spaces $\mathrm{Res}_{\pm n}(\lambda)$, $\mathrm{Res}^0_{\pm n}(\lambda)$ can be regarded as subspaces of $\D'(\M;\otimes_{S,0}^{n}\E)$, where $\M=S{\bf M}$.  Moreover, if $n\geq 1$, then by  \eqref{eq:tracefree} the direct sum $\L^{n}\oplus\L^{-n}$ is isomorphic to $\otimes_{S,0}^{n}\E$. 
We may then appeal to \cite[Thm.\ 6]{dfg} to obtain
\begin{prop}\label{PR_then_laplacian}
For each $n\in \N_0$, the pushforward $$\pi_\ast:\D'(S{\bf M};\otimes_S^{n}\E)\to \D'({\bf M};\otimes^{n}_ST^\ast {\bf M})$$ induced by fiber-wise integration in the sphere bundle $\pi:S{\bf M}\to {\bf M}$ restricts to
\bq
\pi_\ast: \mathrm{Res}^0_{n}(\lambda)\oplus \mathrm{Res}^0_{-n}(\lambda)\longrightarrow  \ker(\Delta_{n}-\mu_{n}(\lambda))\cap\ker \nabla^*\subset \Cinft({\bf M};\otimes^{{n}}_{S,0}T^*{\bf M})\label{eq:pistariso}
\eq
if $n> 0$, and if $n=0$, then it restricts to
\bq
\pi_\ast: \mathrm{Res}^0_{0}(\lambda)\longrightarrow  \ker(\Delta_{0}-\mu_{0}(\lambda))\subset \Cinft({\bf M}),\label{eq:pistariso0}
\eq
where $\Delta_n$ is the Bochner Laplacian acting on symmetric trace-free $n$-tensors on ${\bf M}$, $\nabla^*=-\mc{T}\circ \nabla$ is the divergence operator acting on such tensors, and $\mu_n(\lambda)=-\lambda(\lambda+2)+n$.

Moreover, if $\lambda\not\in-1-\frac{1}{2}\N_0$, then the maps  \eqref{eq:pistariso} and \eqref{eq:pistariso0} are bijective.
\end{prop}
In addition, it is known (see \cite[Lemma 6.1]{dfg}) that for $n\geq 1$ the spectrum of $\Delta_n$ is bounded from below by $n+ 1$. Taking into account also Corollary \ref{cor:isoJ2} and writing  ${\rm Sp}_{\ker\nabla^*}(\Delta_n)$ for the spectrum of $\Delta_n$ acting on divergence-free tensors, we get
\begin{cor}\label{cor:individualgap}
For all $n\in \N_0$ and $m\in \N_0$, the set 
$(\sigma_n^m\cup\sigma_{-n}^m)\setminus (-1-\demi\N_0)$ is equal to 
\[
\big\{ -1-m\pm \sqrt{|n+m-2k|+1-\nu} \, |\, \nu\in {\rm Sp}_{\ker\nabla^*}(\Delta_{|n+m-2k|}),\;0\leq k\leq m \big\} \setminus (-1-\demi\N_0).
\] \end{cor}
To write this set in a more meaningful way, we need to distinguish two cases:
\begin{itemize}[leftmargin=*]
\item If $n+m\in 2\Z+1$ or $n+m\in 2\Z$ and $|n|> m$, then $|n+m-2k|\geq 1$ for $k\in \{0,\ldots,m\}$ and we can write the set  as
\[
\big\{ -1-m+ i\nu\, |\, \nu^2\in {\rm Sp}_{\ker\nabla^*}(\Delta_{|n+m-2k|}-|n+m-2k|-1),\;0\leq k\leq m \big\} \setminus (-1-\demi\N_0)
\] 
thus $\sigma_n^m\subset (-1-\demi \N_0) \cup( -m-1+i\R)$ and  
if $\Re\lambda > -1$, then 
$\mathrm{Res}^m_{n}(\lambda)=\{0\}$.

\item If $n+m\in 2\Z$ and $|n|\leq m$, then we can write the set  as
\begin{multline*}
\big(\big\{ -1-m+ i\nu\, |\, \nu^2\in {\rm Sp}_{\ker\nabla^*}(\Delta_{|n+m-2k|}-|n+m-2k|-1),\;0\leq k\leq m,\,2k\neq n+m   \big\}\\
\cup \big\{ -1-m\pm i\sqrt{\nu-1} \, |\, \nu\in {\rm Sp}(\Delta_{0}),\;\nu\geq 1\}\\
\cup \big\{ -1-m\pm \sqrt{1-\nu} \, |\, \nu\in {\rm Sp}(\Delta_{0}), \nu< 1\} \big) \setminus (-1-\demi\N_0)
\end{multline*}
thus $\sigma_n^m\subset (-1-\demi \N_0) \cup( -m-1+i\R)\cup [-m-2,-m]$.
\end{itemize}
In total, we see that 
\bq
\sigma^{\rm PR}_n\cap \{\lambda\in \C\,|\,\Re \lambda>-1\}=\begin{cases}\emptyset,\qquad &n\neq 0,\\ 
\{-1 + \sqrt{1-\nu}\, ;\,\nu\in {\rm Sp}(\Delta_{0}), \nu< 1 \}, & n=0.\end{cases}\label{eq:resonancesunion}
\eq
We remark that the union $\cup_{n\in\Z}\sigma_n^m$ of resonances in the $m$-th band is likely not a discrete subset of $-1-m+i\R$. To effectively prove this, it would suffice to prove that in a fixed interval
$[0,T]$, the number of eigenvalues in $[0,T]$ of $\Delta_{k}-k-1$ acting on divergence-free tensors tends to infinity as $k\to \infty$. This could probably be shown by semiclassical methods. This strongly suggests that one can \emph{not}  expect a meromorphic extension of the resolvent of the frame flow to the whole complex plane.

\bibliography{bibliography}
\bibliographystyle{amsalpha}
\end{document}

%% file: commands.tex
\def\C{\mathbb{C}} % complex numbers
\def\R{\mathbb{R}} % real numbers
\def\Z{\mathbb{Z}} % integers
\def\N{\mathbb{N}} % natural numbers

\def\H{\mathbb{H}} % hyperbolic space

\def\SOdeux{\textrm{SO}(2)}

\def\g{\mathfrak{g}}

\def\End{\mbox{End}}

\def\E{\mathcal{E}}
\def\L{\mathcal{L}}

\def\U{\mathcal{U}}
\def\Um{\mathcal{U}_-}

\def\demi{\tfrac{1}{2}}

\def\mA{\mathcal{A}}
\usepackage{mathrsfs}
\def\sA{\mathscr{A}}

\def\S{\mathcal{S}}

\def\T{\mathcal{T}}

%% file: frame_flow_revised-arxiv.bbl
\providecommand{\bysame}{\leavevmode\hbox to3em{\hrulefill}\thinspace}
\providecommand{\MR}{\relax\ifhmode\unskip\space\fi MR }
\providecommand{\MRhref}[2]{%
  \href{http://www.ams.org/mathscinet-getitem?mr=#1}{#2}
}
\providecommand{\href}[2]{#2}
\begin{thebibliography}{{Dya}16}

\bibitem[{Arn}12]{arnoldi12}
J.-F. {Arnoldi}, \emph{{Fractal Weyl law for skew extensions of expanding
  maps.}}, {Nonlinearity} \textbf{25} (2012), no.~6, 1671--1693.

\bibitem[BDL18]{BDL18}
V.~Baladi, M.~Demers, and C.~Liverani, \emph{{Exponential decay of correlations
  for finite horizon Sinai billiard flows}}, {Invent. Math.} \textbf{211}
  (2018), 39--177.

\bibitem[BG80]{Brin-Gromov80}
M.~Brin and M.~Gromov, \emph{{On the ergodicity of frame flows}}, {Invent.
  Math.} \textbf{60} (1980), 1--7.

\bibitem[BK84]{Brin-Karcher84}
M.~Brin and H.~Karcher, \emph{{Frame flows on manifolds with pinched negative
  curvature}}, Compositio Math. \textbf{52} (1984), 275--297.

\bibitem[BL07]{BL07}
O.~Butterley and C.~Liverani, \emph{Smooth {A}nosov flows: correlation spectra
  and stability}, J. Mod. Dyn \textbf{1} (2007), no.~2, 301--322.

\bibitem[BP73]{Brin-Pesin}
M.~Brin and Y.~Pesin, \emph{{Flows of frames on manifolds of negative
  curvature}}, Uspekhi Mat.Nauk \textbf{28} (1973), 169--170.

\bibitem[BP03]{Burns-Pollicott03}
K.~Burns and M.~Pollicott, \emph{{Stable Ergodicity and Frame Flows}}, {Geom.
  Dedicata} \textbf{98} (2003), 189--210.

\bibitem[Bri95]{Brin}
M.~Brin, \emph{Ergodicity of the geodesic flow}, Lectures on spaces of
  nonpositive curvature ({Appendix}), {Oberwolfach Seminars}, vol.~25,
  Birkh\"auser Basel, 1995.

\bibitem[BT07]{BT07}
V.~{Baladi} and M.~{Tsujii}, \emph{{Anisotropic H\"older and Sobolev spaces for
  hyperbolic diffeomorphisms}}, {Ann. Inst. Fourier} \textbf{57} (2007), no.~1,
  127--154.

\bibitem[Cha00]{charles2000}
L.~Charles, \emph{{Aspects semi-classiques de la quantification
  g\'eom\'etrique}}, HAL Id: tel-00001289, available at
  \texttt{tel.archives-ouvertes.fr/tel-00001289}, 2000.

\bibitem[CT79]{cahntaylor}
R.~{Cahn} and M.~Taylor, \emph{{Asymptotic behavior of multiplicities of
  representations of compact groups.}}, {Pac. J. Math.} \textbf{84} (1979),
  17--28.

\bibitem[DFG15]{dfg}
S.~Dyatlov, F.~Faure, and C.~Guillarmou, \emph{{Power spectrum of the geodesic
  flow on hyperbolic manifolds}}, Analysis \& PDE \textbf{8} (2015), no.~4,
  923--1000.

\bibitem[DG16]{DG16}
S.~Dyatlov and C.~Guillarmou, \emph{{Pollicott--Ruelle resonances for open
  systems}}, Annales Henri Poincar{\'e} \textbf{17} (2016), no.~11, 3089--3146.

\bibitem[Dol98]{Do98}
D.~Dolgopyat, \emph{{On decay of correlations in Anosov flows}}, {Annals of
  Mathematics} \textbf{147} (1998), no.~2, 357--390.

\bibitem[{Dya}16]{Dyatlov16}
S.~{Dyatlov}, \emph{{Spectral gaps for normally hyperbolic trapping. (Trous
  spectraux pour des ensembles capt\'es normalement hyperboliques.)}}, {Ann.
  Inst. Fourier} \textbf{66} (2016), no.~1, 55--82.

\bibitem[DZ16]{DZ16}
S.~Dyatlov and M.~Zworski, \emph{Dynamical zeta functions for {A}nosov flows
  via microlocal analysis}, Ann. Sci. Ec. Norm. Sup{\'e}r \textbf{49} (2016),
  543--577.

\bibitem[DZ19]{dyatlovzworskibook}
\bysame, \emph{{Mathematical theory of scattering resonances}}, {Graduate
  Studies in Mathematics}, vol. 200, Providence, RI: American Mathematical
  Society (AMS), 2019.

\bibitem[EO]{Edwards-Oh}
S.~Edwards and H.~Oh, \emph{Spectral gap and exponential mixing on
  geometrically finite hyperbolic manifolds}, arXiv 2001.03377.

\bibitem[Fau11]{faure2011}
F.~Faure, \emph{{Semiclassical origin of the spectral gap for transfer
  operators of a partially expanding map.}}, {Nonlinearity} \textbf{24} (2011),
  no.~5, 1473--1498.

\bibitem[FRS08]{FRS08}
F.~Faure, N.~Roy, and J.~Sj\"ostrand, \emph{{Semi-classical approach for Anosov
  diffeomorphisms and Ruelle resonances.}}, {Open Math. J.} \textbf{1} (2008),
  35--81.

\bibitem[FS90]{Flaminio-Spatzier90}
L.~Flaminio and R.~Spatzier, \emph{{Geometrically finite groups,
  Patterson-Sullivan measures and Ratner's theorem}}, {Invent. Math.}
  \textbf{99} (1990), 601--626.

\bibitem[FS11]{FS11}
F.~Faure and J.~Sj{\"o}strand, \emph{Upper bound on the density of {R}uelle
  resonances for {A}nosov flows}, Communications in mathematical physics
  \textbf{308} (2011), no.~2, 325--364.

\bibitem[FT15]{Faure-TsujiiAST}
F.~Faure and M.~Tsujii, \emph{Prequantum transfer operator for anosov
  diffeomorphism}, Ast\'erisque \textbf{375} (2015).

\bibitem[FT17]{FT17a}
\bysame, \emph{The semiclassical zeta function for geodesic flows on negatively
  curved manifolds}, Inventiones mathematicae \textbf{208} (2017), no.~3,
  851--998.

\bibitem[GBW17]{BW17}
Y.~Guedes~Bonthonneau and T.~Weich, \emph{Ruelle resonances for manifolds with
  hyperbolic cusps}, Journal of the European Math. Soc., to appear (preprint
  arXiv 2017).

\bibitem[GL06]{GL06}
S.~{Gou\"ezel} and C.~{Liverani}, \emph{{Banach spaces adapted to Anosov
  systems}}, {Ergodic Theory Dyn. Syst.} \textbf{26} (2006), no.~1, 189--217.

\bibitem[GS13]{guilleminsternbergnotes}
V.~Guillemin and S.~Sternberg, \emph{{Semi-classical analysis}}, Lecture notes,
  available at \texttt{math.mit.edu/$\sim$vwg/semistart.pdf}, 2013.

\bibitem[GU86]{guilleminuribe86}
V.~{Guillemin} and A.~{Uribe}, \emph{{The trace formula for vector bundles.}},
  {Bull. Am. Math. Soc., New Ser.} \textbf{15} (1986), 222--224.

\bibitem[GU89]{guilleminuribe89}
\bysame, \emph{{Circular symmetry and the trace formula.}}, {Invent. Math.}
  \textbf{96} (1989), no.~2, 385--423.

\bibitem[GU90]{guilleminuribe90}
\bysame, \emph{{Reduction and the trace formula.}}, {J. Differ. Geom.}
  \textbf{32} (1990), no.~2, 315--347.

\bibitem[HM79]{Howe-Moore79}
R.E. Howe and C.C. Moore, \emph{Asymptotic properties of unitary
  representations}, {Journal of Functional Analysis} \textbf{32} (1979), no.~1,
  72--96.

\bibitem[HPS83]{hogreve83}
H.~{Hogreve}, J.~{Potthoff}, and R.~{Schrader}, \emph{{Classical limits for
  quantum particles in external Yang-Mills potentials.}}, {Commun. Math. Phys.}
  \textbf{91} (1983), 573--598.

\bibitem[KW19]{kuester-weich18}
B.~K\"uster and T.~Weich, \emph{{Quantum-Classical Correspondence on Associated
  Vector Bundles Over Locally Symmetric Spaces}}, International Mathematics
  Research Notices (2019), rnz068.

\bibitem[Liv04]{Li04}
C.~Liverani, \emph{On contact {A}nosov flows}, Annals of Mathematics (2004),
  1275--1312.

\bibitem[LM87]{marlesymp}
P.\ {Libermann} and C.-M.\ {Marle}, \emph{{Symplectic geometry and analytical
  mechanics. Transl. from the French by Bertram Eugene Schwarzbach}},
  {Mathematics and its Applications, 35. Dordrecht etc.: D. Reidel Publishing
  Company, a member of the Kluwer Academic Publishers Group}, 1987.

\bibitem[MO15]{MoOh2015}
A.~Mohammadi and H.~Oh, \emph{{Counting and Primes for orbits of geometrically
  finite groups}}, Journal of the European Math. Soc. \textbf{17} (2015),
  no.~837--897.

\bibitem[Moo87]{Moore87}
C.~Moore, \emph{Exponential decay of correlation coefficients for geodesic
  flows. in: Moore c.c. (eds) group representations, ergodic theory, operator
  algebras, and mathematical physics.}, Mathematical Sciences Research
  Institute Publications \textbf{6} (1987), 163--181.

\bibitem[Nau05]{Na05}
F.~Naud, \emph{{Expanding maps on Cantor sets and analytic continuation of zeta
  functions}}, {Annales de l'Ecole Normale Sup\'erieure} \textbf{38} (2005),
  no.~1, 116--153.

\bibitem[Nau19]{Naud2019}
\bysame, \emph{On the rate of mixing of circle extensions of anosov maps.},
  Journal of Spectral Theory \textbf{9} (2019), no.~3, 791--824.

\bibitem[NZ15]{NZ15}
S.~Nonnenmacher and M.~Zworski, \emph{{Decay of correlations for normally
  hyperbolic trapping}}, {Invent. Math.} \textbf{200} (2015), 345--438.

\bibitem[Pol92]{Pollicott92}
M.~Pollicott, \emph{Exponential mixing for the geodesic flow on hyperbolic
  three-manifolds.}, Journal of statistical physics \textbf{67} (1992),
  no.~3/4, 667--673.

\bibitem[RS80]{Reed-Simon1}
M.~Reed and B.~Simon, \emph{Methods of modern mathematical physics. {I}
  {Functional Analysis}}, Academic Press, 1980.

\bibitem[ST84]{schradertaylor84}
R.~Schrader and M.~Taylor, \emph{{Small h asymptotics for quantum partition
  functions associated to particles in external Yang-Mills potentials.}},
  {Commun. Math. Phys.} \textbf{92} (1984), 555--594.

\bibitem[ST89]{schradertaylor89}
\bysame, \emph{{Semiclassical asymptotics, gauge fields, and quantum chaos.}},
  {J. Funct. Anal.} \textbf{83} (1989), no.~2, 258--316.

\bibitem[Sto11]{St11}
L.~Stoyanov, \emph{{Spectra of Ruelle transfer operators for Axiom A flows}},
  {Nonlinearity} \textbf{24} (2011), no.~4, 1089--1120.

\bibitem[SW]{SarkarWinter}
P.~Sarkar and D.~Winter, \emph{Exponential mixing of frame flows for convex
  cocompact hyperbolic manifolds}, arXiv 2004.14551.

\bibitem[Tsu12]{Ts12}
M.~Tsujii, \emph{{Contact Anosov flows and the Fourier--Bros--Iagolnitzer
  transform}}, {Ergodic Theory Dyn. Syst.} \textbf{32} (2012), no.~06,
  2083--2118.

\bibitem[Tsu18]{Tsujii18}
\bysame, \emph{{Exponential mixing for generic volume-preserving Anosov flows
  in dimension three}}, {J. Math. Soc. Japan} \textbf{70} (2018), no.~2,
  757--821.

\bibitem[TZ]{tsujiizhang20}
M.~Tsujii and Z.~Zhang, \emph{{Smooth mixing Anosov flows in dimension three
  are exponential mixing}}, arXiv 2006.04293.

\bibitem[Win15]{Winter}
D.~Winter, \emph{Mixing of frame flow for rank one locally symmetric spaces and
  measure classification}, Israel. J. Math. \textbf{210} (2015), 467--507.

\bibitem[{Zel}92]{zelditch19921}
S.~{Zelditch}, \emph{{On a ``Quantum Chaos'' theorem of R. Schrader and M.
  Taylor.}}, {J. Funct. Anal.} \textbf{109} (1992), no.~1, 1--21.

\bibitem[Zwo12]{zworski}
M.~Zworski, \emph{Semiclassical analysis}, Graduate Studies in Mathematics,
  vol. 138, American Mathematical Society, Providence, 2012.

\end{thebibliography}
